\documentclass{amsart}
\RequirePackage{fix-cm} 
\usepackage{etex}
\usepackage{fixltx2e}
\usepackage{tikz}
\usepackage{tikz-cd}
\usepackage{adjustbox}
\usepackage{enumitem}
\usepackage{dsfont}

\def\Mg{M^{F\text{-}\mathrm{gal}}}
\def\Mgv{M^{F_v\text{-}\mathrm{gal}}}

\def\varrhobar{\overline{\varrho}}
\def\gal{\mathrm{gal}}
\def\vvv{\widetilde{v}}
\def\sss{\widetilde{s}}
\def\www{\widetilde{w}}
\def\ww{x}
\def\hh{h}
\def\Out{\mathrm{Out}}
\def\Pr{\mathbf{P}{r}}
\def\Ad{\mathrm{Ad}}
\def\kbar{\overline{k}}
\def\PG{\mathbb{P}G}
\newcommand{\tr}{\operatorname{tr}}

\renewcommand{\mathbb}{\mathbf}
\renewcommand{\ell}{l}

\newcommand{\psibar}{\overline{\psi}}

\newcommand{\Favoid}{F^{(\mathrm{avoid})}} 
\DeclareMathOperator{\SL}{SL}
\def\sl{\mathfrak{sl}}

\def\sp{\mathfrak{sp}}

\def\ad{\mathrm{ad}}

\def\h{\mathfrak{h}}
\def\fa{\mathfrak{a}}
\def\fb{\mathfrak{b}}
\def\fn{\mathfrak{n}}
\def\Lg{L^{\mathrm{gal}}}

\def\sp{\mathfrak{sp}}
\newcommand{\ttt}{\mathfrak{t}}
\newcommand\Gder{G^{\mathrm{der}}}

\newcommand\Adero{A^{\mathrm{der},\circ}}
\newcommand\Bdero{B^{\mathrm{der},\circ}}

\newcommand\Hdero{H^{\mathrm{der},\circ}}

\newcommand\Gdero{G^{\mathrm{der},\circ}}

\newcommand{\Qlbartimes}{\overline{\Q}_l^\times}

\newcommand{\tG}{\widetilde{G}}

\DeclareMathOperator{\Sp}{Sp}
\DeclareMathOperator{\PGSp}{PGSp}

\newcommand{\mubar}{\overline{\mu}}

\def\hotimes {\widehat{\otimes}}

\def\CC{\mathscr{C}}
\def\CSS{\mathscr{S}}
\def\chop{\psi}
\def\sl{\mathfrak{sl}}

\def\sp{\mathfrak{sp}}

\def\vv{\mathbf{v}}

\def\fg{\mathfrak{g}}
\def\fa{\mathfrak{a}}
\def\fb{\mathfrak{b}}

\def\fd{\mathfrak{d}}

\def\der{\mathrm{der}}
\def\OL{\mathcal{O}}
\def\g{\mathfrak{g}}

 \usepackage{mathrsfs}
\usepackage{amssymb,amsmath,amsfonts,amsthm,epsfig,amscd}
\usepackage{stmaryrd}
\usepackage[all,cmtip,poly]{xy}
\usepackage{color}

\newcommand{\univ}{{\operatorname{univ}}}
\newcommand{\To}{\longrightarrow}
\newcommand{\isoto}{\stackrel{\sim}{\To}}

\newtheorem{theorem}[subsubsection]{Theorem}
\newtheorem{thm}[subsubsection]{Theorem}
\newtheorem{lemma}[subsubsection]{Lemma}
 \newtheorem{ithm}{Theorem}

\newtheorem{lem}[subsubsection]{Lemma}
\newtheorem{cor}[subsubsection]{Corollary}

\newtheorem{prop}[subsubsection]{Proposition}

\theoremstyle{definition}
\newtheorem{df}[subsubsection]{Definition}
\newtheorem{defn}[subsubsection]{Definition}

\theoremstyle{remark}
\newtheorem{remark}[subsubsection]{Remark}
\newtheorem{rem}[subsubsection]{Remark}

\newtheorem{convention}[subsubsection]{Convention}

\def\numequation{\addtocounter{subsubsection}{1}\begin{equation}}
\def\nummultline{\addtocounter{subsubsection}{1}\begin{multline}}
\def\anumequation{\addtocounter{subsection}{1}\begin{equation}}
\def\anummultline{\addtocounter{subsection}{1}\begin{multline}}

\newif\iffinalrun
\iffinalrun
\else
\fi

\iffinalrun
  \newcommand{\need}[1]{}
  \newcommand{\mar}[1]{}
\else
  \newcommand{\need}[1]{{\tiny *** #1}}
  \newcommand{\mar}[1]{\marginpar{\raggedright\tiny WibblyWobbly  #1}}
\fi

\newcommand{\CM}{\cM}
\newcommand{\m}{\frakm}

\def\A{\mathbf{A}}
\def\C{\mathbf{C}}
\def\F{\mathbf{F}}
\def\Q{\mathbf{Q}}
\newcommand{\Z}{{\mathbb Z}}
\newcommand{\cA}{{\mathcal A}}
\newcommand{\cG}{{\mathcal G}}
\newcommand{\cM}{{\mathcal M}}
\newcommand{\cO}{{\mathcal O}}
\newcommand{\CR}{{\mathcal R}}
\newcommand{\frakm}{\mathfrak{m}}
\newcommand{\Fbar}{\overline{\F}}
\newcommand{\Qbar}{\overline{\Q}}
\newcommand{\Zbar}{\overline{\Z}}
\newcommand{\Fp}{\F_p}

\newcommand{\Fpbar}{\Fbar_p}

\newcommand{\Zp}{\Z_p}

\newcommand{\Zpbartimes}{\Zbar_p^\times}

\newcommand{\Ql}{\Q_{\ell}}
\newcommand{\Qp}{\Q_p}

\newcommand{\Qlbar}{\Qbar_{\ell}}
\newcommand{\Qpbar}{\Qbar_p}

\DeclareMathOperator{\Gal}{Gal}
\DeclareMathOperator{\GL}{GL}

\DeclareMathOperator{\GSp}{GSp}

\DeclareMathOperator{\Hom}{Hom}

\DeclareMathOperator{\im}{im}
\DeclareMathOperator{\Ind}{Ind}

\DeclareMathOperator{\PGL}{PGL}

\DeclareMathOperator{\PSL}{PSL}

\def\sl{\mathfrak{sl}}

\DeclareMathOperator{\Sym}{Sym}

\DeclareMathOperator{\WD}{WD}

\newcommand{\Frob}{\mathrm{Frob}}

\newcommand{\rhobar}{\overline{\rho}}

\newcommand{\into}{\hookrightarrow}
\newcommand{\onto}{\twoheadrightarrow}

\newcommand{\rbar}{\overline{r}}
\newcommand{\sbar}{\overline{s}}
\newcommand{\tbar}{\overline{t}}
\newcommand{\ubar}{\overline{u}}
\newcommand{\abar}{\overline{a}}
\newcommand{\bbar}{\overline{b}}
\newcommand{\cbar}{\overline{c}}
\newcommand{\dbar}{\overline{d}}
\newcommand{\ebar}{\overline{e}}

\newcommand{\Res}{\operatorname{Res}}

\usepackage[bookmarksopen,bookmarksdepth=2]{hyperref}

\begin{document}
\title[Globally realizable components of deformation rings]{Globally realizable components of local
  deformation rings}

\author[F. Calegari]{Frank Calegari}\email{fcale@math.uchicago.edu}
\address{Department of Mathematics, University of Chicago,
5734 S.\ University Ave., Chicago, IL 60637, USA}

\author[M. Emerton]{Matthew Emerton}\email{emerton@math.uchicago.edu}
\address{Department of Mathematics, University of Chicago,
5734 S.\ University Ave., Chicago, IL 60637, USA}

\author[T. Gee]{Toby Gee} \email{toby.gee@imperial.ac.uk} \address{Department of
  Mathematics, Imperial College London,
  London SW7 2AZ, UK}

\thanks{The first author was supported in part by NSF  Grants
  DMS-1404620 and DMS-1701703.  The second author was supported in part by the
  NSF grants DMS-1303450, DMS-1601871, and DMS-1902307.  The third author was
  supported in part by a Leverhulme Prize, EPSRC grant EP/L025485/1, Marie Curie Career
  Integration Grant 303605, 
  ERC Starting Grant 306326, and a Royal Society Wolfson Research
  Merit Award. }

\maketitle
\begin{abstract}
Let $n$ be either~$2$, or an odd integer greater than~$1$, and fix a
prime~$p>2(n+1)$. Under standard ``adequate image'' assumptions, we
show that the set of components of $n$-dimensional $p$-adic potentially semistable local
Galois deformation rings that are seen by potentially automorphic
compatible systems of  polarizable
 Galois representations over some CM field is
independent of the particular global situation. We also (under the
same assumption on~$n$) improve on the main potential automorphy
result of~\cite{BLGGT}, replacing ``potentially diagonalizable'' by
``potentially globally realizable.''
\end{abstract}
\setcounter{tocdepth}{2}
\tableofcontents
\section{Introduction}\label{sec: introduction}
\subsection{Our results}The results of this paper are motivated by the questions
of level-raising/lowering for automorphic forms, and the weight part of
Serre's conjecture, as well as a question about ``automorphic components'' related to the
Fontaine--Mazur conjecture and to global approaches to the $p$-adic
local Langlands program. However,
our main results are stated purely in terms of Galois representations. 

We need to introduce a few definitions before stating our main theorem.
Let~$F$ be a CM field with maximal totally real subfield~$F^+$, and
write~$G_F$ for the absolute Galois group~$\Gal(\overline{F}/F)$.
(In this paper, all CM fields will be imaginary and so~$F/F^{+}$ is a quadratic extension.)

 We say that a
representation~$\sbar:G_F\to\GL_n(\Fpbar)$ is \emph{reasonable} if it
satisfies the hypotheses needed to apply the automorphy lifting
theorems of~\cite{BLGGT}; that is, $p>2(n+1)$, $F$ does not
contain a primitive~$p$th root of unity, and ~$\sbar$ is
polarizable, odd, and is irreducible when restricted
to~$G_{F(\zeta_p)}$. We say that a
representation~$s:G_F\to\GL_n(\Qpbar)$ is reasonable if its reduction
mod~$p$ is reasonable.

We say that a compatible system of $l$-adic representations of~$G_F$
is \emph{weakly irreducible} if for a positive density set of
primes~$l$, its $l$-adic Galois representations are
irreducible. Conjecturally, this is equivalent to all of the
$l$-adic representations being irreducible, but this seems to be very
hard to prove; weak irreducibility is a well-behaved substitute for
that stronger condition. By the results of~\cite{BLGGT,MR3314824}, if
a compatible system is odd, regular, and polarizable, then it is
weakly irreducible if and only if it is potentially
automorphic (in the sense that it potentially corresponds to a
cuspidal automorphic representation). 

The condition that the representation~$\sbar$ is polarizable is 
best expressed in terms of the group~$\cG_n$ introduced in~\cite{CHT};
it is a reductive group with connected component~$\GL_n\times\GL_1$
and component group of order~$2$, and as explained in~\cite[\S
8.3]{MR3444225}, it is very closely related to the $C$-group of an
$n$-dimensional unitary group over~$F^+$ which splits over~$F$. Then
the polarizability of~$\sbar$ is equivalent to the existence of a
prolongation of~$\sbar$ to a
representation~$\rhobar:G_{F^+}\to\cG_n(\Fpbar)$. In particular, this
implies that~$\sbar^c\cong\sbar^{\vee}\mubar|_{G_{F^+}}$ for some
character~$\mubar$ of~$G_{F^+}$.

Suppose that~$\rhobar:G_{F^+}\to\cG_n(\Fpbar)$ is a prolongation 
of $\sbar$, and let~$v$ be a finite place of~$F^+$.
A \emph{polarized component}
for~$\rhobar|_{G_{F^+_v}}$ is, by definition, an irreducible component of a deformation
ring  for 
lifts of~$\rhobar|_{G_{F^+_v}}$ which are of
some fixed inertial type, and of a fixed regular Hodge type
if~$v|p$. If~$v|p$, then we say that such a component is \emph{globally
realizable} if it occurs
globally, in the sense there is some totally real field~$L^+$ with a
CM quadratic extension~$L$, a place~$w|p$
of~$L^+$ for which $L^+_w\cong F^+_v$, and a (polarizable, odd, reasonable) representation $r:G_L\to\GL_n(\Qpbar)$ which is part of
a (polarizable, odd, weakly irreducible) compatible system, which prolongs to a representation~$\rho':G_{L^+}\to\cG_n(\Qpbar)$ whose restriction~$\rho'|_{G_{L^+_w}}$ gives rise to a
point on this component (so that in particular,
$\rhobar'|_{G_{L^+_w}}\cong \rhobar|_{G_{F^+_v}}$). 
If the place~$v$ splits in~$F$ as~$ww^c$, then
the deformation ring for~$\rhobar|_{G_{F^+_v}}$ can be identified with
a deformation ring for~$\sbar|_{G_{F_w}}$, as in~\cite{CHT}.
Conjecturally, every component is expected to be globally
realizable (and the analogous statement for places~$v\nmid p$ is
known), but proving this seems to be a very hard problem. 

 Our main theorem is the following (see Theorem~\ref{thm: main theorem
  on existence of lifts} for a more precise statement, and see the
notation and conventions section at the end of this introduction for
any unfamiliar terminology).

\begin{ithm}\label{thm: main theorem on existence of lifts - intro version} Assume
  that either $n$ is odd, or that $n=2$. Let~$F$ be a CM field, and let
 $\sbar: G_{F} \rightarrow \GL_n(\Fbar_p)$ be a reasonable
 representation, with 
 prolongation~$\rhobar$.

Let $S$ be a finite set of finite places of~$F^+$, such that~$S$ contains
all of the places at which~$\rhobar$ is ramified and all of the
places lying over~$p$. 
 For each place $v\in S$, let~$C_v$ be a  component
for~$\rhobar|_{G_{F^+_v}}$, which is globally realizable if~$v|p$.

 Then there exists an odd, regular, polarized, weakly irreducible compatible
system~$(\{s_{\lambda}\},\{\mu_\lambda\})$ of~$G_F$-representations
with associated $p$-adic representation~$s$, and a prolongation~$\rho$
of~$s$, which satisfies:
\begin{enumerate}
\item $\rho$ lifts~$\rhobar$, and for each place $v\in S$, the
  representation~$\rho |_{G_{F^+_v}}$ lies on~$C_v$.
\item $\rho$ is unramified outside~$S$. 
\end{enumerate}
\end{ithm}
Note that by the very definition of global realizability, the
hypothesis that each~$C_v$ is globally realizable is a necessary
condition for the conclusion of the theorem to hold. 

 The hypothesis that~$\sbar$ is reasonable is needed in order to apply the
theorems of~\cite{BLGGT}, and some restriction on~$p$, $n$ and the
size of the image of~$\sbar$ is certainly necessary; for example,  the results
of~\cite{2015arXiv150101344L} show that the analogous result
fails for modular forms of weight~$2$ if~$p=2$ (in fact there are also dihedral
counterexamples due to Serre with~$p=3$, see~\cite[\S 4.4]{MR1046750}). More generally,
calculations in Galois cohomology suggest that if $p\le n+1$,
and~$\sbar|_{G_{F(\zeta_p)}}$ is reducible, then it unreasonable to
hope for a global lifting result with control of the local
representations at all places. Thus the only unnaturally restrictive
hypothesis in Theorem~\ref{thm: main theorem on existence of lifts -
  intro version} is the exclusion of even integers~$n>2$, which is a
byproduct of our methods; 
this is because given a
compatible systems of $n$-dimensional polarizable $l$-adic
representations, we cannot deduce the oddness of all the
representations in the compatible system from the oddness of a single representation.

An almost immediate corollary of our results is the following
potential automorphy theorem, which may be of independent
interest. Subject to the restriction that~$n=2$ or~$n$ is odd, it
improves on~\cite[Thm.\ 4.5.1]{BLGGT} by replacing ``potentially
diagonalizable'' by ``globally realizable''. Note that if~$v|p$ in~$F^{+}$ splits in~$F$, then the global realizability
of~$\rho |_{G_{F^+_v}}$ only depends on~$s|_{G_{F_w}}$ for~$w|v$
in~$F$.

\begin{ithm}[Corollary~\ref{cor: potential automorphy for globally
    realizable}]\label{thm: potential automorphy for globally
    realizable intro version}Assume that either $n$ is odd or $n=2$. Let~$F$ be a CM
  field, and let~$(s,\mu)$ be a polarized representation, where
$$s: G_{F} \rightarrow \GL_n(\Qpbar)$$
is odd and ramified at only finitely many primes. 
Suppose that~$\sbar$ is reasonable. Let~$\rho$ be the
corresponding prolongation of~$s$,  and assume that~$\rho |_{G_{F^+_v}}$ is globally realizable for
each~$v|p$. Then~$(s,\mu)$ is potentially automorphic.
\end{ithm}
 In fact (as noted in the abstract),
we could  weaken the hypothesis that ~$\rho
|_{G_{F^+_v}}$ is globally realizable to requiring only that it is
potentially globally realizable, because this is equivalent to global
realizability by Corollary~\ref{cor: potentially globally realizable}.
Perhaps surprisingly,  if~$\sbar$
is automorphic, then we cannot deduce that ~$s$ is also automorphic;
this is because our methods make considerable use of potential
automorphy results for other representations in a
compatible system containing~$s$. On the other hand, if we
did know that weakly irreducible compatible systems are automorphic
(rather than just potentially automorphic), then a version of the
Breuil--M\'ezard conjecture for odd-dimensional globally realizable
representations and a version of the weight part of Serre's conjecture
for odd-dimensional unitary groups would both follow from combining
Theorem~\ref{thm: main theorem on existence of lifts - intro version}
with the methods of~\cite{MR3292675}.
\subsection{History and motivation}\label{subsec: history and motivation}

We now give a somewhat leisurely overview of our motivations and of
previous work on similar questions. 
Ultimately, the problems that we are working
on are motivated by congruences between modular forms; more specifically, we are concerned
with congruences between eigenforms. 
Such congruences can often best be understood in terms of the
corresponding Galois representations, and in particular in terms of
the restrictions of these (global) Galois representations to (local)
decomposition groups. It is therefore natural to wonder whether there is a
local to global principle for the existence of such congruences.

The first results in the literature that we know of that are
explicitly formulated in this way are those of~\cite{MR1272977}, which
we now recall. Let~$p>3$ be prime, and let $f$ be a newform of level
prime to~$p$ and weight~$2$. We can (after choosing an
embedding~$\Qbar\into\Qpbar$) associate a~$p$-adic Galois
representation $\rho_f:G_\Q\to\GL_2(\Qpbar)$ to~$f$, 
and thus a mod~$p$ representation
$\rhobar_f:G_\Q\to\GL_2(\Fpbar)$. We say that an irreducible
representation $\rhobar:G_\Q\to\GL_2(\Fpbar)$ or
$\rho:G_\Q\to\GL_2(\Qpbar)$ is \emph{modular of weight 2} if it is isomorphic to
some~$\rhobar_f$ (respectively~$\rho_f$).

Then the main result of~\cite{MR1272977} is as follows. Suppose we are
given a modular representation~$\rhobar:G_\Q\to\GL_2(\Fpbar)$ of weight~$2$, and
that for each prime~$l\ne p$, we are given a lifting
of~$\rhobar|_{G_{\Ql}}$ to a $p$-adic representation~$\rho_l:G_{\Ql}\to\GL_2(\Qpbar)$. Suppose
also that all but finitely many of the~$\rho_l$ are unramified. Then
there is a lift~$\rho:G_\Q\to\GL_2(\Qpbar)$ of~$\rhobar$ which is
modular of weight~$2$, with
the property that for all~$l\ne p$, we have an isomorphism of
restrictions to inertia $\rho|_{I_{\Ql}}\cong\rho_l|_{I_{\Ql}}$.

This result is best possible, in the sense that any modular
representation is necessarily only ramified at finitely many primes,
and that (for example, because spaces of modular forms are
finite-dimensional) it is unreasonable to pin down $\rho|_{G_{\Ql}}$
more than specifying it on inertia. In particular, by local-global
compatibility, the conductor of~$\rho|_{I_{\Ql}}$ determines the
$l$-power part of the level of the modular form associated to~$\rho$.

It is natural to wonder whether this result can be extended to cover
modular forms of higher weight or with $p$ dividing their level,
 and to allow some control
of~$\rho|_{G_{\Qp}}$.   Since the local
behavior at $p$ of $\rho_f$ is highly dependent on the weight of the
newform $f$ and the power of $p$ dividing its level,
these two questions are closely related,
and the most general local to global result of this
kind is best formulated in terms of components of local deformation
rings. If we fix a weight~$k\ge 2$, and a finite extension~$K/\Qp$,
then we can consider the deformation ring~$R^{k}_K$ for liftings
of~$\rhobar|_{G_{\Qp}}$ which become semistable over~$K$ with
Hodge--Tate weights~$0,k-1$. If~$\rho$ is modular of weight~$k$ and
some level, then for some sufficiently large~$K$, $\rho|_{G_{\Qp}}$
corresponds to a point of~$R_K^{k}$.
(For example, if the corresponding modular form has level prime to~$p$, then
we can take~$K=\Qp$.)

The spectrum of~$R_K^{k}$ has finitely many irreducible components, and
given such a component, we can ask whether there exists
a lift~$\rho$ with the property that~$\rho|_{G_{\Qp}}$ lies on this
component. This is the correct analogue of what we are
demanding at the places~$l\ne p$; indeed, it turns out that specifying
~$\rho|_{I_{\Ql}}$ is equivalent to demanding that~$\rho|_{G_{\Ql}}$
lies on a particular component of a deformation ring
for~$\rhobar|_{G_{\Ql}}$. Proving that the lifting problem still
admits a solution when we specify a component at~$p$ is much harder
then the case in which we only specify components away from $p$,
but it follows from the results of~\cite{MR2505297} that such a lift
exists under a mild (``Taylor--Wiles'') condition on~$\rhobar$. In
particular, 
this again gives a complete
understanding of the possibilities for the weight and 
level of the corresponding modular form.

When we restrict to the case that~$2\le k\le p+1$ and~$K=\Qp$, the
lifting problem is closely related to Serre's
conjectures~\cite{MR885783} on the weight and level of modular Galois
representations. For example, while not formulated in this way,
Serre's conjecture on the minimal weight and level is equivalent to
asking that~$\rhobar$ admits a modular lift whose ramification away
from~$p$ is as small as possible (as measured by the Swan
conductor), and whose weight is as small as possible, compatible with
the property of locally having a crystalline lift of the corresponding
Hodge--Tate weights.

It is relatively straightforward to formulate conjectural
generalizations of these results. For example, a detailed formulation
of a generalization of Serre's conjectures to Hilbert modular
forms was made in~\cite{bdj}, and the weights are described in terms of the
existence of local crystalline lifts in a similar fashion to that
described above. One can make similar conjectures for automorphic
representations on unitary groups over CM fields (or equivalently, for conjugate
self dual automorphic representations of~$\GL_n$ over CM fields), and
it is these generalizations that will concern us below. (However, we
don't expect any straightforward generalization of these results to
hold outside of settings which are discrete series at the infinite
places.)

While some of the arguments of~\cite{MR1272977} and of the related
papers on the weight and level parts of Serre's conjecture can be
generalized to Hilbert modular forms, it seems hard to adapt them to
prove the conjectures of~\cite{bdj} completely, and much harder still
to study congruences between forms on~$\GL_n$ in this way. 
However, in the mid 2000s a new approach to these problems was
discovered by Khare--Wintenberger~\cite{MR2480604} and the third
author~\cite{MR2286631}, which we will now describe in the setting of
modular forms.

The approach is via the deformation theory of global Galois
representations and automorphy lifting theorems. Suppose that as above
we are given components of deformation rings for~$\rhobar|_{G_{\Ql}}$
for each prime~$l$, which are unramified for all but finitely
many~$l$. Then there is a corresponding deformation ring~$R^\univ$ for
the global representation~$\rhobar$, and the~$\Qpbar$-points of its
spectrum precisely correspond to the Galois representations that
we are hoping
to construct. If we can show that the set of $\Qpbar$-points is
nonempty, then we can hope to show that the Galois representations are
modular using modularity lifting theorems (the Taylor--Wiles method). 

The tangent space to ~$R^\univ$ 
can be computed by Galois
cohomology, and it turns out 
that~$R^\univ$ always has dimension at
least~$1$. (This computation relies on the weight~$k$ being at
least~$2$, and more generally on us being in a discrete series
context.) Heuristic arguments lead us to expect that~$R^\univ$ is a
finite~$\Zp$-algebra, and if this is the case, the lower bound on the
dimension guarantees the existence of $\Qpbar$-points. 

There is no known purely Galois-theoretic argument guaranteeing this
finiteness in general (although it can sometimes be arranged at the
cost of allowing additional ramification away from~$p$ by an argument
of Ramakrishna~\cite{MR1935843}). However, modularity lifting theorems
are proved by identifying deformation rings such as~$R^\univ$ with
Hecke algebras, which are finite over~$\Zp$ by definition, so in
principle it is enough to prove an appropriate modularity lifting
theorem (which can then be used in the final step of the argument to
deduce that the Galois representations that we have constructed are
actually modular).

Unfortunately, this argument is circular as written, because what the
Taylor--Wiles method allows us to prove is that if some~$\Qpbar$-point
of $R^\univ$ is modular, then~$R^\univ$ may be identified with a Hecke
algebra; but it gives us no assistance with producing a $\Qpbar$-point
in the first place. A key insight of Khare--Wintenberger is that this
argument can be combined with base change and/or potential modularity to
avoid the circularity. Suppose that
$F$ is a totally real finite extension of $\Q$,
and that~$\rhobar|_{G_F}$ is irreducible. Then we
may consider the deformation problem for~$\rhobar|_{G_F}$ given by 
the restrictions to places of~$F$ of the conditions we imposed
over~$\Q$, and the corresponding deformation ring~$R^\univ_F$ is again
of dimension at least one.

Now, by definition~$R^\univ$ is an~$R^\univ_F$-algebra, and it is in
fact a finite $R^\univ_F$-algebra. It is therefore enough to prove a
modularity lifting theorem for~$R^\univ_F$. 
This allows us to reprove many cases of the theorem
of~\cite{MR1272977}, in the following way. Suppose for simplicity that
for each prime~$l\ne p$, both the original modular
representation~$\rho_f|_{G_{\Ql}}$ and the local
representation~$\rho_l$ are finitely ramified (that is, they become
unramified after restriction to a finite extension of~$\Ql$). Then we
may choose a finite solvable totally real extension~$F/\Q$ so that the
restrictions to the finite places of~$F$ of these representations are
actually unramified. In particular, the corresponding restrictions of
$\rho_f|_{G_{\Ql}}$ and the local representation~$\rho_l$ lie on the
same component of the corresponding local deformation ring (indeed,
the unramified local deformation ring is formally smooth).

By solvable base change, $\rho_f|_{G_F}$ is modular, and by
the choice of $F$, it gives a point of $R^\univ_F$.
A modularity lifting theorem then shows that $R^\univ_F$ is a finite
$\Zp$-algebra. Thus the same is true of $R^\univ$,
and so $R^\univ$ has $\Qpbar$-points, which give the
sought-after Galois representations; the modularity of these
representations again follows from solvable base change. The more
general case, in which the ramification can be potentially unipotent,
can be handled in the same way when given some level-raising and
level-lowering results over~$F$; by choosing~$F$ appropriately, one
reduces to a relatively straightforward case (see~\cite{MR1815248}). 

A variant on this argument makes it possible to state and prove
results about Galois representations that make no reference to
automorphic forms (although the proofs make heavy use of automorphic
techniques). To this end, rather than assuming that~$\rhobar$ is
modular, assume only that it is irreducible and odd, in the sense
that~$\rhobar(c)$ is non-scalar, where~$c$ is a complex
conjugation. (Of course, since Serre's conjecture is a theorem, this
implies that~$\rhobar$ is modular, but we can and will make an
analogous assumption in more general contexts where the analogue of
Serre's conjecture is open.) Then the same Galois cohomology
calculations go through, and if we want to produce lifts of~$\rhobar$
with specified local properties, it is enough by the above arguments
to find a finite (not necessarily solvable) extension of totally real fields~$F/\Q$ for
which~$\rhobar|_{G_F}$ is modular (that is, it comes from a Hilbert
modular form).

An argument of
Taylor~\cite{MR1954941,MR2290604} can be used to
prove such ``potential modularity'' theorems. The idea is as follows:
one can find a moduli space whose $F$-points 
correspond to abelian varieties, part of whose $p$-torsion is
isomorphic to~$\rhobar|_{G_F}$, and the corresponding part of whose
$l$-torsion, for some fixed prime~$l\ne p$, is isomorphic to an
induction of a character. Since inductions of characters are always
modular, in favourable circumstances one can use modularity lifting
theorems to prove that (part of) the $l$-adic Tate module of the
corresponding abelian variety is modular, and thus that (part of) the
$p$-adic Tate module is modular, and finally that~$\rhobar|_{G_F}$ is
modular. That $F$-points exist for~$F$ sufficiently large follows from
a theorem of Moret--Bailly, 
which also allows one to
impose the kinds of local conditions that are needed in order to apply
modularity lifting theorems.

As well as producing lifts of~$\rhobar$ with specified
local properties, it turns out that potential modularity allows one
to prove that that each $p$-adic representation~$\rho$ that is
constructed in this way is
part of a compatible system of $l$-adic representations. Indeed, this
property is automatic for Galois representations associated to
automorphic forms, so that~$\rho|_{G_F}$ is part of a compatible
system. By solvable base change, the same is true for~$\rho|_{G_{F'}}$
whenever $F/F'$ is a solvable extension of totally real fields, and an
argument with Brauer's theorem (see~\cite{dieulefait2}) makes it possible to put these together
to give the required compatible system.

We now digress  briefly to discuss another aspect of
compatible systems that will be of fundamental importance throughout
this paper. In general, it seems to be hopeless to understand the
components of local potentially semistable deformation rings of
mod~$p$ representations in any concrete way, and this in turn places
serious restrictions on automorphy lifting theorems. However, it is
possible to understand them in the Fontaine--Laffaille case, 
which by definition is case that~$p$ is unramified in
the base field, and the weight is small relative to~$p$. It is also very hard to prove automorphy lifting
theorems when the global mod~$p$ Galois representation has small image
(in particular, when the image is reducible). This can make it very
difficult to prove the potential automorphy of a given Galois
representation (or its automorphy, even if we know that it is
residually automorphic). However, if instead we are given a weakly
irreducible compatible
system of $l$-adic Galois representations,
then for all large~$l$, the
$l$-adic representation will be Fontaine--Laffaille, and its residual
image will be suitably large. Accordingly, one can hope to prove the potential
automorphy of the $l$-adic representation for some~$l$, and
immediately deduce the potential automorphy for all~$l$. As we now
explain, this has been carried out in considerable generality. In our
arguments, this will allow us to use compatible systems of
Galois representations as a kind of proxy for automorphic
representations, without assuming the Fontaine--Mazur conjecture.

We briefly review the history of higher-dimensional potential
automorphy theorems. Many of the arguments discussed above
were generalized to polarizable (that is,
essentially conjugate self-dual) $n$-dimensional Galois
representations of CM fields in the
papers~\cite{CHT,MR2470688,MR2630056} (the corresponding automorphic
representations being those on general unitary groups). In
particular,~\cite{MR2470688} uses automorphy lifting techniques to
prove the kind of level raising/lowering results that we applied after
base change above, and the third author's paper~\cite{MR2785764}
deduced an $n$-dimensional version of the theorem of Diamond--Taylor
in low weight.

Potential modularity, while powerful, has its limitations, the chief
of which is that the method explained above only works for modular forms of weight~$2$
(that is, for Galois representations with Hodge--Tate weights~$0,1$),
because for reasons of Griffiths transversality, the required moduli
spaces only exist in this case. Allied to this  is the difficulty
mentioned above that the deformation rings for~$\rhobar_{| G_{\Qp}}$ are
much more complicated than those for~$l\ne p$, and much less well
understood (for example, it is certainly no longer the case that
deformations are potentially unipotent on inertia). Indeed, when~$k$ is large
compared to~$p$, very little is known about the components of these
deformation rings, and we do not know whether we can make base changes
to make representations lie on the same component in any generality, 
which limits our ability to change components at~$p$ (recall that for~$l\ne p$, in
the discussion above we used that any finitely ramified
representation can be made to lie on the unique unramified component
after a suitable base change). (It is however
worth noting that when~$k=2$ arguments of this kind are possible, even
for Hilbert modular forms over totally real fields in which~$p$ is
highly ramified, by the results of~\cite{MR2600871}
and~\cite{MR2280776}. These results were a crucial part of the proof
of the weight part of Serre's conjecture for Hilbert modular
forms~\cite{MR3292675}, and the lack of anything similar for
higher-dimensional representations is one of the reasons that less
is known about the weight part of Serre's conjecture in dimension
greater than~$2$.)

In higher dimensions the situation is worse; the potential automorphy
theorems of~\cite{MR2630056} apply only in weight~$0$ (that
is, the lowest discrete series; the corresponding Galois
representations have Hodge--Tate weights $0,1,\dots,n-1$), and in fact only to ordinary Galois
representations. While this was enough to prove the Sato--Tate
conjecture (by proving the potential automorphy of the symmetric
powers of the 2-dimensional Galois representations attached to
elliptic curves, which are ordinary at most primes), it falls well
short of proving the potential modularity of compatible systems of
Galois representations in any generality.
This shortcoming was resolved in~\cite{BLGGT}, which also introduced a
way of systematically changing the weight of the representations, or
more generally of moving between components of deformation rings
at~$p$. The argument involves a refinement of a method of
Harris~\cite{harris:manin} (the ``tensor product trick''), the basic
idea of which is as follows: given a global Galois representation with
regular Hodge--Tate weights, by taking the tensor product with
another representation, one can produce a representation (of much
higher dimension) of weight~$0$. It is then possible to apply
potential modularity techniques to this representation. 

Of course, to
be useful, one has to have a way of ``undoing'' the tensor product
again (on both the Galois and automorphic sides). In general, this is a hard problem, but if the representation we
tensor with is the induction of a character, it turns out to be
relatively straightforward, essentially because the tensor product of
an $n$-dimensional representation with the induction of a character is
itself the induction of an $n$-dimensional representation.

We also obtain a way of moving between weights and between components
of deformation rings at places dividing~$p$, in the following
way. Let~$\rhobar:G_F\to\GL_n(\Fpbar)$ be our original mod~$p$
representation, and let~$r:G_F\to\GL_n(\Qpbar)$ be induced from a
character. Fix a deformation ring~$R^\univ_1$ for $\rhobar$
corresponding to certain given local conditions. Given a
deformation~$\rho$ of~$\rhobar$, we obtain a
deformation~$\rho\otimes r$ of~$\rhobar\otimes\rbar$, and we
let~$R_2^\univ$ be the corresponding deformation ring. This procedure
makes~$R^\univ$ into an~$R_2^\univ$-algebra, and an argument analogous
to the one of Khare--Wintenberger that we mentioned above shows that
it is even a finite~$R^\univ_2$-algebra. 

It is not immediately obvious that this buys us anything, as to apply
the techniques we used above we will need to prove an automorphy
lifting theorem for~$R^\univ_2$. To this end, suppose that we also have another deformation
ring~$R^\univ_3$ for~$\rhobar$, and another lift~$r'$  of~$\rbar$, and
that the deformation problem corresponding to the tensor product of a
deformation for~$R^\univ_3$ with~$r'$  again corresponds to~$R^\univ_2$. If~$r$ and~$r'$ are both inductions of characters, and
if we know that~$R^\univ_3$ has automorphic points, then we can prove
an automorphy lifting theorem for~$R^\univ_2$, deduce its finiteness
over~$\Zp$, and then prove the existence of lifts (and automorphy
lifting) for~$R^\univ_1$, the original problem of interest.

It may not be obvious that this is generally applicable, but in fact
in combination with base change techniques it gives enough flexibility
to prove the potential automorphy of compatible families, by moving
between Fontaine--Laffaille and ordinary weight~$0$ deformation problems. A little
thought shows that this argument allows us to move freely between
components of local deformation rings at places dividing~$p$, provided
that the corresponding representations are \emph{potentially
  diagonalizable}, in the sense that after some base change, the
components contain a point which is a direct sum of crystalline
characters.  It turns out to be straightforward to show that
Fontaine--Laffaille representations and ordinary representations are
potentially diagonalizable, giving the claimed potential 
automorphy result. (The most general
result about the existence of Galois representations proved
in~\cite{BLGGT} is essentially Theorem~\ref{thm: main theorem on
  existence of lifts - intro version} above, but with ``globally
realizable'' replaced by ``potentially diagonalizable''.)

While this gives a general potential automorphy result for compatible
systems, it is unsatisfactory as an answer to our original lifting
question, due to the restriction to potential diagonalizable representations. This
restriction is problematic for two reasons: firstly, beyond the
Fontaine--Laffaille and ordinary cases, nothing is known about the
potential diagonalizability or otherwise of $n$-dimensional
representations. (It is quite plausible that all potentially crystalline
representations are potentially diagonalizable, but this seems to be
very hard to establish.) In addition, potentially diagonalizable
representations are by definition potentially crystalline, so they do
not (unlike Theorem~\ref{thm: main theorem on
  existence of lifts - intro version}) tell us anything about general potentially semistable
representations.

\subsection{A sketch of the proofs}\label{subsec: sketch of proofs}
We now briefly sketch the proof of Theorem~\ref{thm: main theorem on
  existence of lifts - intro version}, omitting many of the
trickier technical details; in particular, we completely ignore places
away from~$p$ in our discussions, as the difficulties they present are
similar to, but simpler than, the difficulties for places
dividing~$p$. We also suppress all mention of choices of polarization. In essence, our idea is to
go beyond the potentially diagonalizable case, by allowing ourselves
to tensor with representations that are not necessarily induced from
characters. There are some obvious difficulties with this approach,
chief among them that on both the automorphic and Galois sides it is
hard to ``undo'' a tensor product. We overcome this by using
compatible systems, rather than individual Galois representations.

Our first technique is a variant of the argument of~\cite{BLGGT}
explained above, that of tensoring with auxiliary representations in
order to move between different components. However, for our purposes
it is insufficient to tensor with a fixed global
representation. Instead, we put ourselves in the following situation:
suppose given representations~$r,s:G_F\to\GL_n(\Qpbar)$ which belong to
 compatible systems, and suppose for simplicity that for each
place $v|p$, we have
$\rbar|_{G_{F_v}}\cong\sbar|_{G_{F_v}}$. Let~$A_v$ be the component
corresponding to~$r|_{G_{F_v}}$, and let~$B_v$ be the component
corresponding to~$s|_{G_{F_v}}$. For each~$v$, let $C_v$ be one
of~$A_v$ and~$B_v$, and let~$D_v$ be the other. Let~$R_C$ and~$R_D$ be
the corresponding global deformation rings for~$\rbar$ and~$\sbar$
respectively. (We should really be considering prolongations
of these representations to $\cG_n$-valued representations
of~$G_{F^+}$, but we ignore this point for the purposes of this introduction.)

We would like to produce representations $r',s'$ corresponding to points
of~$R_C$ and~$R_D$, which belong to compatible systems; in this way we
will be able to swap components between different residual
representations. We initially accomplish this under very restrictive
hypotheses, which we will later relax; note that we certainly need to
assume at first that $\rbar\otimes\sbar$ is irreducible, and
that~$r\otimes s$ has regular Hodge--Tate weights. Now, the
deformation ring~$R^\univ$ corresponding to~$r\otimes s$ is finite
over~$\Zp$, because~$r\otimes s$ is part of a compatible system, and
thus potentially automorphic. Taking the tensor product of
representations coming from~$R_C$ and~$R_D$ makes
$R_C\widehat{\otimes}R_D$ into an~$R^\univ$-algebra, and we are able
to show that it is in fact a finite $R^\univ$-algebra, so that in the
same way as before, we can see that~$R_C$, $R_D$ both
have~$\Qpbar$-points, which will correspond to the lifts $r',s'$ that we want
to produce. 

However, we need to show that these lifts belong to compatible
systems. The tensor product $r'\otimes s'$ does belong to a compatible
system of representations~$\{t_l\}$ (because it corresponds to a point of~$R^\univ$, and is thus
potentially automorphic), so we would like to have a result ensuring that
if one representation in a compatible system is a tensor product, they
all are. This should be true quite generally, but it
seems hard to prove, and we only establish it under rather restrictive
hypotheses (which are ultimately sufficient for our needs). 

In outline, we argue as follows. We assume that the Zariski closure of
the image of~$s'$ contains~$\SL_n(\Qpbar)$. 
 Under a strong
irreducibility hypothesis (which can be arranged at infinitely many primes~$l$
by imposing local
conditions), we can show that the representations~$t_l$ decompose
accordingly as  tensor products $t_l=r'_l\otimes s'_l$, where the Zariski
closure of the image of~$s'_l$  contains~$\SL_n(\Qlbar)$.  Here
we make crucial use of a result of Larsen and Pink~\cite{MR1150604}.
Then, if~$l$ is
sufficiently large, the representations~$r'_l$ and~$s'_l$ belong to
compatible families by the results of~\cite{BLGGT}; this is the only
place that we need our assumption on~$n$, as we have to know
that~$r'_l$ and~$s'_l$ are odd in the sense of~\cite{MR2834723}. This
oddness is automatic if~$n$ is odd, and can be proved if~$n=2$ by the
methods of~\cite{MR2810794,MR2869026}, but it seems to be beyond the
reach of current technology if $n>2$ is even. Our assumptions then
show that (after possibly twisting) the~$p$-adic representations in
these compatible systems are~$r'$ and~$s'$, as required.

With this component swapping result in hand, the basic outline of the
proof of Theorem~\ref{thm: main theorem on existence of lifts - intro
  version} is as follows. Suppose for simplicity that there is only a
single place~$v|p$ of~$F$, let~$r$ be some lift of~$\rbar$, and let~$C$ be the given globally realizable
component. By definition, this
means that there is another CM field~$L$, a place~$w|p$ and a
representation~$s$ (which is part of a compatible system) such
that~$s|_{G_{L_w}}$ lies on~$C$. Then if we apply our swapping result
to~$r|_{G_{FL}}$ and~$s|_{G_{FL}}$, we can produce a lift
of~$\rbar|_{G_{FL}}$ which lies on the correct component at some place
over~$v$. If we could produce a lift with this property at all places
above~$v$, then we would be done by the usual Khare--Wintenberger
method; in order to do this, we replace~$FL$ with its Galois
closure, and~$s$ with its various Galois conjugates, and then
inductively apply the swapping result to each of these conjugates in
turn.

Unfortunately, the actual argument is much more complicated than
this straightforward outline. The problem is that all of the results
that we are applying have hypotheses that we have been ignoring; for
example, we need~$\rbar\otimes\sbar$ to be irreducible, we
need~$r\otimes s$ to have regular Hodge--Tate weights, and we need to
satisfy the restrictive hypotheses of our main swapping result, which
are a mixture of local and global assumptions. 

We are able to handle the various local assumptions away from~$p$ by
more base change tricks, but these cannot help with the global
problems. Since the fields~$F,L$ and the representations~$\rbar,\sbar$
are arbitrary, we cannot hope to arrange that their restrictions to
the Galois closure of~$FL$ are irreducible. Instead, we make use of an
idea introduced in~\cite{MR2869026}, and use the theorem of
Moret--Bailly mentioned above to construct auxiliary global
representations with large image, which locally admit potentially
diagonalizable lifts of arbitrarily large weights.

These representations are constructed over extensions
of~$F,L$ that we have little control over, and we have to go to some
lengths to ensure that we can arrange all of the properties we
need. Rather than swapping directly between~$\rbar$ and~$\sbar$, we
instead make a long chain of swaps, going via many auxiliary
representations, and making many base changes and descents by the
Khare--Wintenberger argument.

\subsection{Acknowledgements}
The authors are grateful to Florian Herzig and the anonymous referee for their close readings of an
earlier version of the paper, and their many helpful remarks
and suggestions, which have improved its readability. We particularly
thank the referee for pointing out a gap in the proof of
Theorem~\ref{thm: first factoring theorem} in the case~$m=n=2$, which
led  us to add Lemma~\ref{added}.

\subsection{Notation, conventions, and background material}\label{subsec:
  notation}
  All representations considered in this paper are assumed to be
continuous with respect to the natural topologies, and we will never
draw attention to this.  If~$M/F$ is an extension of number fields, then we will
write~$M^{\gal}$ for the Galois closure of~$M$ over~$\Q$, and~$\Mg$
 for
the Galois closure of~$M$ over~$F$.
 As
usual, if~$K$ is a field of characteristic zero, then we write
$G_K=\Gal(\overline{K}/K)$ for its absolute Galois group, and if~$K$
is furthermore a local field, we write~$I_K$  for the inertia subgroup of~$G_K$.

\subsubsection{Polarizable representations}\label{subsubsec: polarizable
  repns}We begin by recalling some definitions and results
from~\cite{CHT,BLGGT} concerning polarizable representations. 

Recall from~\cite{CHT} that the reductive group $\cG_n$
over~$\Z$ is given by the semi-direct product of $\cG_n^0=\GL_n \times \GL_1$ by the group $\{1 , \jmath\}$
where 
\[ \jmath \kern+0.1em{(g,a)} \jmath^{-1}=(a \cdot {}^tg^{-1},a). \]
 We let $\nu:\cG_n \to \GL_1$ be the character which sends $(g,a)$ to $a$ and sends $\jmath$ to $-1$.

Let~$\Gamma$ be a group, with an index~2 subgroup~$\Delta$. Fix an
element~$\gamma_0\in\Gamma\setminus \Delta$. Let~$R$ be a
(commutative) ring. Then by ~\cite[Lem.\
2.1.1]{CHT}, there is a natural bijection between:
\begin{itemize}
\item the set of homomorphisms $\rho:\Gamma\to\cG_n(R)$ which induce
  isomorphisms $\Gamma/\Delta\isoto\cG_n/\cG_n^0$, and
\item the set of triples $( r,\mu,\langle \, ,\rangle )$ consisting of
  homomorphisms $r:\Delta\to\GL_n(R)$ and~$\mu:\Gamma\to R^\times$,
  and a perfect $R$-linear pairing
  \[\langle\, ,\rangle:R^n\times R^n\to R\]
  which for all $x,y\in R^n$ and $\delta\in\Delta$ satisfies
  \begin{itemize}
  \item $\langle x,r(\gamma_0^2)y\rangle=-\mu(\gamma_0)\langle
    y,x\rangle$, and 
  \item
	 $ \langle r(\delta)x,r(\gamma_0\delta\gamma_0^{-1})y\rangle
	 =
\mu(\delta)\langle x,y\rangle.$
  \end{itemize}
\end{itemize}
This correspondence is given by taking $r= \text{ proj.\ onto
	first factor of }\cG_n^{\circ}\text{ of } \rho|_\Delta$, and
$\mu=\nu\circ \rho$, and setting \[\langle x,y\rangle={}^tx
  A^{-1}y,\]where $\rho(\gamma_0)=(A,-\mu(\gamma_0))\jmath$. We say
that the pair~$(r,\mu)$ is \emph{polarized}\footnote{This is established
	terminology, and so we use it here.  Note though that
	in general the pair $(r,\mu)$ may not determine
	the pairing $\langle \, , \rangle $ uniquely, even
       up to a scalar multiple.  However, if $R$ is a field
and $r$ is absolutely irreducible, then, as we will observe below,
the pairing $\langle \, , \rangle$ {\em is} uniquely determined
up to a scalar.} and that~$r$
is~\emph{polarizable}, and is $\mu$-\emph{polarized}.
If we are given a polarized pair $(r,\mu)$, then we will sometimes
refer to a corresponding homomorphism $\rho: \Gamma \to \cG_n(R)$
(which depends on the choice of $\gamma_0$, as well as on a choice
of pairing $\langle \, , \rangle$ witnessing the polarizability
of $(r,\mu)$) as a {\em prolongation} of the pair $(r,\mu)$.

Given a polarized pair~$(r,\mu)$, we call~$\mu$ the \emph{multiplier character} of the
pair~$(r,\mu)$. 
Given two polarized representations~$(r_1,\mu_1)$
and~$(r_2,\mu_2)$, there is a polarized
representation~$(r_1\otimes r_2,\delta_{\Gamma/\Delta}\mu_1\mu_2)$,
where~$\delta_{\Gamma/\Delta}$ denotes the 
 unique non-trivial character of~$\Gamma/\Delta$ (see~\cite[\S
 1.1]{BLGGT} for the explicit description of this construction as an
 operation on $\cG_n$-valued representations).

Suppose now that~$R$ is a complete local Noetherian ring, that
$r:\Delta\to\GL_n(R)$ is such that $r\bmod \m_R$ is absolutely irreducible, and that $\mu:\Gamma\to R^\times$ is a
character such that 
\numequation
\label{eqn:twisted iso}
r^{\gamma_0}\cong r^\vee\otimes\mu|_\Delta. 
\end{equation}
Giving such an isomorphism is equivalent to giving a pairing
  \[\langle\, ,\rangle:R^n\times R^n\to R\]
  which for all $x,y\in R^n$ and $\delta\in\Delta$ satisfies
	 $ \langle r(\delta)x,r(\gamma_0\delta\gamma_0^{-1})y\rangle
	 =
\mu(\delta)\langle x,y\rangle$
  for all $x,y\in R^n$ and $\delta\in\Delta$.
Since the isomorphism~(\ref{eqn:twisted iso})
is unique up to scaling by elements of $R^{\times}$
(because of our assumption that $r\bmod \m_R$ is absolutely irreducible),
we see that the corresponding pairing $\langle\, , \rangle$ is also
unique up to scaling.
In particular, if $(r,\mu)$ is polarized, 
then the pairing~$\langle\, , \rangle$ that yields a 
prolongation of $(r,\mu)$ is unique up to scaling.

If $\rho$ is one particular prolongation, corresponding to a pairing
$\langle\, , \rangle$, then we see that conjugating $\rho$ by
the element $(1,\lambda^{-1}) \in 
\GL_n(R)\times\GL_1(R) = \cG_n^{\circ}(R) \subset
\cG_n(R)$
scales $\langle \, , \rangle$ by $\lambda$;
the relevant computation is that
$$\begin{aligned} (1,\lambda^{-1})(g,a) \jmath (1,\lambda)
= & \ (1,\lambda^{-1})(g,a) \jmath (1,\lambda) \jmath^{-1} \jmath \\ 
 = & \ (1,\lambda^{-1})(g,a) (\lambda,\lambda) \jmath \\ 
 = & \ (\lambda g,a) \jmath. \end{aligned}$$
Thus we see that all possible prolongations of $(r,\mu)$
are obtained from the given
prolongation $\rho$ by such conjugations.
We also see that the possible pairings arising from the choice
of a prolongation are independent of the choice of the element
$\gamma_0 \in \Gamma\setminus \Delta$ used to construction the
bijection described above between (certain) homomorphisms
$\rho:\Gamma \to \cG_n(R)$ and (certain)
triples $(r,\mu,\langle \, , \rangle).$

We now consider the particular case that~$\Gamma=G_{F^+}$,
$\Delta=G_F$, where~$F$ is a (totally complex) CM field with maximal totally real
subfield~$F^+$. 
We say that the pair~$(r,\mu)$ is \emph{polarized and odd} if it is
polarized, and for all complex conjugations~$c\in G_{F^+}$, we have
~$\mu(c)=-1$. In particular, we have the following standard lemma.

\begin{lem}\label{lem: odd implies odd}
Suppose that the characteristic of~$R/\m_R$  is not~$2$.
  If~$(r,\mu)$ is polarized, ~$n$ is odd, and~$r\bmod \m_R$ is
  absolutely irreducible, then~$(r,\mu)$ is automatically odd. 
\end{lem}
\begin{proof}
This follows from the fact that any odd-dimensional
perfect pairing that is preserved up to scaling by a residually
absolutely irreducible
group action (in characteristics other than~$2$) is necessarily symmetric.
Indeed, let~$c$ be any complex conjugation,
take $\gamma_0$ equal to $c$,
and let $\langle \, , \rangle$ 
denote the pairing arising from a choice of prolongation.
Since $c^2 = 1$, 
we find (using the first of the properties satisfied by the pairing
arising from a prolongation)
that $\langle x,y\rangle = -\mu(c) \langle y,x\rangle.$
On the other hand, as we already remarked, the pairing is
necessarily symmetric.  Thus we find that $\mu(c) = -1$,
as required.
\end{proof}

We can restrict a global representation $\rho:G_{F^+}\to\cG_n(R)$ to
the decomposition group~$G_{F_v^+}$ of any finite place~$v$
of~$F^+$. Note that if $v$ is inert or ramified in~$F$,
then~$G_{F_v^+}$ is not contained in~$G_F$, so we are in the situation
above with~$\Gamma=G_{F^+_v}$ and~$\Delta=G_{F_v}$. If however~$v$
splits in~$F$ then $G_{F^+_v}$ is contained in~$G_F$, so that
$\rho(G_{F^+_v})\subset \cG_n^0(R)=\GL_n(R)\times R^\times$, so that
($\mu$ being fixed)
the data of the 
representation~$\rho$
is the same as
the data of the corresponding
representation $r:G_{F^+_v}\to\GL_n(R)$.

\subsubsection{Compatible systems}Let $F$ be a number field. We 
recall some definitions from~\cite[\S5]{BLGGT}. Note that what
we call a ``compatible system'' is a ``weakly compatible system''
in~\cite{BLGGT}.

By a {\em  compatible system} $\CR$
{\em of} $n$-dimensional representations of~$G_F$ {\em defined over} $M$ we shall mean a $5$-tuple 
\[ (M,S,\{ Q_v(X) \}, \{r_\lambda \}, \{H_\tau\} ) \]
where
\begin{enumerate}
\item $M$ is a number field.
\item $S$ is a finite set of primes of $F$.
\item for each  prime $v\not\in S$ of $F$, $Q_v(X)$ is a monic degree $n$
polynomial in $M[X]$.
\item for each prime $\lambda$ of $M$ (with residue characteristic $l$, say) 
\[r_\lambda:G_F \to \GL_n(\overline{M}_\lambda) \]
is a semisimple representation such that 
\begin{itemize}
\item if $v \not\in S$ is a prime of~$F$ and $v \nmid l$, then $r_\lambda$
is unramified at $v$ and $r_\lambda(\Frob_v)$ has characteristic
polynomial $Q_v(X)$.
\item if $v|l$ then $r_\lambda|_{G_{F_v}}$ is de Rham  and in the case $v \not\in S$ crystalline.
\end{itemize}
\item for $\tau:F \into \overline{M}$, $H_\tau$ is a multiset of $n$ integers such that 
for any $\overline{M} \into \overline{M}_\lambda$ over $M$, the
$\tau$-labelled Hodge--Tate weights of~$r_\lambda$ are~$H_\tau$.
\end{enumerate}
We will call $\CR$ {\em regular} if for each $\tau:F \into \overline{M}$ every element of $H_\tau$ has multiplicity~$1$. We will refer to a rank $1$  compatible system of representations as a  compatible system of characters.

\begin{remark} 
	\label{rem:unramified for compatible systems}
By abuse of terminology, we refer to a collection of Galois representations~$\{r_{\lambda}\}$  as a compatible system if it   extends to a~$5$-tuple~$\CR$ as above. In this case, we say that the compatible system~$\{r_{\lambda}\}$ is unramified outside~$S$ if it extends to such a $5$-tuple
with the given finite set~$S$.  Note that if~$\{r_{\lambda}\}$ is unramified outside~$S$, then the individual
representations~$r_{\lambda}$ are unramified outside~$S \cup \{v|l\}$ where~$\lambda$ has residue
characteristic~$l$.
\end{remark}

\begin{rem}\label{rem: unram outside of S for extensions.} 
	  By a slight abuse of terminology, if~$F'/F$ is a finite extension of
  number fields, and~$S$ is a finite set of places of~$F$, then we
  will sometimes say that a compatible system of representations
  of~$G_{F'}$ is unramified outside~$S$ if it is unramified outside of
  the set of places of~$F'$ lying over~$S$. Similarly, we will say
  that an extension~$F''/F'$ is unramified outside of~$S$ if it is
  unramified at all places of~$F'$ not lying over a place of~$S$, and
  so on. In particular, we will apply frequently apply this convention
  to quadratic extensions $F/F^+$ where~$F$ is CM with maximal totally
  real field~$F^+$.
\end{rem}

Note that if $M'/M$ is a finite extension, then a compatible
system defined over $M$ naturally determines a compatible
system with $M'$-coefficients. We regard these two compatible systems
as equivalent. Similarly one can enlarge~$S$, and we also regard
compatible systems associated in this way as equivalent. We may then
consider the equivalence classes of the equivalence relation generated
by these equivalences, and it follows
easily from~\cite[Lem.\ 1.1]{MR3314824} that to each equivalence class of compatible
systems is associated a minimal choice of~$M$, namely the field
generated by the coefficients of the polynomials~$Q_v(X)$. 

For this reason, we generally suppress~$M$ in the below. Somewhat abusively,
we shall often assume that~$M$ comes with a fixed
embedding~$M \hookrightarrow \Qbar_p$ for each prime~$p$, and hence talk of {\em the}~$p$-adic
representation
$$s: G_F \rightarrow \GL_n(\Qbar_p)$$
associated to~$\{s_{\lambda}\}$. 

We also introduce the following convenient short-hand terminology.

\begin{defn}
  \label{defn: FL representations}Let~$F$ be a number field, and let
  $r:G_F\to\GL_n(\Qlbar)$ be a representation. Then we say that~$r$ is
  Fontaine--Laffaille, or Fontaine--Laffaille at all primes dividing $l$
(for emphasis), if~$l$ is unramified in~$F$, and for all
  $\tau:F\into\Qlbar$, the~$\tau$-labelled Hodge--Tate weights of~$r$
  are contained in an interval of length~$(l-2)$ (the precise interval
  possibly depending on~$\tau$).
\end{defn}

  Note that in
  particular if~$\{r_\lambda\}$ is a compatible system, then all but
  finitely many of the~$r_\lambda$ are Fontaine--Laffaille.

If $F$ is CM (in this paper all CM fields are imaginary), we denote
its maximal totally real subfield by~$F^+$. 
If $F$ is CM, and if
$\CM=(M,S_{F^+},\{X-\alpha_v\}, \{ \mu_\lambda \}, \{ w \})$ is a
compatible system of characters of $G_{F^+}$, then we will call
$(\CR,\CM)$ a polarized (and odd) compatible system if for all primes
$\lambda$ of $M$ the pair $(r_\lambda,\mu_\lambda)$ is 
polarized (and odd). (Here $S_{F^+}$ denotes the set of places of $F^+$ lying
below an element of $S$.) We will call $\CR$ polarizable (and odd) if
there exists an $\CM$ such that $(\CR,\CM)$ is a polarized (and odd)
compatible system. Note that $\mu_\lambda(c_v)$ is independent
of~$\lambda$, so oddness of a polarized compatible system can be
checked at a single~$\lambda$.

Recall from~\cite[\S2.1]{BLGGT} that a \emph{polarized automorphic
  representation} of $\GL_n(\A_F)$ is a pair~$(\pi,\chi)$ consisting
of an automorphic representation~$\pi$ of~$\GL_n(\A_F)$, and a 
character~$\chi:\A_{F^+}^\times/(F^+)^\times\to\C^\times$
with~$\chi_v(-1)=(-1)^n$ for all $v|\infty$, such that
$\pi^c\cong\pi^\vee\otimes(\chi\circ\mathbf{N}_{F/F^+}\circ\det)$. We
say that an automorphic representation~$\pi$ of~$\GL_n(\A_F)$ is
\emph{polarizable} if there exists a~$\chi$ such that~$(\pi,\chi)$ is polarized.

If~$(\pi,\chi)$ is a regular algebraic cuspidal polarized
automorphic representation of~$\GL_n(\A_F)$, then there is an associated polarized
and odd
compatible system $(\{r_\lambda(\pi)\},\{\varepsilon^{1-n}r_\lambda(\chi)\})$, as explained in~\cite[\S
5.1]{BLGGT}. (See also ~\cite[Thm.\ 2.1.1]{BLGGT}; note
that~$\varepsilon$ denotes the cyclotomic character,
and~$r_\lambda(\chi)$ is the compatible system associated to~$\chi$,
regarded as an automorphic representation of~$\GL_1(\A_{F^+})$. The
assumption that~$\chi_v(-1)=(-1)^n$ ensures the oddness of this
compatible system.) We say that the pair of compatible systems~$(\{r_\lambda(\pi)\},\{\varepsilon^{1-n}r_\lambda(\chi)\})$ is
automorphic.
 We write~$r_p(\pi)$ for the associated $p$-adic
representation, and we say that a representation
$r:G_F\to\GL_n(\Qpbar)$ is automorphic if it is isomorphic to
$r_p(\pi)$ for some~$\pi$; note that a compatible system is
automorphic if and only if for some prime~$p$, its associated $p$-adic representation is
automorphic. 

The following definition is one of several closely related definitions
that one could make of what it means for a compatible system to be
potentially automorphic; conjecturally, all of these definitions are
equivalent, and are equivalent to automorphy, but this seems to be
very difficult to prove.
\begin{defn}\label{defn: potentially automorphic} If $F$ is CM,
	then we say that a pair of compatible systems~$(\{s_\lambda\},\{\psi_\lambda\})$, with the $s_\lambda$ being $n$-dimensional and
       the $\psi_{\lambda}$ being characters,
       is \emph{potentially automorphic} if for every finite Galois extension~$\Favoid/F$, there is a finite Galois extension of CM
fields $L/F$, which is linearly disjoint from~$\Favoid/F$, and is such that~$(\{s_\lambda|_{G_L}\},\{\psi_\lambda|_{G_L}\})$ is automorphic.  

Similarly, we say that a compatible system~$\{s_{\lambda}\}$
	is \emph{potentially automorphic} if it may be extended
	to a potentially automorphic pair of compatible systems~$(\{s_{\lambda}\}, \{\psi_{\lambda}\}).$
\end{defn}

\begin{defn}\label{defn: pure}
We will call $\CR$ {\em pure} (of weight $w\in\Z$) if
\begin{itemize}
\item for each $v \not\in S$, each root $\alpha$ of $Q_v(X)$ in $\overline{M}$ and each $\imath:\overline{M} \into \C$
we have 
\[ | \imath \alpha |^2 = q_v^w; \]
\item and for each $\tau:F \into \overline{M}$ and each complex conjugation $c$ in $\Gal(\overline{M}/\Q)$ we have
\[ H_{c \tau} = \{ w-h: \,\,\, h \in H_\tau\}. \]
\end{itemize}
\end{defn}

In the following definition, and throughout the body of the paper,
``density'' means ``Dirichlet density''.
\begin{defn}
	\label{defn: weakly irreducible}
We say that a compatible system~$\{s_\lambda\}$ is \emph{weakly irreducible} if there
is a positive density set of rational primes~$l$ so that for all
primes~$\lambda|l$, the representation $s_\lambda$ is 
irreducible. 
\end{defn}

One expects that the irreducibility of a single Galois representation in a
compatible system should imply the irreducibility of all representations, but
this is unknown in general.
On the other hand,  the notion of weak irreducibility turns out to be easy to work with in light of
the following results.

\begin{lem}
  \label{lem: weakly irreducible equals potentially automorphic}Let
  $F$ be CM. Then a
  regular,  odd, polarizable compatible system of
  representations of~$G_F$ is weakly irreducible if and only if it is
  potentially automorphic. 
\end{lem}
\begin{proof}
  Any automorphic compatible system is weakly irreducible
  by~\cite[Thm.\ 1.7]{MR3314824}; it follows immediately that
  potentially automorphic compatible systems are also weakly
  irreducible. Conversely, a weakly irreducible compatible system is
  potentially automorphic by the results of~\cite{BLGGT}; see
  Theorem~\ref{thm: BLGGT splitting at primes} below.
\end{proof}

\begin{lem}
  \label{lem: purity of weakly compatible systems}  Let $F$ be CM, and let~$\{r_\lambda\}$ be a weakly irreducible
  regular,  odd, polarizable compatible system of
  representations of~$G_F$. Then~$\{r_\lambda\}$ is pure in
  the sense of~\emph{\cite[\S 5.1]{BLGGT}}. 
\end{lem}
\begin{proof}
  Since~$\{r_\lambda\}$ is potentially automorphic by Lemma~\ref{lem:
    weakly irreducible equals potentially automorphic}, it is pure by~\cite[Cor.\ 5.4.3]{BLGGT}.
\end{proof}

\begin{lem}
  \label{lem: PT decomposition into weakly irreducible}Let~$F$ be CM, and let~$\{r_\lambda\}$ be a 
  regular,  odd, polarizable  compatible system of
  representations of~$G_F$. Suppose furthermore that~$\{r_\lambda\}$
  is pure. Then we may
  write~$\{r_\lambda\}=\oplus_{i=1}^s\{r_{i,\lambda}\}$, where
  each~$\{r_{i,\lambda}\}$ is a weakly irreducible, regular, odd,
  polarizable compatible system of representations of~$G_F$.
\end{lem}
\begin{proof}
  This is immediate from~\cite[Thm.\
  2.1]{MR3314824} and Lemma~\ref{lem: weakly irreducible equals potentially automorphic}.
\end{proof}

We will occasionally need to make use of compatibility at ramified
places. While we have not built this into our definition of a
compatible system, it follows from potential automorphy, as in the
following result.

\begin{prop}
  \label{prop: weakly irreducible implies compatibility at ramified places}
  Let $F$ be CM, and let~$\{r_\lambda\}$ be a weakly irreducible
  regular,  odd, polarizable compatible system of
  representations of~$G_F$ with field of coefficients~$M$.

Let~$v$ be a finite place of~$F$, and suppose
that~$v\nmid\mathbf N\lambda$ \emph{(}respectively,
that~$v|\mathbf N\lambda$\emph{)}. Then:
\begin{enumerate}
\item For each finite extension $K/F_v$, $r_\lambda|_{G_K}$ is
  unramified \emph{(}respectively, crystalline\emph{)} for some~$\lambda$ if and
  only if it is so for all~$\lambda$.
\item Suppose that~(1) holds. Then there is a representation $r_v$
  of~$I_{K/F_v}$ over~$\overline{M}$ such that for each~$\lambda$,
  $r_\lambda|_{I_{K/F_v}}\cong\ r_v$ \emph{(}respectively, $\WD(r_\lambda|_{G_{F_v}})|_{I_{K/F_v}}\cong\ r_v$\emph{)}.
\end{enumerate}
\end{prop}
\begin{proof}By Lemma~\ref{lem: weakly irreducible equals potentially
    automorphic}, $\{r_\lambda\}$ is potentially automorphic, so it is
  strictly compatible  by~\cite[Cor.\ 5.4.3]{BLGGT}. Strict
  compatibility means by definition that the Weil--Deligne representation
  corresponding to~$r_\lambda|_{G_{F_v}}$ is independent of~$v$, so
  the consequences follow immediately.
\end{proof}

We next establish some results describing how the property of weak irreducibility
of a compatible system behaves under restriction. To begin with,
suppose that~$F$ is a number field and that~$\{r_{\lambda}\}$ is
a compatible system of representations of~$G_F$, and that
the Zariski closure of the image of~$r_\lambda$ is~$G_{\lambda}$.
Let~$G^{\circ}_{\lambda} \subset G_{\lambda}$   denote the connected component of 
the identity. The following is a theorem of Serre. 

\begin{thm}\label{thm: Serre theorem on connected component group}The pre-image
of~$G^{\circ}_{\lambda}$ in~$G_{F}$ is independent of~$\lambda$. 
  \end{thm}
\begin{proof}
  See~\cite[Prop.\ 6.14]{MR1150604}.
\end{proof}

As a corollary, we have the following result, that allows us to ensure
that certain restrictions of weakly irreducible compatible systems remain
weakly irreducible.

\begin{lemma} \label{lem: base change weak irreducibility} Let~$F$ be
  a number field, and let~$\{r_{\lambda}\}$ be a 
 compatible system of~$G_F$-representations. 
 Then there exists a finite extension~$\Favoid/F$ with the following property:
if~$L/F$ is a finite extension 
linearly disjoint from~$\Favoid$, and if~$r = r_{\lambda}$ is any
representation in the compatible system which is irreducible,
then~$r|_{G_L}$ is irreducible. In particular, if~$\{r_{\lambda}\}$ is
weakly irreducible then~$\{r_{\lambda} |_{G_L}\}$
is weakly irreducible. 
\end{lemma}

\begin{proof}
Suppose that~$r = r_{\lambda}$ is irreducible. 
If~$L$ is any finite degree extension of~$F$,
then the Zariski closure of the image of~$r|_{G_L}$ will also  contain~$G^{\circ}_{\lambda}$,
because~$G^{\circ}_{\lambda}$ is connected and hence does not contain any finite index subgroups.

By  
Theorem~\ref{thm: Serre theorem on connected component group}, the pre-image
of~$G^{\circ}_{\lambda}$ in~$G_{F}$ is independent of~$\lambda$.  Let~$\Favoid$ denote the corresponding
fixed field. If $L$ is disjoint from $\Favoid$,
then the component group of~$r |_{G_{L}}$ will be isomorphic to the component group of~$r$,
and hence the Zariski closures of~$r$ and~$r|_{G_{L}}$ will coincide. In particular, $r|_{G_{L}}$ will
be irreducible.
\end{proof}

We also have the following variant of the preceding result, 
which shows that {\em arbitrary} restrictions (to CM extensions))
of weakly irreducible
compatible systems (satisfying the appropriate hypotheses) remain
weakly irreducible, provided that at least one member of the restricted
compatible system is irreducible.

\begin{lemma}
	\label{lem:restricting weakly irr}
  Let $F$ be CM, and let~$\{r_\lambda\}$ be a weakly irreducible
  regular,  odd, polarizable compatible system of
  representations of~$G_F$.
	Let $M$ be a CM extension of $F$.
	If some $r_\mu|_{G_M}$ is irreducible,
	then $\{r_{\lambda}|_{G_M}\}$ {\em (}which is {\em a priori}
	a regular, odd, polarizable compatible system of representations
	of~$G_M${\em )},
	is again weakly irreducible.
\end{lemma}
\begin{proof}
By Lemma~\ref{lem: purity of weakly compatible systems}, the weakly compatible system~$\{r_{\lambda}\}$ 
 is odd, polarized,
regular and pure.
These properties are inherited
by the system ~$\{r_{\lambda}|_{G_M}\}$, which is therefore a direct sum
of weakly irreducible compatible systems by Lemma~\ref{lem: PT
  decomposition into weakly irreducible}. 
  Since $r_\mu|_{G_M}$ is irreducible, there can only be one
compatible system in this direct sum, and $\{r_{\lambda}|_{G_M}\}$ is
weakly irreducible, as required.  
\end{proof}

We close the present discussion of weak irreducibility with the following
result, which establishes
the weak irreducibility of certain tensor products of
compatible systems.

\begin{lem}
  \label{lem: tensor product weak irreducibility}Let~$F$ be CM, and let $\{s_\lambda\}$,
  $\{t_\lambda\}$ be regular, odd, weakly irreducible polarizable compatible systems of
  representations of~$G_F$. Assume that $\{s_\lambda\otimes
  t_\lambda\}$ is regular, and that at least one representation
  $s_\mu\otimes t_\mu$ is  irreducible. Then $\{s_\lambda\otimes
  t_\lambda\}$ is weakly irreducible.
\end{lem}
\begin{proof}
  By 
  Lemma~\ref{lem: purity of weakly compatible systems}
  each of~$\{s_\lambda\}$, $\{t_\lambda\}$ is
   pure, so that
  $\{s_\lambda\otimes t_\lambda\}$ is regular, pure, odd, and
  polarizable. By Lemma~\ref{lem: PT decomposition into weakly
    irreducible} it is therefore a direct sum of weakly irreducible
  compatible systems; since~$s_\mu\otimes t_\mu$ is irreducible, 
  this direct
  sum can only consist of a single compatible system, as required.
\end{proof}

\subsubsection{Deformation rings}\label{subsubsec: deformation
  rings}When we consider deformation rings and automorphy lifting theorems, we can no longer use
algebraically closed coefficient fields. To this end, we adopt the
convention that $\cO$ will denote the ring of integers in a finite
extension~$E/\Qp$ with residue field~$\F$, and that~$E$ will be chosen
large enough such that all representations under consideration are
defined over~$E$ (and all mod~$p$ representations are defined
over~$\F$); as always, the precise choice of~$E$ is unimportant.

As usual, let~$F$ be a CM field with maximal totally real
subfield~$F^+$. Following~\cite{2017arXiv170804885B}, we work in a slightly more general context than~\cite{CHT,BLGGT},
allowing ramification at primes of~$F^+$ which are inert or ramified
in~$F$. This allows us to make cleaner statements, and is also
necessary for some of our arguments with auxiliary primes. 

Fix a prime $p>2$ and a polarized residual representation~$(\sbar,\mubar)$
of $G_{F^+}$ with $\sbar$
absolutely irreducible and (as always in the body of this paper)~$p>2$.
We choose once and for all an element $\gamma_0 \in G_{F_+} \setminus G_F$
(e.g.\ we could choose $\gamma_0$ to be the complex conjugation at
one of the real places of $F^+$),
and we
let~$\rhobar:G_{F^+}\to\cG_n(\F)$ be a fixed 
prolongation of~$(\sbar,\mubar)$,
following the procedure discussed in Subsection~\ref{subsubsec: polarizable repns}.

Let~$\mu:G_{F^+}\to\cO^\times$  be a lift
of~$\mubar$. 
Let $A$ be a
complete local Noetherian~$\cO$-algebra. Then a
\emph{$\mu$-polarized lifting} of~$\rhobar$
to~$A$  is a representation
$\rho:G_{F^+}\to\cG_n(A)$ with~$\nu\circ\rho=\mu$ and
$\rho\pmod{\m_A}=\rhobar$. A~\emph{$\mu$-polarized deformation} of~$\rhobar$
to~$A$ is a~$1+M_n(\m_A)$-conjugacy class of
liftings. 
As in~\cite[Lem.\
1.5]{MR2919688}), $\mu$-polarized liftings and deformations~$\rho$ in this
sense are equivalent to the data of a lifting or deformation~$s$
of~$\sbar$ which satisfies $s^c\cong \mu s^\vee$ (where the equivalence
arises from taking $s=\rho|_{G_F}$).

We also need to consider the corresponding local deformation
problems. We refer to~\cite[\S 3]{2017arXiv170804885B} for the
definitions of deformations of fixed inertial and Hodge types. Let~$v$ be a finite place of~$F$,
let~$\rhobar_v:G_{F^+_v}\to\cG_n(\F)$ be a representation with
multiplier~$\mubar$, and let~$\mu$ be a lift of~$\mubar$. Then a
\emph{$\mu$-polarized lifting} of~$\rhobar_v$
to~$A$ is  a representation
$\rho_v:G_{F^+_v}\to\cG_n(A)$ with~$\nu\circ\rho_v=\mu$ and
$\rho_v\pmod{\m_A}=\rhobar_v$. If we fix an inertial type
$I_{F^+_v}\to\cG_n(E)$, 
then in the case $l\ne p$, we may consider the universal
framed deformation $\cO$-algebra $R^{\square,\tau}$ of inertial type~$\tau$;
this ring is non-zero for
a finite and nonempty set of inertial types~$\tau$. 
We refer to an irreducible component of
any~$R^{\square,\tau}[1/p]$ as simply ``a $\mu$-polarized component
for $\rhobar_v$''; such a component uniquely
determines~$\tau$. By~\cite[Lem.\ 3.4.1]{2017arXiv170804885B}, each
irreducible component is invariant under conjugation, in the sense
that conjugation by elements
of~$\ker(\cG_n(R^{\square,\tau})\to\cG_n(\F))$ preserves each
irreducible component. We will sometimes speak
of polarized components, rather than $\mu$-polarized components, when
the choice of~$\mu$ is clear from the context.

If $v|p$, then in the same way
we let~$R^{\square,\tau,\vv}$ denote the universal
framed deformation $\cO$-algebra of~$\rhobar$ for $\mu$-polarized
potentially semistable lifts of inertial type~$\tau$ and Hodge
type~$\vv$. 
We again refer to an irreducible
component of any~$R^{\square,\tau,\vv}[1/p]$ as simply ``a
$\mu$-polarized component for $\rhobar_v$''. Note that such a component again uniquely
determines~$\tau$ and~$\vv$; we say that a component is \emph{regular}
if $\vv$ is regular (that is, the
labelled Hodge--Tate weights are distinct); we will always assume this
in our main results.

Return now to the global situation of a polarized residual
representation~$(\sbar,\mubar)$ with prolongation~$\rhobar$. Suppose that~$v$ splits in~$F$ as~$w w^c$.
A choice of embedding~$\overline{F^{+}} \rightarrow \overline{F^{+}_v}$  gives rise to a choice of~$w|v$ in~$F$.
With respect to this choice,  the representation~$\rhobar|_{G_{F^{+}_v}}$
 has image
in~$\cG_n^\circ(\F)=\GL_n(\F)\times\GL_1(\F)$,
and the projection to the first factor is the representation~$\sbar|_{G_{F_w}}$. (A difference choice
of embedding~$\overline{F^{+}} \rightarrow \overline{F^{+}_v}$ corresponding to~$w^c | v$ in~$F$
would have the effect of replacing~$\sbar|_{G_{F_w}}$ by~$\mu \otimes \sbar^{\vee}|_{G_{F_w}} \simeq \sbar|_{G_{F_{w^{c}}}}$.)
If~$\rho_v:G_{F^+_v}\to\cG_n(A)$ is a $\mu$-polarized lifting
of~$\rhobar|_{G_{F^+_v}}$,  then the projection to~$\GL_n(A)$ gives rise to a 
lift~$s_w: G_{F_w} = G_{F^{+}_v} \rightarrow \GL_n(A)$  of~$\sbar|_{G_{F_w}}$. 
\begin{lem}
  \label{lem: identifying split local deformation ring with GLn}If~$v$
  splits in~$F$, then the assignment $\rho_v\mapsto s_w$ is an
  equivalence of categories between the $\mu$-polarized liftings
  of~$\rhobar|_{G_{F^+_v}}$ and the liftings of~$\sbar|_{G_{F_w}}$.
\end{lem}
\begin{proof} Under this identification, the representation~$\rho_v$ is simply the 
representation~$\rho_v=(s_w,\mu):G_{F^+_v}\to\cG_n^\circ(A)\subset\cG_n(A)$,
from which the result is clear. (cf. the discussion in~\cite[\S2.3]{CHT}.)
\end{proof}
By Lemma~\ref{lem: identifying split local deformation ring with GLn},
if~$v$ splits in~$F$, then
we can identify $\mu$-polarized components with components of the
corresponding lifting rings for~$\sbar|_{G_{F_w}}$ after choosing a prime~$w|v$
(or after choosing an embedding~$\overline{F^{+}} \rightarrow \overline{F^{+}_v}$ which gives
a canonical choice of~$w|v$).  We will
sometimes do this without comment later in the paper. 

\begin{convention}\label{convention: w lies over v}
	Let~$F/F^{+}$ be a CM extension. Given a prime~$v$ in~$F^{+}$, we choose an 
embedding~$\overline{F^{+}} \rightarrow \overline{F^{+}_v}$ which in turn determines a choice of prime in~$F$ above~$v$
which we denote by~$w$.
\end{convention}

\begin{convention}\label{convention: w lies over v 2}
  If
  $\sbar:G_{F}\to\GL_n(\F)$ is as above, with
  prolongation~$\rhobar:G_{F^+}\to\cG_n(\F)$,
  and if (following Convention~\ref{convention: w lies over v})
  $w$ denotes a prime of $F$ lying over the prime $v$ of $F^{+}$,
  then we will often write
  ``$\mu$-polarized component for~$\sbar|_{G_{F_w}}$'' rather than
  ``$\mu$-polarized component for~$\rhobar|_{G_{F^+_v}}$''. 
  \end{convention}

Given another representation $\rhobar':G_{F_v^+}\to\cG_m(\F)$, and
polarized components $C$, $D$ for $\rhobar$, $\rhobar'$ respectively,
then there is a well-defined component $C\otimes D$
for~$\rhobar\otimes\rhobar'$. Similarly, if $L/F_v^+$ is a finite
extension, there is a well-defined $\mu|_{G_L}$-polarized
component~$C|_L$ for~$\rhobar|_{G_L}$. (In the case that~$v$ is a
split prime, this is~\cite[Lem.\ 3.5.1]{2017arXiv170804885B}, and the
general case is proved in exactly the same way.)

We will frequently make use of the following lemma without further comment.
\begin{lem}
  \label{lem: points of potentially automorphic are robustly
    smooth}Let $F$ be a CM field, and let~$\{r_\lambda\}$ be a regular, odd, polarizable, weakly
  irreducible compatible system of representations of~$G_F$. Then for
  each~$\lambda$ and each finite place~$w$ of~$F$, the representation
  $r_\lambda|_{G_{F_w}}$ lies on a unique component of the
  corresponding deformation ring.
\end{lem}
\begin{proof}
  It suffices to prove that~$r_\lambda|_{G_{F_w}}$ defines a smooth point
  of the corresponding deformation ring. By~\cite[Cor.\
  3.3.5]{2017arXiv170804885B},  it is enough to prove
  that there is a finite extension~$L/F_w$ such that~$r_\lambda|_{G_{L}}$ is pure  in the sense of~\cite[\S
  1]{MR2276777}. 
  Since~$\{r_\lambda\}$ is
  potentially automorphic by Lemma~\ref{lem: weakly irreducible equals
    potentially automorphic}, this follows from 
      the main theorems
  of~\cite{MR2972460,MR3272276} (which show that automorphic
  Galois representations are pure).
\end{proof}

We now return to global deformation problems.

\begin{prop}
  \label{prop: global deformation ring is positive dimensional}
  Let~$F$ be a CM field, and let~$p>2$ be prime. Let~$(\rbar,\mubar)$
  be an absolutely irreducible polarized representation
  of~$G_F$. Suppose that~$\rbar$ is odd, and
  that~$\rbar|_{G_{F(\zeta_p)}}$ is absolutely irreducible, and let~$\mu$ be a de Rham
  lift of~$\mubar$. Let~$S$ be a finite set of finite
  places of~$F^+$ such that~$\rbar$ and~$\mu$ are unramified
  outside~$S$. For each place $v\in S$, let~$C_v$ be a~$\mu$-polarized
  component for~$\rbar|_{G_{F_v}}$, which is regular if~$v|p$.

  Let~$R^\univ$ be the universal deformation ring for $\mu$-polarized
  deformations of~$\rbar$ which are unramified outside~$S$, and lie on
  the component~$C_v$ for each~$v\in S$. Then~$R^\univ$ has Krull
  dimension at least one.
\end{prop}
\begin{proof}This is~\cite[Cor.\ 5.1.1]{2017arXiv170804885B} (note
  that the condition there of being ``discrete series and odd'' is by
  definition the same as being odd in the sense of this paper).
 \end{proof}

\subsubsection{Automorphy lifting and adequate representations}\label{subsubsec: auto
  lifting}
  We end this section by recalling some results concerning automorphy lifting
theorems. Let~$F$ be a CM field with
maximal totally real subfield~$F^+$. We say that a finite place~$v$
of~$F$ is \emph{split} if $v|_{F^+}$ splits  in~$F$. In order to apply
the (potential) automorphy results of~\cite{BLGGT}, we need  
to assume that all of the places
$v|p$, and all of the places at which our Galois representations are
ramified, are split places; we will avoid making such assumptions in
our main results by making base changes.

We have the following theorem; the notion of adequacy is recalled in Definition~\ref{defn:
  adequate} below.
\begin{thm}
  \label{thm: FC's pst R=T}Let $F$ be a CM field, and let $p>2$ be
  prime. Suppose that $p\nmid n$ and that
  $\zeta_p\notin F$. Let $(r,\mu)$ be a polarized automorphic Galois
  representation, where $r:G_F\to\GL_n(\Qpbar)$,
  and assume that $\rbar(G_{F(\zeta_p)})$ is adequate.
  
Let $S$ be a finite set of places of $F^+$  which includes
 all places at which
~$(r,\mu)$ is ramified, and all places dividing $p$,
and for each $v \in S$ let $C_v$ denote the local component at $v$
on which $\rho|_{ G_{F_v}}$ lies, where~$\rho:G_{F^+}\to\cG_n(\Qpbar)$ is the prolongation of~$(r,\mu)$. Assume that every place in~$S$ is a
split place.

Let $R_C$ denote the global deformation $\cO$-algebra for $\rbar$ which
parameterizes deformations of $\rhobar$ that are $\mu$-polarized,
that are unramified outside $S$, and that for each $v \in S$,
lie on the corresponding component $C_v$. 
Then $R_C$ is a finite $\cO$-algebra, and any representation
corresponding to a $\Qpbar$-point of~$R_C$ is automorphic.
\end{thm}
\begin{proof} The automorphy statement is essentially~\cite[Thm.\
  7.1]{MR2869026}, and the finiteness statement follows easily from
  the proof of \emph{loc.\ cit.} (cf.\ \cite[Thm.\
  10.1]{MR2979825}). We only need to justify the slightly weaker
  hypotheses that we are making here, in comparison to
  assumptions~4(c) and~4(d) in the statement of~\cite[Thm.\
  7.1]{MR2869026}. Assumption 4(d) was only made because local-global
  compatibility at places dividing~$p$ was unknown at the time
  that~\cite{MR2869026} was written, but it is now available in the
  required generality thanks to~\cite{MR3205603}. Assumption 4(c) is
  satisfied by~Lemma~\ref{lem: points of potentially automorphic are robustly
    smooth}.
  \end{proof}
As a consequence, we have the following useful finiteness result.
\begin{lemma}
	\label{finiteness lemma}
        Let~$p$ be an odd prime, and let $F$ be a CM field
        with~$\zeta_p\not\in F$. Let
        $\{(s_{\lambda},\mu_\lambda)\}$ be a weakly irreducible, odd, regular,
        polarized compatible system of $n$-dimensional
        representations of $G_F$.
	Assume that $p\nmid n$,  let $(s,\mu)$ be
	the $p$-adic representation coming from the given compatible
        system, with corresponding prolongation~$\rho$, and assume that
        $\sbar(G_{F(\zeta_p)})$ is adequate.
        
Let $S$ be a finite set of finite places of~$F^+$ which contains all
of the places dividing~$p$, and is such that~$\rho$ is
unramified outside~$S$.
For each $v \in S$, let $C_v$ denote the local component at $v$
on which $\rho|_{ G_{F_v}}$ lies.
Let $R_C$ denote the global deformation $\cO$-algebra which
parameterizes deformations of $\rhobar$ that are $\mu$-polarized,
that are unramified outside~$S$, and that, for each $v \in S$,
lie on the corresponding component~$C_v$.
Then $R_C$ is a finite $\cO$-algebra.
\end{lemma}
\begin{proof}
  By Lemma~\ref{lem: weakly irreducible equals potentially
    automorphic}, $\{(s_\lambda,\mu_\lambda)\}$ is 
     potentially automorphic. By~\cite[Lem.\ 1.2.3]{BLGGT}, we can reduce
  to the case that~$\{(s_\lambda,\mu_\lambda)\}$ is in fact
  automorphic, and all the places in~$S$ are split places, in which
  case it follows from Theorem~\ref{thm: FC's pst R=T}
  that~$R_{C}$ is a finite $\cO$-algebra.\end{proof}
  \subsubsection{Adequate representations}

Let~$k$ be a field of characteristic~$p$. We always assume it is
sufficiently large to contain all the eigenvalues of any
representation under consideration.  Let~$V$ be a vector space over~$k$ and
let~$G \subset \GL(V)$ be a group which acts
absolutely irreducibly. We firstly recall from~\cite{MR2979825} what it means
for~$G$ to be adequate.

\begin{defn}
  \label{defn: adequate}$G$ is \emph{adequate} if the following
  conditions hold.
\begin{enumerate}
\item  \label{schur} $H^0(G,\ad^0(V)) = 0$. 
\item $H^1(G,k) = 0$.
\item $H^1(G,\ad^0(V)) = 0$.
\item \label{item: condition C} For every irreducible $G$-submodule $W\subset\ad^0 V$, there is
  an element $g\in G$ with an eigenvalue~$\alpha$ such that
  $\tr e_{g,\alpha}W\ne 0$ (where $e_{g,\alpha}$ is the projection to
  the generalized $\alpha$-eigenspace of~$g$). 
  \end{enumerate}
\end{defn}

%
%
%
%

\begin{lemma} \label{lemma:one} Suppose that~$G$ acts  absolutely irreducibly on~$V$.
Then the following are equivalent.
\begin{enumerate}
\item Condition~(\ref{item: condition C}) of Definition~\ref{defn: adequate}.
\item The set of semisimple elements of~$G$ spans~$\ad(V) \otimes_{k} \kbar$ as
  a~$\kbar$-vector space.
  \end{enumerate}
\end{lemma}

\begin{proof}
This follows from Lemma A.1 of the appendix to~\cite{MR2979825}, namely
the equivalence between~(i) and~(iii).
\end{proof}

\begin{lemma}\label{lem: projective images} Suppose that~$V$ and~$V'$ are absolutely irreducible representations of  a group~$\Gamma$.
Suppose that the projective images of~$\Gamma$ on~$V$ and~$V'$ are
disjoint, that is,  the group~$\Gamma$ surjects onto the product of
the projective images of~$\Gamma$ on~$V$ and~$V'$, 
and denote the projective images by~$\PG$ and~$\PG'$. Then the images of~$\Gamma$ on~$\ad(V)$ and~$\ad(V')$
are~$\PG$ and~$\PG'$ respectively, and the image of~$\Gamma$ on~$\ad(V \otimes V')$
is~$\PG \times \PG'$.
\end{lemma}

\begin{proof}
The fact that~$\Gamma$ acts on~$\ad(V)$ as~$\PG$ is completely formal.
Hence under the assumption that
the projective images~$\PG$
and~$\PG'$ are disjoint, $\Gamma$ acts on~$\ad(V) \oplus \ad(V')$ via~$\PG \times \PG'$. The kernel of the map
$$\GL(\ad(V)) \times \GL(\ad(V')) \rightarrow \GL(\ad(V \otimes
V'))$$
consists precisely of pairs of scalar matrices~$(\lambda,\lambda^{-1})$. But it is not possible
for any~$g \in G$ (or~$\PG$) to act on~$\ad(V)$ as a non-trivial scalar --- this would force the action of~$g$ on~$V$
itself to be scalar and then to be trivial on~$\ad(V)$. 
\end{proof}

\begin{rem}
  The proof of Lemma~\ref{lem: projective images} is just another way
  of saying that the map
$$\PGL(V) \times \PGL(V') \rightarrow  \GL(\ad(V \otimes V'))$$
is injective.
\end{rem}

The following lemma is similar to Lemma~A.2 of the appendix
to~\cite{MR2979825}. 
\begin{lemma}  \label{lemma:three} Suppose that~$V$ and~$V'$ are absolutely irreducible representations of  a group~$\Gamma$.
Suppose that the projective images of~$\Gamma$ on~$V$ and~$V'$ are
disjoint,
 and that the images of~$\Gamma$ on~$V$ and~$V'$ satisfy
condition~(\ref{item: condition C}) of Definition~\ref{defn:
  adequate}. 
  Then the image of~$\Gamma$ on~$V \otimes V'$ satisfies condition~(\ref{item: condition C}) of Definition~\ref{defn: adequate}.
\end{lemma}

\begin{proof}  Let~$G$ and~$G'$ denote the images of~$\Gamma$ in~$V$ and~$V'$
respectively. By  Lemma~\ref{lemma:one}, the semisimple elements~$g$ of~$G$
and~$g'$ of~$G'$ span~$\ad(V)$ and~$\ad(V')$ respectively. In particular,
the elements~$g \otimes g'$, which are also semisimple, span~$\ad(V)
\otimes \ad(V') = \ad(V\otimes V')$. 

Let~$g$ and~$g'$ be any pair of semisimple elements in~$G$ and~$G'$
respectively. 
By Lemma~\ref{lem: projective images}, there is a~$\gamma \in \Gamma$
which acts projectively on~$V$ and~$V'$ by~$g$ and~$g'$ respectively.
Hence it acts on~$V$ and~$V'$ by~$\lambda g$ and~$\lambda' g'$ respectively for scalars~$\lambda$ and~$\lambda'$.
Hence it acts on~$V \otimes V'$ by a scalar multiple of~$g \otimes
g'$. In particular, it spans the same line in~$\ad(V\otimes V')$ as~$g \otimes g'$. Hence these elements
span~$\ad(V\otimes V')$, as required. 
\end{proof}

\begin{lemma} \label{lemma:tensor} Suppose that~$V$ and~$V'$ are absolutely irreducible representations of  a group~$\Gamma$
of dimensions~$n,n' > 2(p+1)$ respectively, whose projective images are disjoint. Then the image of~$\Gamma$
acting on~$V \otimes V'$ is adequate.
\end{lemma}

\begin{proof}By Theorem A.9 of the appendix to~\cite{MR2979825}, the
  images of~$\Gamma$ acting on~$V$ and~$V'$ are both adequate. That condition~(\ref{item: condition C}) of Definition~\ref{defn:
  adequate} holds for the image~$H$ of~$\Gamma$
acting on~$V \otimes V'$ then follows from Lemma~\ref{lemma:three}. 

The adjoint representation of~$H$ has image~$\PG \times \PG'$ by Lemma~\ref{lem: projective images}, so there is a surjective 
map~$H \rightarrow \PG \times \PG'$ whose kernel~$Z$ is central in~$H$ (and acts by scalars on~$V \otimes V'$).
Certainly~$Z$ injects into~$\kbar^{\times}$ and so has order prime to~$p$.
Let~$M$ and~$M'$ be~$\PG$- and~$\PG'$- modules respectively. Since~$Z$ has order prime to~$p$,
inflation--restriction gives
$$H^1(H,M \otimes M') = H^1(\PG \times \PG',M \otimes M').$$
Another application of inflation--restriction gives  an exact sequence
$$H^1(\PG,M) \otimes (M')^{\PG'}
\rightarrow H^1(\PG \times \PG',M \otimes M') 
\rightarrow M^{\PG}\otimes H^1(\PG',M').$$

Letting~$M = M' = k$ or~$M = \ad(V)$ and~$M' = \ad(V')$, we see the two exterior groups vanish
by the adequacy of the images of~$\Gamma$ on~$V$ and~$V'$, and
hence so does the middle group. Absolutely  irreducibility is easy, and the lemma follows.
\end{proof}

\begin{lemma} \label{lemma:irr} Let~$G \subset \GL(V)$, and let~$H \subset G$ be a normal subgroup with~$G/H$
of order prime to~$p$, such that~$H$ is adequate. Then~$G$ is adequate.
\end{lemma}

\begin{proof} If condition~(\ref{item: condition C}) of Definition~\ref{defn:
  adequate} holds for~$H$, it obviously holds for~$G$. So it suffices to check the cohomological
conditions. We have (since~$G/H$ has order prime to~$p$) that
$$H^1(G,M) = H^1(H,M)^{G/H}.$$
Hence if the right hand side vanishes, then so does the left hand
side. Similarly, if a representation of~$G$ is absolutely irreducible after restriction
to~$H$, it is absolutely irreducible.
\end{proof}

\begin{lemma} \label{lem:lookupfool} Let~$A \subset \GL_n(k)$ be absolutely irreducible with~$n > 2(p+1)$.
If~$B \subset \GL_n(\kbar)$ has projective image
containing~$\PSL_n(l)$ for some sufficiently large extension ~$l/k$
(depending on~$A$), then
the image of~$A \otimes B$ is adequate.
\end{lemma}

\begin{proof}
By using Lemma~\ref{lemma:irr}, we may assume that~$B$ has projective image exactly~$\PSL_n(l)$ for some~$l$.
This is a simple group. By taking~$l$ large enough so that the projective image of~$A$ has
order less than that of~$\PSL_n(l)$, we deduce that the projective images of~$A$ and~$B$ have
no nontrivial common quotients. It follows by Goursat's Lemma that the image of~$A \otimes B$ surjects onto
to the product of the projective images of~$A$ and~$B$. We now finish by invoking Lemma~\ref{lemma:tensor}.
\end{proof}

The following is an immediate consequence of Lemma~\ref{lem:lookupfool}.

\begin{lem}
  \label{lem: can always make a tensor product adequate}Suppose that
  $p>2(n+1)$, that~$L$ is a number field, and that
  $\abar:G_L\to\GL_n(\Fpbar)$ is an irreducible representation. Then if $q$
  is a sufficiently
  large power of~$p$ \emph{(}depending on~$\abar$\emph{)}, and
  $\bbar:G_L\to\GL_n(\Fpbar)$ has projective image containing~$\PSL_n(\F_{q})$,
  then $(\abar\otimes\bbar)(G_L)$ is adequate. 
  \end{lem}
  
\section{Globally realizable representations}\label{sec: local automorphic types}
\subsection{Global realizability}

Let $E/\Qp$ be a finite extension with ring of integers~$\cO$ and
residue field~$\F$.
Let~$K/\Q_p$ be a finite extension, and let 
$$\rbar: G_{K} \rightarrow \GL_n(\F)$$
be a representation. 

\begin{df}\label{defn: well spread blanketd} Say the representation~$\rbar$ admits {\bf many
    diagonalizable lifts\rm} if the following holds: for any~$C\ge 0$, $\rbar$
admits a potentially diagonalizable lift with the
property that, for each embedding $\sigma:K\into\Qpbar$, the
$\sigma$-labelled Hodge--Tate weights of the lift  all differ by at least~$C$.
\end{df}
\begin{rem}
  \label{rem: probably everything has a well spread blanket}We expect 
  that \emph{every} representation~$\rbar$
  admits many diagonalizable lifts. In this paper, we will use
  a base change trick (based on Lemmas~\ref{lem:can base change to get
  pd lifts} and~\ref{lem: existence of compatible system from potential
    existence}) to avoid needing to know this.
\end{rem}

\begin{lem}For any representation~$\rbar:G_K\to\GL_n(\F)$, there is a
  finite extension~$L/K$ such that~$\rbar|_{G_L}$ admits many
  diagonalizable lifts. Moreover, any~$L/K$ such that~$\rbar$ and the mod~$p$
  cyclotomic character~$\overline{\varepsilon}$ become trivial over~$G_{L}$  has this property.
  \label{lem:can base change to get pd lifts}
\end{lem}
\begin{proof}Choose~$L$ so that each~$\rbar|_{G_{L}}$ is trivial, and the mod~$p$ cyclotomic character
  of~$G_{L}$ is also trivial. 
  For each integer $C>0$,
  $1\oplus\varepsilon^C\oplus\cdots\oplus\varepsilon^{(n-1)C}$ is a
  potentially diagonalizable (in fact, diagonal) crystalline lift of~$\rbar|_{G_{L}}$, all
  of whose $\sigma$-labelled Hodge--Tate weights
  differ by at least~$C$.  
\end{proof}

\begin{convention}\label{convention: lifts will always be well spread}
  We will frequently consider potentially diagonalizable lifts of
  an~$\rbar$ which admits many diagonalizable lifts. Whenever we do
  so, we will always choose the lifts to have Hodge--Tate weights that
  are sufficiently spread out (in the sense that the condition of Definition~\ref{defn:
    well spread blanketd} holds for some sufficiently large~$C$) that all representations formed in the
  arguments that we make (which will involve tensoring representations
  together) have regular Hodge--Tate weights. In order to streamline
  the paper, we will not make this explicit in any of our arguments.
\end{convention}

\begin{df}[Polarized Local Isomorphisms]\label{defn: polarized local isomorphisms} 
Let~$F$ be a CM field. Suppose that~$\abar: G_{F} \rightarrow \GL_n(\F)$ and~$\bbar: G_{F} \rightarrow \GL_n(\F)$
are absolutely irreducible polarizable representations with respect to a  character~$\mubar$, so (in particular) they
both prolong to representations
$$\rho(\abar), \rho(\bbar): G_{F^{+}} \rightarrow \cG_n(\F),$$
each of which is uniquely determined up to conjugation by an
element in $\cG^0_n(\F).$
Let~$v$ be a prime in~$F^{+}$, and let~$w$ be a prime above~$v$ in~$F$. 
We define a polarized isomorphism~$\abar|_{G_{F_w}} \simeq \bbar|_{G_{F_w}}$ to be an isomorphism of
representations which extends to an isomorphism of polarized representations:
$$\rho(\abar)|_{G_{F^{+}_v}} \simeq \rho(\bbar)|_{G_{F^{+}_v}}.$$
\end{df}

If~$v \in F^{+}$ splits in~$F$, then any isomorphism between~$\abar|_{G_{F_w}}$ and~$\bbar|_{G_{F_w}} $
extends to such an isomorphism, because $G_{F^+_v} = G_{F_w} \subset G_F$,
and
$$\rho(\abar)|_{G_F} = \abar \times \mubar|_{G_F}: 
G_F \to \GL_n(\F) \times \GL_1(\F) = \cG_n^{\circ}(\F) \subset \cG_n(\F)$$
(and similarly for $\rho(\bbar)|_{G_F}$),
so that 
$\rho(\abar)|_{G_{F^+_v}} = \rho(\abar)|_{G_{F_w}} 
= \abar|_{G_{F_w}} \times \mubar|_{G_{F_w}}$
(resp.\ $\rho(\bbar)|_{G_{F^+_v}} = \rho(\bbar)|_{G_{F_w}} 
= \bbar|_{G_{F_w}} \times \mubar|_{G_{F_w}}$).
On the other hand, if~$v$ is inert or ramified in~$F/F^{+}$ and~$\abar|_{G_{F_w}} = \bbar|_{G_{F_w}} $ is reducible,
then this restriction
may admit more than one polarization, and so the requirement that the representations~$\rho(\abar)|_{G_{F^{+}_v}}$ and~$\rho(\bbar)|_{G_{F^{+}_v}}$ 
be isomorphic may be a non-trivial condition.

\begin{defn}
  \label{defn:reasonable}Let~$F$ be a CM field. We say that a representation
  $\sbar:G_F\to\GL_n(\Fpbar)$ is \emph{reasonable} if:
  \begin{itemize}
  \item $\zeta_p\notin F$, and $\sbar|_{G_{F(\zeta_p)}}$ is
    irreducible.
  \item $\sbar$ is polarizable and odd.
  \item $p> 2(n+1)$.
   \end{itemize}
\end{defn}

\begin{defn}
  \label{defn:pleasant} Let~$F$ be a CM field. We say that a representation
  $\sbar:G_F\to\GL_n(\Fpbar)$ is \emph{pleasant} if:
  \begin{itemize}
   \item $\zeta_p\notin F$, and $\sbar|_{G_{F(\zeta_p)}}$ is
    irreducible.
  \item $\sbar$ is polarizable and odd.
  \item $p> 2(n+1)$.
  \item All the primes~$v|p$ in~$F^{+}$ split  in~$F$.
  \item For each place~$w|p$ of~$F$, $\sbar|_{G_{F_w}}$ admits many
    diagonalizable lifts.
   \end{itemize}
\end{defn}

\begin{lem}
  \label{lem:can base change reasonable to pleasant}
 Let~$\sbar$  be a
  reasonable representation of~$G_{F}$. 
  Let~$\Favoid/F$ be a finite extension.
Then there is a finite  
extension~$L/F$ of CM fields, which is
linearly disjoint from~$\Favoid$ over~$F$, such that~$\sbar |_{G_{L}}$ is pleasant.
\end{lem}
\begin{proof}
  Replace~$\Favoid$
  with~$\Favoid\cdot \overline{F}^{\ker\rbar}(\zeta_p)$.
  Let~$E = E^{+}F$ where~$E^{+}/F^{+}$ is any totally real extension  linearly
  disjoint from~$\Favoid/F^+$ with the property 
  that for each place~$w_E|p$ of~$E$, both~$\sbar |_{G_{E_{w_{E}}}}$
  and~$\overline{\varepsilon}|_{G_{E_{w_{E}}}}$ are trivial, where~$\overline{\varepsilon}$ is the mod-$p$
  cyclotomic character. It follows from  Lemma~\ref{lem:can base change
    to get pd lifts} that~$\sbar |_{G_{E_{w_E}}}$ then
  admits many  diagonalizable lifts, and moreover, the same is true
  if one replaces~$E^{+}$ by any further finite extension.
  It now suffices to ensure the primes~$v|p$ in~$E^{+}$ 
  split in~$E$. To achieve this, we cross with a quadratic extension.
  Namely, let~$L^{+} = M^{+} E^{+}$, where~$M^{+}/F^{+}$ is a quadratic extension
  with the property that~$M^{+}_v \simeq F_v$ for~$v|p$ in~$F^{+}$,
 and such that~$L = L^{+} F$ is  linearly
  disjoint from~$\Favoid$ over~$F$.
\end{proof}

\begin{df} \label{def:glob}Let $K/\Qp$ be a finite extension, and
  let~$\rhobar:G_K\to\cG_n(\F)$ be a representation with
  multiplier~$\mubar_K$. Let~$\mu_K$ be a de Rham lift of~$\mubar_K$.  
  A $\mu_K$-polarized component~$C$ for~$\rhobar$ is
{\bf globally realizable\rm} if there exists a CM number field~$F$ and
an odd, regular, polarized, weakly irreducible compatible
system~$(\{s_{\lambda}\},\{\mu_\lambda\})$ over~$F$, with corresponding $p$-adic representation
$(s,\mu)$, with the following properties:
\begin{enumerate}
\item The residual representation~$\sbar$ is reasonable.
\item There exists a prolongation $\xi:G_{F^+} \to \cG_n(\F)$ 
	of $(s,\mu)$,
	and a place~$v$ of~$F^+$, such that~$F^+_v \simeq K$,
  $\mu|_{G_{F^+_v}}=\mu_K$, $\overline{\xi}|_{G_{F^+_v}}\cong\rhobar$, and
  the representation~$\xi | _{G_{F^+_v}}$ lies on~$C$.

\end{enumerate}
We say that a de Rham lift~$\rho:G_K\to\cG_n(\cO)$ of~$\rhobar$ is globally
realizable if it lies on a globally realizable component.
\end{df}

\begin{rem}
  \label{rem: independence of globally realizable from choice of
    prolongation}In Definition~\ref{def:glob}, if the condition holds
  for one prolongation~$\xi$, then it holds for any
  prolongation. Indeed, we saw in Section~\ref{subsec:
  notation} that any two prolongations are conjugate by some element
of $1\times\GL_1\subset\cG_n^\circ$, and the components of the local
deformation ring are invariant under conjugation by~\cite[Lem.\ 3.4.1]{2017arXiv170804885B}.
\end{rem}
\begin{rem}
  \label{rem: globally realizable doesn't depend on the polarization
    in the split case}
If~$v$ splits in~$F$, then by Lemma~\ref{lem: identifying split local
  deformation ring with GLn}, we can identify the $\mu$-polarized
deformation ring for~$\rhobar$ with the lifting ring
for~$\sbar|_{G_{F_w}} :G_{F_w} \to\GL_n(K)$ (which is independent of~$\mu_K$). Then,
in the setting of Definition~\ref{def:glob}, 
it follows from~\cite[Lem.\ 4.1.6]{CHT} that the
condition of a component being globally realizable is independent
of the choice of~$\mu_K$ (as we can twist compatible systems by
algebraic characters). 
\end{rem}
\begin{rem}
  \label{rem: globally realizable implies regular}Note that by
  definition if a component is globally realizable then it is regular.
\end{rem}

\begin{rem}
\label{rem: potentially globally realizable}While it is not obvious
from the definition, as a consequence of our main results we can show
that if~$n=2$ or~$n$ is odd, then any potentially globally realizable
component is globally realizable. More precisely, a component~$C$ for
$\rho:G_K\to\cG_n(\F)$ is globally realizable if and only if there
exists a finite extension~$L/K$ such that~$C|_L$ is globally
realizable. See Corollary~\ref{cor: potentially globally realizable}.
\end{rem}

\begin{rem}
  \label{rem: potentially diagonalizable implies globally
    realizable}Any (regular) potentially diagonalizable representation
  is globally realizable; this is easily proved using the methods
  of~\cite[App.\ A]{MR3134019}, and in particular if~$n=2$ or~$n$ is
  odd, it is a simple
  consequence of Corollary~\ref{cor: potentially globally realizable}
  below (which shows that it is enough to prove this after an
  arbitrary base change), together with~\cite[Lem.\ A.2.5]{BLGGT}
  (which shows how to globalize local representations which are
  induced from characters).
\end{rem}

The notion of being globally realizable can also be formulated in automorphic terms:

\begin{lem} \label{lem: globally realizable, the automorphic version}
A $\mu_K$-polarized component~$C$~for~$\rhobar$ is
 globally realizable if and only if there exists a CM number field~$F$,
 a
regular algebraic cuspidal polarized automorphic representation~$(\pi,\chi)$
of $\GL_n/F$, and a prolongation~$\rho_p(\pi)$ of~$r_p(\pi)$ such that:
\begin{enumerate}
\item There is a prime~$v$ in~$F^+$ such that~$F^+_{v} \simeq K$,
  $\rhobar_p(\pi)|_{G_{F_v}}\cong\rhobar$, $(\varepsilon^{1-n}r_p(\chi))|_{G_K}=\mu_K$,
  and the representation~$\rho_p(\pi) | _{G_{F^+_v}}$ lies on~$C$.
\item The residual representation~$\rbar_p(\pi)$ is reasonable.
\end{enumerate}
\end{lem}

\begin{proof}For the ``if'' direction, note that if these conditions
  are satisfied, then we may take~$(\{s_\lambda\},\{\mu_\lambda\})$ in the definition of
  global realizability to be~$(\{r_\lambda(\pi)\},\{\varepsilon^{1-n}r_\lambda(\chi)\})$. Conversely, if~$C$ is globally realizable, then we apply
 Theorem~\ref{thm: BLGGT splitting at
    primes} below, taking~$\Favoid$ to
  be~$\overline{F}^{\ker\sbar}(\zeta_p)$, and~$S$ to be the set
  of places of~$F$ which lie over~$p$. Then the conditions in the
  Lemma are satisfied by the automorphic representation
  corresponding to the compatible
  system~$(\{s_\lambda|_{G_L}\},\{\mu_\lambda|_{G_{L^+}}\})$.
\end{proof}
\begin{thm}
  \label{thm: BLGGT splitting at primes}Let $(\{s^{(i)}_\lambda\},\{\mu_\lambda^{(i)}\})$, $i=1,\dots,r$ be  compatible systems of odd, regular, weakly irreducible 
   polarized  Galois
  representations over a CM field~$F$. Let $S$ be a finite set of 
  finite places of~$F^+$, and let $\Favoid/F$ be a finite
  extension. Then there is a finite Galois extension $L/F$ of CM
  fields with the properties that:
  \begin{itemize}
  \item $L$ is linearly disjoint from $\Favoid$ over~$F$.
  \item Every place in~$S$ splits completely in~$L^+$.
  \item Each  $(\{s^{(i)}_\lambda|_{G_L}\},\{\mu_\lambda^{(i)}|_{G_{L^+}}\})$ is automorphic.
  \end{itemize}
\end{thm}
\begin{proof} By~\cite[Thm.\ 5.4.1]{BLGGT}, each compatible
  system~$(\{s^{(i)}_\lambda\},\{\mu_\lambda^{(i)}\})$  is potentially automorphic over some finite
  extension~$L_i/F$, which is linearly disjoint from $\Favoid$ over~$F$. (Strictly speaking, that result assumes that all of
  the $s^{(i)}_\lambda$ are  irreducible, but as explained in the
  introduction to~\cite{MR3314824}, all that is actually needed is
  that there is a positive density set of rational primes~$l$ such
  that for each $\lambda|l$, $s^{(i)}_\lambda$ is irreducible.) 
  
  It suffices to show that~$L/F$ can be chosen simultaneously for
  all~$i$, in such a way that all places of
  $F$ above~$S$ split completely. This can be arranged
  by a slight refinement of the arguments of~\cite{BLGGT}; 
  we explain the main idea here,
  referring the interested reader to the proof of~\cite[Prop.\
  A.6]{MR3134019} for a more detailed treatment of a similar
  result. 

As a first step, note that since each~$\{s^{(i)}_{\lambda}\}$ is
potentially automorphic, it follows from~\cite[Lem.\ 1.5, Thm.\
1.7]{MR3314824} that there is a positive density set of rational
primes~$l$ such that if~$\lambda|l$, then each~$\{s^{(i)}_\lambda\}$ is
irreducible. Therefore by~\cite[Prop.\ 5.3.2]{BLGGT}
we can choose~$l$ and~$\lambda|l$ such that each~$s^{(i)}_{\lambda}$ is
Fontaine--Laffaille at all primes dividing~$l$,
each~$s^{(i)}_{\lambda}$ is irreducible, and indeed
$\sbar^{(i)}_\lambda|_{G_{F(\zeta_l)}}$ is irreducible. We also assume
that $l> 2(\max\dim s^i_\lambda+1)$

In the main argument of~\cite{BLGGT}, the field~$L/F$ is constructed by a (finite) number of
applications of the theorem of Moret--Bailly, applied to a particular
moduli space~$T$ over~$F^+$ (see the proof of~\cite[Thm.\ 3.1.2]{BLGGT}). By the version of the theorem of Moret--Bailly given
in~\cite[Prop.\ 3.1.1]{BLGGT}, we can arrange that the places in~$S$ all
split completely in~$L^+$ provided that~$T$ has
$F^+_v$-rational points for all places $v\in S$.
This need not be the case, but we can in any case choose a finite
solvable extension of totally real fields $M^+/F^+$ so that $T$ has $M^+_w$-rational points for
each place~$w$ of~$M^+$ lying over a place in~$S$. We then replace~$T$ by the Weil
restriction~$\Res_{M^+/F^+}T$, and running the arguments of~\cite{BLGGT}, we find a
finite Galois extension $L^+/F^+$, linearly disjoint from $\Favoid$
over~$F^+$, in which all places in~$S$ split
completely, with the property  that if we set~$L=L^+F$ then the
restrictions~$\{s^{(i)}_\lambda|_{G_{LM^+}}\}$ are automorphic. Since the
extension $LM^+/L$ is solvable, it follows
that each~$\{s^{(i)}_\lambda|_{G_{L}}\}$ is automorphic, as required.
\end{proof}

\section{Compatible systems}\label{sec: factoring}Our aim in this section is to prove
results showing that if one representation in a compatible system is a
tensor product, then the compatible system is a tensor product. We do
this under somewhat restrictive hypotheses (see Theorem~\ref{thm:
  factoring compatible systems} below), which we will suffice for the
results of the  the
following section due to some base change tricks and arguments with
auxiliary places. We also prove a number of other results about
compatible systems that we will use in Section~\ref{sec: building
  lifts}. Our results are mostly Lie-theoretic, and in particular we
make crucial use of the results of~\cite{MR1150604}.
\subsection{Component groups}

Recall that by Theorem~\ref{thm: Serre theorem on connected component group}, any
compatible system has a well-defined component group. We have the following technical lemma.
\begin{lemma}\label{lem: control of component groups}
Let~$F$ be a number field, let~$\rbar:G_{F} \rightarrow \GL_n(\Fbar_p)$ be irreducible, and suppose that~$p >\max(n,3)$.
Let~$L/F$ denote  the field~$F(\ker(\rbar))$.
Let~$\{r_{\lambda}\}$ be any compatible system of Galois representations such that~$\rbar_p = \rbar$ and such that~$\det(r_p)$ has infinite image.
Let~$F'/F$ be a finite Galois extension which
is linearly disjoint from~$L/F$.
Then the component group of~$\{r_{\lambda}|_{G_{F'}}\}$ is independent of~$F'/F$.
\end{lemma}

\begin{proof}
  Let~$G$ denote the Zariski closure of~$\im(r_p)$, and let~$G^0$ denote the connected component of~$G$. 
  Let~$\im^\circ(r_p)$ denote the intersection of~$\im(r_p)$ with~$G^0$.
  Let~$F^{\circ}$ denote the fixed field of~$\im^\circ(r_p)$. By Theorem~\ref{thm: Serre theorem on connected component group},
  $F^\circ$ is independent of~$p$, and~$G/G^0 \simeq \Gal(F^\circ/F)$.
  It suffices to show that~$F^{\circ}$ is completely contained in~$L$.
The image~$\im(r_p)$ of~$r_p$ inside~$\GL_{n}(\Zbar_p)$ naturally admits  a surjection onto~$\Gal(F^{\circ}/F)$.
  If~$\Gal(F^\circ/F)$ has order prime to~$p$, then this surjection 
  factors through~$\im(\rbar_p)$, since the kernel
  of~$\im(r_p) \rightarrow \im(\rbar_p)$ is a pro-$p$ group. But this implies 
  that  ~$F^\circ$ is contained in~$L$.
Thus we may assume that~$\Gal(F^\circ/F)$ has order divisible by~$p$,
and hence that the component group~$G/G^{\circ}$ has order divisible by~$p$. 

Note that~$G$ acts
irreducibly because~$\rbar$ is irreducible. By assumption, an element of order~$p$ in~$G/G^{\circ}$ induces an outer automorphism of~$G^{\circ}$ of order~$p$.  If this automorphism is trivial, then, by Schur's lemma,
this automorphism may be modified by an inner automorphism to be a scalar which does not lie in~$G^{\circ}$. The assumption that the determinant has infinite image, however, implies that (since~$G$ is reductive
and irreducible)  the center~$Z$ is infinite, and hence that~$Z^{\circ} = Z$. This  is a contradiction, and hence this order~$p$ element induces a non-trivial outer automorphism of~$G^{\circ}$,
and hence also of its Lie algebra~$\fg$. The automorphism group of any simple Lie algebra has order at most~$3 < p$. Thus this order~$p$ automorphism must act by permuting the simple factors.
Yet~$G$ acts on a space of dimension~$n$, and hence there are at most~$n$ simple factors. Hence we obtain a non-trivial element of order~$p$ in~$S_n$, which is impossible for~$p > n$.
\end{proof}

\subsection{Representation theory}

In this section, we begin by 
proving some basic representation theoretic lemmas for reductive groups~$G$.
All the representations we consider below are assumed to be finite dimensional.

\subsubsection{Reductive linear algebraic groups}
\label{subsubsec:reductive}
Let~$k$ be an algebraically closed field of characteristic zero. 
If $G$ is a connected reductive linear algebraic group over~$k$,
then we let $G^{\der}$ denote the derived subgroup of $G$ --- it is a
connected semisimple linear algebraic group --- and let $Z$ denote the
centre of $G$.  The natural morphism of connected reductive linear
algebraic groups 
$$G^{\der}\times Z^{\circ} \to G$$
(where as usual $Z^{\circ}$ denotes the connected component 
of the identity in $Z$)
is surjective, and its kernel is contained in (the anti-diagonally embedded 
copy of) the intersection $G^{\der} \cap Z^{\circ}$, and thus
is contained in the centre of $G^{\der}$; in particular, it is finite.

We let $\tG^{\der}$ denote the simply connected cover of $G^{\der}$;
it is again a connected semisimple linear algebraic group,
and the kernel of the natural surjection $\tG^{\der} \to G^{\der}$
is finite and central.
We write $\tG := \tG^{\der} \times Z^{\circ}$ (and note that the possible
ambiguity in our use of the notation $\tG^{\der}$ is ameliorated 
by the fact that $\tG^{\der}$ is naturally identified with 
the derived subgroup of $\tG$).  The composite morphism
$$\tG = \tG^{\der} \times Z^{\circ} \to G^{\der} \times Z^{\circ} \to G$$
is 
again surjective, and its kernel is finite and central.

Since $\tG^{\der}$ is semisimple and simply connected,
it may be written as a direct product of almost simple linear algebraic
groups.  Thus $\tG$ is a direct product of such groups and a torus.

If $H$ 
and $J$ are linear algebraic groups, then any 
irreducible representation $W$ of the product $H\times J$ may
be factored (uniquely, up to isomorphism) as a tensor product $W\cong U\otimes_k V,$
where $U$ (resp.\ $V$) is an irreducible representation of $H$
(resp.\ $J$).
Applying this remark in the context of the preceding discussion
(thinking of a representation of $G$ as a representation
of $G^{\der} \times Z^{\circ}$ via inflation),
we find that
the irreducible representations
of $G$ are obtained from the irreducible representations 
$V$ of $G^{\der}$ by choosing a character of $Z^{\circ}$
which coincides with the given action of $G^{\der} \cap Z^{\circ}$ 
on $V$ (Schur's Lemma ensures that this action is indeed given 
by a character)
and extending the $G^{\der}$-action on $V$ to an action
of $G$ (thought of as a quotient
of $G^{\der} \times Z^{\circ}$) via having
$Z^{\circ}$ act through this choice of character.

If $\mathfrak{g}$ denotes the Lie algebra of $G$ (or equivalently
of $\tG$), so that $\mathfrak{g}^{\der}$ is the Lie algebra
of $G^{\der}$ (or equivalently of $\tG^{\der}$), then passing
to the induced $\mathfrak{g}^{\der}$-action induces an equivalence
of categories between the category of finite-dimensional
$\tG^{\der}$-representations over~$k$,
and the category of finite-dimensional $\mathfrak{g}^{\der}$-representations
over~$k$.
In particular, this equivalence induces a bijection between
the isomorphism classes of
irreducible representations of $\tG^{\der}$ 
and the isomorphism classes 
of irreducible representations of $\mathfrak{g}^{\der}$.

The following lemma (and its proof) is a very special case of a theorem of Rajan~\cite{Rajan}.

\begin{lemma} \label{lemma:rajaneasy}
Let~$U$ and~$V$ be two non-trivial representations of a simple Lie
algebra~$\g$ over an algebraically closed field of characteristic zero. Then~$U \otimes V$
is reducible.
\end{lemma}

\begin{proof}Without loss of generality we may assume that~$U$ and~$V$
  are irreducible. Let the highest weights of~$U$ and~$V$ be~$\lambda$ and~$\mu$ respectively. Then,
if~$U \otimes V$ is irreducible, it must be the irreducible representation of highest weight~$\lambda + \mu$.
By the Weyl character formula, this implies that
$$1 = \frac{\dim(U) \dim(V)}{\dim(U \otimes V)}
= \prod_{\Phi^{+}} \frac{ \langle \rho + \lambda, \alpha \rangle  \langle \rho + \mu, \alpha \rangle}
{\langle \rho, \alpha \rangle \langle \rho + \lambda + \mu, \alpha \rangle }.$$
Each individual factor has the form:
$$ 
\frac{\langle \rho,\alpha \rangle^2 + \langle \rho, \alpha \rangle (\langle \lambda, \alpha \rangle
+ \langle \mu, \alpha \rangle) + \langle \lambda, \alpha \rangle \langle \mu,\alpha \rangle}
{\langle \rho,\alpha \rangle^2 + \langle \rho, \alpha \rangle (\langle \lambda, \alpha \rangle
+ \langle \mu, \alpha \rangle) } \ge 1.$$
Since the pairing is non-negative, we obtain a contradiction unless
$\langle \lambda, \alpha \rangle \langle \mu,\alpha \rangle = 0$ for each root~$\alpha \in \Phi^{+}$.  
Because~$\g$ is simple,
there exists a maximal root~$\beta \in \Phi^{+}$ such that, for any dominant weight~$\nu$, one has
$\langle \nu,\beta \rangle \ge \langle \nu, \alpha \rangle$ for any~$\alpha \in \Phi^{+}$. In particular, 
assuming without loss of generality that~$\langle \lambda, \beta \rangle = 0$, we deduce
that~$\langle \lambda, \alpha \rangle = 0$ for all roots in~$\Phi^{+}$, which implies that~$\lambda = 0$ and
the corresponding representation is trivial.
\end{proof}

We say that an irreducible representation~$W$ of
the connected reductive linear algebraic group~$G$ is tensor indecomposable if,
for any 
isomorphism~$W \simeq U \otimes V$ of~$G$-representations, 
either~$U$ or~$V$ is a character.
\begin{lemma} \label{lemma:factor} Let~$G = G^{\circ}$ be a connected
  reductive Lie group over an algebraically closed field of
  characteristic zero, and let $\tG$ denote the finite cover of $G$
  constructed in the preceding discussion.
If~$W$ is an irreducible representation of~$G$ of dimension~$> 1$,
thought of a representation of $\tG$ via inflation, has a  factorization
$$W \simeq \bigotimes V_i$$
as a tensor product of tensor indecomposable representations
of~$\tG$,
where the~$V_i$ are unique up to reordering and twisting and have dimension~$> 1$.
\end{lemma}

\begin{proof}
As noted above, the algebraic group $\tG$
is a direct product
of almost simple linear algebraic groups~$G_i$ and a torus~$T$. Any  irreducible representation of~$G$
is then a tensor product of irreducible representations of the~$G_i$ up to twist by a character of~$T$.
It then suffices to show that the tensor indecomposable representations are precisely the
representations of a single simple factor~$G_i$ up to twist.
This follows immediately from Lemma~\ref{lemma:rajaneasy} (which shows
that an irreducible representation of any~$G_i$ is automatically tensor indecomposable).
\end{proof}

We next establish some
purely representation-theoretic results.

\begin{lemma}  \label{lemma:onedim} Let~$V$ and~$W$ be finite dimensional
	linear representations of a  group~$G$
over an algebraically closed field. 
Suppose that~$V \otimes W$ decomposes as a direct sum of characters.
Suppose, in addition, there are at most three isomorphism classes
of characters which occur in direct sum.
Then~$V$ and~$W$ also admit such a decomposition.
\end{lemma}

\begin{proof}
	First consider the case when both~$V$ and~$W$ are irreducible. Note that
 or any character~$\chi$,
$$\dim \Hom_G(V \otimes W,\chi) = \dim  \Hom_G(V,W^\vee \chi) \le 1,$$
where the latter inequality follows from the irreducibility of~$V$ and~$W$ together with Schur's Lemma.
Since the number of distinct characters is at most three, it follows that~$\dim(V \otimes W) \le 3$, and thus at
least one of~$V$ or~$W$ is a character, and the result follows.

	In the general case, when $V$ and 
	$W$ are not necessarily irreducible,  choose irreducible
	subrepresentations $V'\subset V$ and $W' \subset W$.
	Then $V'\otimes W' \subset V\otimes W$, and so from what we've
	already proved,
	we find that each of $V'$ and $W'$ is necessarily a character.
	We then see that
	$$V = V\otimes W'\otimes(W')^{-1} \subset
	(V\otimes W)\otimes (W')^{-1}$$
	is a direct sum of at most three characters
	(possibly with multiplicities),
	and similarly for 
	$$W = (V')^{-1}\otimes V' \otimes W \subset (V')^{-1}\otimes
	V\otimes W.$$
\end{proof}

\begin{remark}
The preceding result is false when there are four distinct characters.
Indeed, 
one can take~$V = W$ to be the irreducible~$2$-dimensional
representation (over $\Qbar$, say) of  the quaternion group~$Q$ of order~$8$.
\end{remark}

We recall the definition of a strongly irreducible representation.

\begin{df} A representation of a group~$G$ (either a Lie group or a Galois group) is
\emph{strongly irreducible} if it remains irreducible after restriction to any finite index subgroup~$H \subset G$.
\end{df}

\begin{remark}
A representation of a reductive Lie group~$G$ is strongly irreducible if and only if it remains irreducible
after restriction to  the connected component~$G^{\circ} \subset~G$. 
If a topological group~$G$ acts continuously on~$W$ and a subgroup~$H \subset  G$ fixes a  closed subspace~$0 \subsetneq V \subsetneq W$,
then the closure of~$H$ in~$G$ also fixes~$V$. 
Hence continuous representations of a Galois group~$G$
are strongly irreducible if and only if they remain irreducible after restriction to any \emph{closed} finite index subgroup~$H$.

\end{remark}

\begin{lemma} \label{lem: strongly irreducible is uniquely polarizable} Let~$V$
	and~$W$ be irreducible representations of a group $G$ over
	an algebraically closed field $k$ of characteristic zero,
  and suppose that the action of~$G$ on~$V$ is strongly irreducible.
  Then, up to scaling, there is at most one non-trivial bilinear $G$-equivariant pairing:
$$W \times V \rightarrow k(\mu),$$
where~$k(\mu)$ is the twist of the trivial representation
by a character~$\mu$ of~$G$, and~$\mu$ is allowed to range over all characters of~$G$.
\end{lemma}

\begin{proof}
Any such pairing gives rise to an isomorphism~$V \simeq W^{\vee}(\mu)$.
If there existed two such pairings with the same~$\mu$ which were not the same up to scalar, then there would be two corresponding
isomorphisms in~$\Hom_G(V,W^{\vee}(\mu))$. Using either of the identifications of~$V$ with~$W^{\vee}(\mu)$,
we deduce that
$$2 \le \dim \Hom_G(V,W^{\vee}(\mu)) = \dim \Hom_G(V,V),$$
which contradicts the irreducibility of~$V$ by Schur's Lemma.
 If there were two such pairings with \emph{different}~$\mu$, then, denoting one character by~$\mu$ and the other by~$\mu \otimes \chi$
for a non-trivial character~$\chi$, we deduce that~$V \simeq W^{\vee}(\mu)$ and~$V \simeq  W^{\vee}(\mu \otimes \chi)$ and thus~$V \simeq V(\chi)$.
%
%
%
%
Taking determinants of both sides, it
follows that~$\chi^n = 1$. Hence~$\chi$ defines a map:
$G \rightarrow \mu_n \subset k^{\times}$, and in particular the
image of~$\chi$ is finite. Let~$H$ denote the kernel of this map, which is of finite
index. Then we deduce that $\dim \Hom_H(V,V) \ge 2$.
 By Schur's Lemma, this implies
that~$V|_{H}$ is reducible, contradicting our assumptions.
\end{proof}

\subsubsection{Galois representations}

Throughout this subsection we let~$F$ be a (not necessarily CM) number field.

\begin{lemma} \label{lemma:large} 
Let~$r: G_F \rightarrow \GL_m(\Qbar_p)$ be a 
representation with image~$\Gamma:=r(G_F)$, and suppose that the  {\em
(}semisimple\emph{)} residual representation~$\rbar$
has image~$\rbar(G_F)=\overline{\Gamma} \subset \GL_m(\Fbar_p)$
containing~$\SL_m(\F_q)$ for ~$q$ a sufficiently large power of~$p$ \emph{(}we can
take~$q=p$ if~$p> 5$, and~$q=25$ if~$p=5$\emph{)}. 
Then the Zariski closure~$G$ of~$\Gamma$
contains~$\SL_m(\Qbar_p)$.\end{lemma}

\begin{proof}
After conjugation, the image is a closed subgroup of~$\GL_{m}(\cO)$
for~$\cO$ the ring of integers in some finite extension of~$\Q_p$.
By the main result of~\cite{Manoharmayum}, the image therefore contains a conjugate of~$\SL_{m}(W(\F_q))$,
from which the result follows immediately (since~$\SL_m$ is
unirational, by~\cite[Thm.\ 18.2]{MR1102012}). 
\end{proof}

\begin{lemma}  \label{lemma:adjoint} Let~$\{r_{\lambda}\}$ be a weakly compatible system of~$G_F$-representations of
dimension~$m = n^2$. Suppose that~$r = r_p = a \otimes b$ where~$a$ and~$b$
are~$n$-dimensional representations which
have images whose Zariski closure contains~$\SL_n$, and the corresponding Lie algebra
of~$r_p$ contains~$\sl_n \times \sl_n$. 
Then the component group of the compatible system~$\{\ad(r_{\lambda})\}$ is trivial.
\end{lemma}

\begin{proof}
By Theorem~\ref{thm: Serre theorem on connected component group}, the pre-image
of~$G^{\circ}_{\lambda}$ in~$G_{F}$ is independent of~$\lambda$, and so the component group
is independent of~$\lambda$, and thus we may choose~$\lambda = p$.
We have~$\ad(r) = \ad(a) \otimes \ad(b)$. By assumption, the Zariski closures of the images of~$\ad(a)$
and~$\ad(b)$ are both~$\PGL_n$, and hence the Zariski closure of~$\ad(r)$ is~$G = \PGL_n \times \PGL_n$.
Since~$G = G^{\circ}$, the component group is trivial.
\end{proof}

The following result relates irreducibility and strong irreducibility
of Galois representations.

\begin{lemma} \label{lemma:induced} Let~$r: G_F \rightarrow \GL_n(\Qbar_p)$ be an  irreducible
Galois representation
which is Hodge--Tate with regular Hodge--Tate weights. If~$r$ is not strongly
irreducible,  then~$r$ is induced from a strongly irreducible representation over some finite extension~$M/F$.
\end{lemma}

\begin{proof} 
 This is proved in
  the course of the proof of~\cite[Cor.\ 4.4]{MR3186511}; we recall
  the argument.  For any finite extension $E/F$, either $r |_{G_{E}}$
  is irreducible or it decomposes into a sum of \emph{distinct}
  irreducible representations.  This follows immediately from the fact
  that $r|_{G_{E}}$ has distinct Hodge--Tate weights at any prime
  $w|l$.  (Note that $r |_{G_{E}}$ is necessarily semisimple: if $V$
  denotes any irreducible subrepresentation, then the various
  translates of $V$ by elements of $G_F$ are stable under the corresponding
  conjugates of $G_E$, and so we
   see that $r |_{G_{E}}$ becomes completely decomposable under
  restriction to a finite index subgroup, so must already have been
  semisimple.) Suppose then that $r|_{G_E}$ is reducible for some
  finite extension $E/F$. Replacing $E$ by its normal closure over
  $F$, we may assume that the extension $E/F$ is Galois, and the claim that~$r$ is induced
 is immediate from~\cite[Lem.\ 4.3]{MR3186511}. If~$r = \Ind^{G_F}_{G_M} s$, then~$s$
 is also Hodge--Tate with regular Hodge--Tate weights, so by induction on~$n$
 we may assume that~$r$ is induced from a strongly irreducible representation.
\end{proof}

The following lemma will prove useful for lifting Galois representations
along central extensions:

\begin{lemma}
\label{conrad}
Suppose that~$r: G_{F} \rightarrow \GL_{mn}(\Qbar_p)$ is a Galois representation whose image has
Zariski closure inside the image of the map~$\GL_n \times \GL_m \rightarrow \GL_{nm}$. Then there exist Galois representations~$r_A$
and~$r_B$ of dimension~$n$ and~$m$ respectively such that~$r \simeq r_A \otimes r_B$.
\end{lemma}

\begin{proof}  There is a central extension
$$0 \rightarrow Z \rightarrow \GL_n(\Qbar_p) \times \GL_m(\Qbar_p) \rightarrow 
\Gamma(\Qbar_p) \rightarrow 0,$$
where~$ \Gamma$ denotes the image of~$\GL_n \times \GL_m$ in~$\GL_{nm}$, and~$Z$ is the~$\Qbar_p$-points
of a torus embedded anti-diagonally. The result then follows directly from~\cite[Proposition~5.3]{Conrad}.
\end{proof}

\begin{lemma} \label{lemma:prep}
	Fix $n \geq m$,
and let~$r: G_F \rightarrow \GL_{mn}(\Qbar_p)$ be a 
strongly irreducible representation 
such that the Lie algebra of the Zariski closure of the image of~$r$
is isomorphic to 
$\ttt \times \h \times \sl_n $, where:
	\begin{itemize}
\item $\ttt$ is a torus
of rank at most~$1$,
\item $\h$ is semisimple, and
\item the corresponding $mn$-dimensional representation of~$\h \times \sl_n$ is the tensor product of an $m$-dimensional representation of~$\h$ with the standard
representation of~$\sl_n$.
\end{itemize}
 Then either:
\begin{enumerate}
\item $r$ decomposes as a tensor product~$a \otimes b$ where $a$ is of dimension~$m$ and~$b$ is of dimension~$n$, and furthermore 
the Zariski closure of the image of $b$ contains~$\SL_n$; or
\item~$m = n$, and $r$ is the tensor induction of an~$n$-dimensional representation of $G_L$ for some 
	quadratic extension $L$ of~$F$.
\end{enumerate}
Moreover, in case~(1), the representations~$a$ and~$b$ are unique up to twisting by a character, or possibly 
permuting~$a$ and~$b$ if~$m = n$ and~$\h = \sl_n$.
\end{lemma}

\begin{proof}
Let~$G$ denote the Zariski closure of the image of~$r$,
and let~$G^{\circ}$ denote the connected component of $G$.
For now, let us 
 regard $r$ as a (faithful) representation of $G$.
 We shall exhibit a factorization of~$r$ as a tensor product of
 representations of some cover of~$G$ (possibly after passing to  a quadratic extension)
 and then promote this to an actual factorization of Galois
 representations by Lemma~\ref{conrad}.
 
By our assumptions, together with the general discussion
of~(\ref{subsubsec:reductive}),
we may find a finite cover of $G^{\circ}$
by a group of the form~$T \times H \times \SL_n$,
where~$H$ is a product of almost simple groups, and $T$ is a torus. 
If we let~$r^{\circ}$ denote the restriction of $r$ to~$G^{\circ}$,
regarded as a representation of $T \times H\times \SL_n$ by inflation,
then $r^{\circ}$ is irreducible (as $r$ is strongly irreducible by
assumption), 
and so we may write $r^{\circ} = a^{\circ}\otimes b^{\circ}\otimes c^{\circ}$,
where $a^{\circ}$ is an irreducible representation of $H$,
$b^{\circ}$ is the standard $n$-dimensional representation
of $\SL_n$, and $c^{\circ}$ is an irreducible (and hence one-dimensional)
representation of $T$.  (We do not assert that any of~$a^{\circ}$, $b^{\circ}$,
or~$c^{\circ}$ are representations of~$G^{\circ}$.)

After twisting by a character, we may assume (for example
by~\cite[Lem.\ 2.3.15]{2012arXiv1207.6724P})
that the determinant
of $r$ has finite image, and hence that the determinant of $r^{\circ}$
is trivial.  Thus we may in fact assume that $c^{\circ}$,
and hence $\ttt$ and $T$,
are trivial, and we do so from now on.
Thus we assume that $G^{\circ}$ admits a finite cover by $H\times \SL_n$,
and that $r^{\circ}$ admits a corresponding tensor factorization
$a^{\circ}\otimes b^{\circ}$ with $b^{\circ}$ being the standard
representation; we now attempt to extend this tensor factorization 
to a corresponding tensor factorization of $r$.

Since $G^{\circ}$ is normal in $G$, we obtain a conjugation action
of $G$ on $G^{\circ}$, and hence on its universal cover $H\times
\SL_n$.
(Recall that the formation of universal covers is functorial in pointed 
spaces, and note that the conjugation action of $G$ on $G^{\circ}$
acts via automorphisms of the pointed space $(G^{\circ}, 1)$.) 
Suppose firstly that~$\h \ne \sl_n$; then there are no non-trivial
morphisms
$\sl_n\to \h$ (since $\h$ has a faithful representation of dimension $m \leq n$),
so that any automorphism of $\h\times \sl_n$
must fix the~$\sl_n$ factor, and correspondingly any
automorphism of $H\times \SL_n$ must fix the $\SL_n$-factor.

The component group~$G/G^{\circ}$ is then endowed with a homomorphism
\numequation
\label{eqn:outer hom}
G/G^{\circ} \to \operatorname{Out}(\SL_n)
\end{equation}
to the group of
outer automorphisms of $\SL_n$, 
which we claim is trivial.
To see this,
note first that if~$n = 2$, then~$\operatorname{Out}(\SL_2)$ is
trivial, and so we are done. 
If~$n \ge 3$, then the outer automorphism group of~$\SL_n$ is cyclic of order~$2$. If~$G/G^{\circ}$ surjects
onto this outer automorphism group, then the restriction of~$r^\circ$ to
the $\SL_n$-factor, which is a direct sum of $m$ copies of $b^{\circ}$,
is isomorphic to its outer twist, which is a direct sum of $m$ copies
of $(b^{\circ})^{\vee}$.  Since $n \geq 3,$ the standard representation
of $\SL_n$ is not self-dual, and hence this is not possible.  Thus~(\ref{eqn:outer hom}) is trivial, as claimed.

The action of $G$ on the $\SL_n$-factor by conjugation is thus an inner
action, and so induces a homomorphism $G\to \PGL_n$, compatible 
with the given map $\SL_n \to G^{\circ}$. 
Let $K\subseteq G$ denote the kernel of this homomorphism;
the compatibility just remarked upon shows that $K$ contains
the image of $H$ in $G^{\circ}$.
Then we obtain a surjection $K\times \SL_n \to G$
compatible with the given surjection $H\times \SL_n \to~G^{\circ}$.
Correspondingly, we obtain a tensor factorization $r =
a\otimes b$ of~$r$ inflated to  the cover~$K \times \SL_n$ of~$G$
 compatible with
the factorization~$r^{\circ} = a^{\circ} \otimes b^{\circ}$ of~$G^{\circ}$.
This now induces a factorization of Galois representations
by Lemma~\ref{conrad}.
Note that~$b^\circ$ is the standard representation of $\SL_n$, so
the Zariski closure of the image of~$b$ contains~$\SL_n$, as claimed.

We now prove that this factorization is unique. Suppose that~$r \simeq a' \otimes b'  \simeq a \otimes b$. We already have uniqueness
of these representations
over~$G^{\circ}$ by Lemma~\ref{lemma:factor}. Hence it follows that~$\Hom(a,a')$ as a~$G$-representation has a  summand which
becomes trivial when restricted to~$G^{\circ}$,
and hence has a summand on which~$G$ acts through the finite quotient~$G/G^{\circ}$. If this factor is one
dimensional, then~$a$ and~$a'$
are isomorphic up to twist. If this factor has dimension~$> 1$, then, over~$G^{\circ}$, we see that~$\Hom(a,a')|_{G^{\circ}}
= \Hom(a,a)|_{G^{\circ}}$
has at least two trivial factors, which implies that~$a$ is reducible over~$G^{\circ}$, contradicting the strong
irreducibility of~$r$. The same logic applies to~$b$, as required.

Suppose finally that~$\h = \sl_n$. The argument proceeds as above, except now we have to allow
the possibility that~$G/G^{\circ}$  also swaps the factors. Assuming we are in this case, replacing~$G$ by~$G'$ where~$G'$
is the kernel of the map
$G \rightarrow  G/G^{\circ} \rightarrow S_2$, we obtain a tensor factorization
of Galois representations
 as above over some
quadratic extension. But then the image of~$r$ must coincide (up to twist) with the tensor induction of
the corresponding~$n$-dimensional representation (of~$a$ or~$b$) from this quadratic extension.
\end{proof}

\begin{lemma} \label{lem: prep factoring theorem}
Consider~$p$-adic representations~$a$ and~$b$ of~$G_F$ 
of dimensions~$m$ and~$n$ with~$m \le n$. Suppose that 
\begin{enumerate}
\item The representation~$a \otimes b$  is 
irreducible.
\item  The residual representation~$\bbar$ has image
  containing~$\SL_n(\F_q)$ for~$q$ a  sufficiently large  power of~$p$
  \emph{(}in the sense
  of~\emph{Lemma ~\ref{lemma:large})}.
\end{enumerate}
Let~$A$ and~$B$ denote the Zariski closures of the images of~$a$ and~$b$ respectively.
 Let~$A^{\der}$ and~$B^{\der}$ denote the corresponding derived subgroups,
and~$A^{\der,\circ}$ and~$B^{\der,\circ}$ the connected components of these groups.
Let~$G$ denote the Zariski closure of the image of~$a \otimes b$. Then
the corresponding representation of~$G^{\der,\circ}$ is the natural representation
of~$A^{\der,\circ} \times B^{\der,\circ}$ corresponding to the tensor product of the two
natural representations.  This identifies~$G^{\der,\circ}$ with the image 
of~$A^{\der,\circ} \times B^{\der,\circ}$ in the automorphism group
of the
exterior tensor product of the two natural representations.
\end{lemma}

\begin{proof}
By Lemma~\ref{lemma:large}, we have~$B^{\der,\circ} = \SL_n$. 
Certainly the connected subgroup~$G^{\der,\circ} \subset \SL_{n^2}$ lies \emph{inside}
the image~$\Gamma$ of~$A^{\der,\circ} \times B^{\der,\circ}$ under the exterior tensor product.
Since~$\Gamma$ is connected (since it is the image of a connected group under an isogeny), 
it contains no proper finite index subgroups. Thus to prove~$G^{\der,\circ} \hookrightarrow \Gamma$
is an isomorphism, it suffices to prove it is an isogeny, which we can do on the level of Lie algebras.

Denote the Lie algebras of~$A$ and~$B$ by~$\fa \oplus \ttt_{A}$ and~$\fb \oplus \ttt_{B}$
where~$\fa$ and~$\fb$ are semisimple, and~$\ttt_A$ and~$\ttt_B$ are tori of rank at most one.
Note that~$\fb = \sl_n$. 
After twisting, we may assume that~$\ttt_{B}$ is trivial, and that the Lie algebra of~$G$ is~$\g \oplus \ttt_{G}$
where~$\ttt_{G} \simeq \ttt_{A}$. 
Let~$\fd \oplus \ttt_D$ denote
 the Lie algebra (decomposed as a semisimple part~$\fd$ and a torus~$\ttt_D$) of the Zariski closure of the image of~$a \oplus b$.
 There is an inclusion~$\fd \oplus \ttt_D \subset \fa \oplus \fb \oplus \ttt_A$, which induces
  an isomorphism~$\ttt_{D} 
 \simeq \ttt_{A}$. Hence~$\fd \hookrightarrow \fa \oplus \fb$.
  On the other hand, since
$(a \oplus b)^{\otimes 2}$ contains~$a \otimes b$, there is a
surjection~$\fd \onto \g$. 
There is an isomorphism
$$\ad^0(a \otimes b) \simeq \ad^0(a) \oplus \ad^0(b) \oplus \bigl(\ad^0(a) \otimes \ad^0(b)\bigr).$$
The corresponding Lie algebras of the  images of~$\ad^0(a)$ and~$\ad^0(b)$ are~$\fa$ and~$\fb$
respectively, and 
the Lie algebra of the image of~$\ad^0(a \otimes b)$  is~$\g$. Hence there are maps
$$\fa \oplus \fb \supseteq \fd \twoheadrightarrow \g \rightarrow \fa \oplus \fb.$$
The corresponding maps~$\fd \rightarrow \fa$ and~$\fd \rightarrow \fb$
(either coming from the inclusion into~$\fa \oplus \fb$ or via the map to~$\g$)
may be identified with each other, because the semisimple part of the Lie algebra of~$a$
is canonically identified with the Lie algebra of~$\ad^0(a)$. Moreover,
these maps are both surjective, and thus~$\g$ is identified with~$\fd$.
By Goursat's Lemma, the inclusion~$\g \subseteq \fa \oplus \fb$
is the pullback to~$\fa \oplus \fb$ of the graph of an isomorphism
$$\fa/\fn_A \simeq \fb/\fn_B,$$
for some ideals~$\fn_A$ or~$\fn_B$ which may be identified with the kernels of the projections 
from~$\g \rightarrow \fb$ and~$\g \rightarrow \fa$ respectively.
Since~$\fb=\sl_n$ is simple, either both sides are trivial, in which case~$\fd \simeq \g \simeq \fa \oplus \fb$
(and we are done),
or  there is a surjection~$\fa \onto \fb$. 
In this latter case,  by rank considerations (since~$A$ acts faithfully on a space of dimension~$m \le n$),
we deduce that~$\fa \simeq \sl_n$ and that
the map above induces 
an isomorphism of Lie algebras $\fa\simeq\fb$,  and thus~$\g \simeq \fa \simeq \fb$
is diagonally embedded in~$\fa \oplus \fb$. 
This implies that, still on the Lie algebra level, the
representation~$a\otimes b$ must come from 
the tensor product of
an~$n$-dimensional
representation of~$\sl_n$ with a second~$n$-dimensional representation of the same~$\sl_n$.
In either case (standard tensor standard or standard tensor dual), the corresponding representation would be reducible (as also follows
from a special case of Lemma~\ref{lemma:rajaneasy}). This implies (returning to the Lie group level) that, over some finite extension,
$a\otimes b$ is either isomorphic to the direct sum of a~$1$-dimensional representation and an irreducible~$n^2 - 1$-dimensional representation or an irreducible~$\binom{n}{2}$-dimensional representation
and an irreducible~$\binom{n+1}{2}$-dimensional representation
(depending on whether the representations of~$\sl_n$ are the same or dual to each other).
But since~$a\otimes b$ itself is irreducible by hypothesis, it can only decompose over a finite extension into 
irreducible representations of the same dimension. The claim follows. 
\end{proof}

\begin{theorem}\label{thm: first factoring theorem}
Let~$\{r_{\lambda}\}$ be a compatible system of~$G_F$-representations
of dimension~$mn$.
Suppose  there exists a prime~$p$ with~$r = r_{p}$ satisfying the following:
\begin{enumerate}
\item  There exist~$p$-adic representations~$a$ and~$b$ of~$G_F$ 
of dimensions~$m$ and~$n$ with~$m \le n$ such 
that~$r \simeq a \otimes b$.
\item The representation~$a \otimes b$  is 
irreducible.
\item  The residual representation~$\bbar$ has image
  containing~$\SL_n(\F_q)$ for~$q$  a sufficiently large power of~$p$ \emph{(}in the sense
  of~\emph{Lemma~\ref{lemma:large})}.
\end{enumerate}

Suppose firstly that~$(m,n)\ne (2,2)$. Let~$\lambda$ be a place with residue characteristic~$\ell$ such
that~$r_{\lambda}$ is strongly irreducible. Then there
exist
representations~$a_{\lambda}$ and~$b_{\lambda}$,  of dimensions~$m$ and~$n$ respectively, such that
$r_{\lambda} = a_{\lambda} \otimes b_{\lambda}$. Moreover,
the image of~$b_{\lambda}$ has Zariski closure containing~$\SL_{n}(\Qbar_{\ell})$,
and~$a_{\lambda}$ and~$b_{\lambda}$ are unique up to twist by a character and up
to permutation --- the latter being possible only when~$n = m$.

Suppose now that~$m=n=2$. In this case, assume also
that~$\{r_{\lambda}\}$ is odd, regular, polarizable, and 
weakly irreducible. Then  there exists a set of primes~$l$
of density one such that for each~$\lambda|l$, $r_{\lambda}$ is strongly irreducible and
admits  a decomposition~$r_{\lambda} = a_{\lambda}
\otimes b_{\lambda}$ satisfying the conditions in the previous paragraph.
\end{theorem}

\begin{proof} 

Let~$A$ and~$B$ denote the Zariski closures of the images of~$a$ and~$b$ respectively.
By Lemma~\ref{lem: prep factoring theorem},  if~$G$ denotes the Zariski closure of the image of~$a \otimes b$, then we may identify
the corresponding representation of~$G^{\der,\circ}$ with the natural representation
of~$A^{\der,\circ} \times B^{\der,\circ}$ corresponding to the tensor product of the two
natural representations.

Let~$\lambda$ denote a prime for which~$r_{\lambda}$ is strongly irreducible. 
Let~$G_\lambda$ denote the Zariski closure of the image of~$r_\lambda$, and let~$\Gder_{\lambda}$ 
denote the corresponding derived subgroup, and~$\Gdero_\lambda$ the connected component
of this group. By  the strong irreducibility assumption, the corresponding representation  of~$\Gdero_\lambda$ is 
irreducible.
By a Theorem of Serre~\cite[Prop.\ 6.12]{MR1150604}, the formal
character of a compatible system of Galois representations is independent of~$\lambda$.
In particular, the formal characters of the corresponding representations
of~$\Adero \times \Bdero=\Adero\times\SL_n$ and~$\Gdero_\lambda$ coincide.
It is possible for the formal characters of irreducible representations of connected
groups to coincide even when the groups differ (for example, there exist~$27$-
dimensional irreducible representations of~$G_2$ and~$\SL_3$ with the same formal character).
However, what \emph{is} true (\cite[Thm.\ 5.6, Prop.\ 5.7]{MR1150604}) is that
every such equality arises from taking the tensor product of a list of
(explicitly given) basic similarity relations, described explicitly in~\S5.3.1--5.3.4 of~\cite{MR1150604}.
In particular,  note that
that the standard representation of~$\SL_n$ for any~$n>2$ does not admit
\emph{any} basic similarity relations (which  one can also deduce by a
consideration of ranks), while in the case~$n=2$, the only 
similarity relation which intervenes in our situation is that given by the coincidence of the formal
character of the standard representation of~$\Sp_4$ with that of the external tensor product of two
copies of the standard representation of~$\SL_2$.  
Applied to our situation, it follows
that if~$n>2$ (which we assume for the time being, returning to the
case~$n=2$ at the end of the proof) there exists  a connected semisimple group~$\Hdero_\lambda$ 
such that the representation of~$\Gdero_\lambda$ is  the tensor product of
  an irreducible~$n$-dimensional representation of~$\Hdero_\lambda$
  with the standard representation of~$\SL_n$.
Note that the Lie algebra~$\h$ of~$\Hdero_\lambda$ need
not \emph{a priori} be equal to the Lie algebra~$\fa$ of~$\Adero$, but this does not concern us.
The result then follows from Lemma~\ref{lemma:prep},  once we show that
 the representation is
not a tensor induction from a quadratic extension. 

We now prove that this case
can not occur.
If~$r_{\lambda}$ is a tensor induction with Lie algebra containing~$\h \times \sl_n$, 
then~we must have~$\h = \sl_n$ acting via the standard representation. Hence, once more by~\cite[Thm.\ 5.6, Prop.\ 5.7]{MR1150604} (and the fact
that the standard representation of~$\SL_n$ does not admit any basic
similarity relations), we deduce that the Lie algebra of~$a \otimes b$
also contains~$\sl_n \times \sl_n$, and thus that the Zariski closure of~$a$ contains~$\SL_n$.
(To see that the Zariski closure of the image of~$a$ contains~$\SL_n$ rather than a quotient of~$\SL_n$ by some finite group,
we use the tautological fact that~$a$ has a faithful representation in dimension~$n$.)
By Lemma~\ref{lemma:adjoint}, 
we deduce that the component group of~$\ad(r_{\lambda})$ is trivial.
Yet suppose that~$G$ is the Zariski closure of the image of~$r_{\lambda}$, and let~$H$ be the index
two subgroup from which~$r_{\lambda}$ is tensor induced. Then
$$ \ad(\mathrm{Tensor Ind}^{G}_{H} V) \simeq \mathds{1} \oplus \Ind^{G}_{H} \ad^0(V)
\oplus \mathrm{Tensor Ind}^{G}_{H} \ad^0(V).$$
In particular, we see (looking at the second factor) that the image of~$G$ acting on~$\ad(r_\lambda)$ surjects onto~$G/H$,
and so the component group is non-trivial, a contradiction.

We now return to the case~$n=2$, where we have the additional
assumption that that the $4$-dimensional compatible
system~$\{r_{\lambda}\}$ is odd, regular, polarizable, and weakly
irreducible, and thus potentially automorphic by Lemma~\ref{lem:
  weakly irreducible equals potentially automorphic}. By~\cite[Thm.\
2]{XiaThesis}, for a density one set of~$l$, $r_\lambda$ is
irreducible for all~$\lambda|l$. We claim that for any~$\lambda$ for
which~$r_\lambda$ is irreducible, $r_\lambda$ is also strongly
irreducible. If this fails to be the case, then since~$r_\lambda$ is
regular, it follows (as in the proof of~\cite[Cor.~4.4]{MR3186511})
that~$r_{\lambda}$ is induced from a quadratic extension of~$F$, and
hence~$r_{\lambda} \simeq r_{\lambda} \otimes \chi$ for some
non-trivial quadratic character~$\chi$. Since~$\chi$ lives in a
compatible system, we see that every~$r_\lambda$ is induced from a
common quadratic extension, and in particular no~$r_\lambda$ is
strongly irreducible, contradicting our assumptions. In particular,
since~$r$ is assumed irreducible, it is strongly irreducible.

Since~$r_\lambda$ is strongly irreducible for a set of primes~$l$ of
density~$1$, it suffices to show (given the argument above in the case of
general~$m,n$) that the set of primes~$l$ for which there
exists~$\lambda|l$ with~$\Gdero_\lambda$ having Lie algebra~$\sp_4$ is
a set of density zero.  Suppose not; then by Lemma~\ref{added} below, we
deduce that the compatible system~$\{\wedge^2 r_{\lambda}\}$
decomposes as a direct sum of a~$1$-dimensional and a~$5$-dimensional
compatible system.
Since~$\wedge^2 (a \otimes b) = \det(b) \otimes \Sym^2(a) \oplus
\det(a) \otimes \Sym^2(b)$, it follows that at least one
of~$\Sym^2(a)$ and~$\Sym^2(b)$ must have a one-dimensional factor. But
then either~$a$ or~$b$ is induced from a quadratic extension, so~$r$
is induced from a quadratic extension, contradicting the strong
irreducibility of~$r$ which we proved in the previous paragraph.
\end{proof}





\subsection{A lemma on~$4$-dimensional polarizable automorphic compatible systems}

The following lemma was  used in the proof of Theorem~\ref{thm: first factoring theorem}.
Recall that for a Galois representation~$r_{\lambda}$, we denote the Zariski closure of
the image of~$r_{\lambda}$ by~$G_{\lambda}$. 

\begin{lemma} \label{added}
  Let~$F$ be a CM field, and let ~$\{r_{\lambda}\}$  be a
  $4$-dimensional compatible system of representations which is odd, regular, polarizable, and 
  weakly irreducible. Suppose that there exists a set of  primes~$l$
  of positive upper density with the
  property that for some $\lambda|l$ we have~$\Gdero_\lambda = \Sp_4$. 
  Then the compatible system of $6$-dimensional representations~$\{\wedge^2 r_{\lambda}\}$
  decomposes as a direct sum of two 
  compatible systems of dimensions~$5$ and~$1$ respectively.
\end{lemma}

\begin{proof} Let~$T$ be a set of primes~$l$ as in
  the statement of the lemma. If~$l\in T$ then we let~$\lambda$ denote
  the choice of a particular ~$\lambda|l$
  with the property that~$\Gdero_{\lambda} = \Sp_4$. Then the Zariski closure of the image
  of~$r_{\lambda}$ is a subgroup of
  ~$\GSp_4(\overline{M}_{\lambda})$ containing ~$\Sp_4(\overline{M}_{\lambda})$. 
 By~\cite[Prop.~5.3.2]{BLGGT}, after possibly replacing~$T$ with a
 smaller set of primes (still of positive upper density), we also may assume 
that, for any fixed finite extension~$H/F$, 
$$\rbar_{\lambda} |_{G_{H(\zeta_l)}} \rightarrow \GSp_4(\Fbar_l)$$ is
irreducible for $l\in T$.
We apply this with~$H$ equal to the compositum of all the quadratic extensions of~$F$
 unramified outside the set of primes of bad reduction for the
 compatible system.

  Consider the (semisimple) Galois representations:
$$\wedge^2 \rbar_{\lambda} |_{G_{F(\zeta_l)}} \rightarrow
\GL_6(\Fbar_l).$$For each~$l\in T$, the representation
$\wedge^2 \rbar_{\lambda}$ admits a 1-dimensional summand.
We claim that for all but finitely many~$l\in T$, the
complementary~$5$-dimensional
summand is also irreducible. To see this, consider the various possible images of~$\rbar_{\lambda}: G_F
\rightarrow \GSp_4(k)$ with~$k = \Fbar_l$
under the additional assumption that they act irreducibly.
 The classification of such maximal
subgroups  (as first computed in~\cite{SillyRef}) shows that, for~$l > 2$, either:
\begin{enumerate} 
\item \label{case:big} The image contains~$\Sp_4(\F_l)$.
\item \label{case:induced} The image stabilizes a decomposition~$k^4 = k^2 \oplus k^2$, and~$\rbar_{\lambda}$ is thus induced from
a quadratic extension of~$F$.
\item \label{case:cube}  The projective image is contained within the group~$\PGL_2(k)$ acting via the symmetric cube
  representation.
  \item \label{case:exceptional} The projective image has absolutely bounded order.
\end{enumerate}
For a more modern reference, one could also consult~\cite{Holt}, in particular tables~8.12 and~8.13,  which
list the maximal subgroups of~$\Sp_4(q):= \Sp_4(\F_q)$ and from which one
can read off the maximal subgroups of the almost simple extension~$ \PGSp_4(\F_q)$ of~$\Sp_4(\F_q)/Z(\Sp_4(\F_q))$
by~$\langle \delta \rangle = \Z/2\Z$ using the rightmost column.
For the convenience of the reader, we note that groups listed there of type~$\CC_1$ correspond to reducible representations,
 those of type~$\CC_2$, $\CC_3$, and~$\CC_5$ correspond to groups  of type~(\ref{case:induced}), and the remaining groups of type~$\CC_6$ or of
class~$\CSS$ have absolutely bounded image (type~(\ref{case:exceptional})) with the exception of~$\SL_2(\F_q)$ which corresponds to~(\ref{case:cube}).

In case~\eqref{case:big},
 the representation~$\wedge^2\rbar_\lambda$ decomposes as a direct sum of an
 irreducible~$5$-dimensional representation and a~$1$-dimensional
 representation.
 Moreover, because~$\Sp_4(\F_l)$ is quasisimple (a perfect central extension of a simple group),
the projective image of the representation does not change after restriction to
the solvable extension~$F(\zeta_l)$.

We claim that cases~(\ref{case:induced}) and~(\ref{case:exceptional}) can only hold for finitely
many~$l$. In the case~(\ref{case:induced}), it follows, similarly to the proof
of~\cite[Lemma~2.6]{MR3186511}, that, for~$l$ sufficiently large, the
representation is induced from a quadratic extension unramified
at~$l$.  But the quadratic extension must also be unramified outside
the fixed finite set~$S$ of primes of bad reduction for the compatible
system, and so must be contained in~$H$, contradicting our assumption
on the irreducibility of~$\rbar_{\lambda}$ restricted
to~$H(\zeta_l)$. For case~(\ref{case:exceptional}), note that since~$r_\lambda$ is regular, the order of the projective image
of~$\rbar_\lambda$ (even after restriction to inertia at primes
above~$\lambda$) tends to infinity with~$l$, as follows
from Fontaine--Laffaille theory.

Finally, in case~(\ref{case:cube}),
$\wedge^2$ of the symmetric cube representation of a subgroup
of~$\GL_2(k)$ is the direct sum of a character plus a (twist of) the
symmetric fifth power.  If~$l$ is large, the only subgroups
of~$\GL_2(k)$ for which the $4$-dimensional representation is
irreducible but the 5-dimensional representation is not are those whose 
projective image is~$S_4$; since the image of the symmetric cube representation
of such a subgroup then has bounded projective order,
arguing again by Fontaine--Laffaille theory, 
as in  case~(\ref{case:exceptional}),
we see that this case may also be ruled out if~$l$ is sufficiently large.
In conclusion,
we see that if~$l$ is
sufficiently large, then the 5-dimensional summand is irreducible, as required.

Shrinking~$T$ further, we see that we may assume that for all~$l\in T$, the representation~$\wedge^2 r_{\lambda}$
decomposes as the sum of a character and a ~$5$-dimensional
representation~$s_\lambda$ with the properties that~$s_\lambda$ is
Fontaine--Laffaille, and~$\sbar_\lambda|_{G_{F(\zeta_l)}}$ is
irreducible.  The representation~$s_\lambda$ is regular and
essentially conjugate self-dual (since~$r_{\lambda}$ is), and is also
odd (by Lemma~\ref{lem: odd implies odd}). We can certainly assume
that~$l\ge 11$, so it follows that~$s_\lambda$ is
potentially automorphic by~\cite[Theorem~C]{BLGGT}, and hence extends
to the desired compatible system.
\end{proof}

\subsection{Factorization of compatible systems}\label{subsec: main
  factoring result}Our main result in this section is Theorem~\ref{thm: factoring compatible systems}.
  We begin with two preparatory  lemmas.
  The following is a variant on the main results of~\cite{MR2810794,MR2869026}.
\begin{lem}\label{lem: odd dimensional representations are
    odd}
    Let~$F$ be a CM field, and let $r:G_F\to\GL_n(\Qlbar)$ be an irreducible regular polarizable
representation of~$G_F$. Suppose either that~$n$ is odd, or that
$n=2$. If~$n=2$, suppose further that~\mbox{$l\ge 11$}, that~$\Sym^2r$ is irreducible and is
Fontaine--Laffaille at all places dividing~$l$,
and that~$\Sym^2\rbar|_{G_{F(\zeta_l)}}$ is irreducible.
Then~$r$ is odd.
\end{lem}
\begin{proof}In the case that~$n$ is odd, this is immediate from Lemma~\ref{lem: odd implies odd}. 
Suppose now that~$n=2$. Let~$\Pr$ denote the projective representation~$\Pr: G_{F} \rightarrow \PGL_2(\Qbar_p)$ associated to~$r$.
  The polarizability of~$r$ implies that~$\Pr^{c} \simeq \Pr^{\vee} \simeq \Pr$, and hence that~$\Pr$ extends to a representation~$\Pr: G_{F^{+}} \rightarrow \PGL_2(\Qbar_p)$.
  Thus
  $\Ad^0(r) = \Ad^0(\Pr)$ also extends to a representation~$\Ad^0(r):
  G_{F^{+}} \rightarrow \GL_3(\Qbar_p)$. It follows from~\cite[Cor.\ 4.5.2]{BLGGT} (and the
  case~$n=3$ of the result being proved) that~$\Ad^0(r)$ is
  potentially automorphic, 
   and thus (by~\cite[Prop.\ A]{MR2966704}) the image of any complex
   conjugation is non-scalar, and thus the image of any complex
   conjugation under~$\Pr$ is non-scalar.

   We show that this implies that~$r$ is odd. Because we are in dimension~$2$, there is certainly a non-degenerate pairing on~$\Qbar^2_p$ and a character~$\mu$ of~$G_{F^{+}}$ which satisfies
   $$\langle r(\sigma) x, r^c(\sigma) y \rangle = \mu(\sigma)    \langle x,y \rangle$$
   for any complex conjugation~$c$. Since~$\Pr$ is not dihedral, the pairing is unique, and to prove oddness it suffices to show that~$\langle x,y \rangle = \langle y,x \rangle$.
 If not, then the pairing is symplectic, and~$\langle x,x \rangle = 0$
 for all~$x$. This implies that
 $$\langle r(\sigma) x, r^c(\sigma) x \rangle = 0$$
 for all~$\sigma$. Because the dimension is~$2$, we have~$\langle x, y
 \rangle = 0$ for~$x \ne 0$ only when~$y$ is a multiple
 of~$x$. Since~$r$ is irreducible, it follows
 that~$\Pr(\sigma)=\Pr^c(\sigma)$ for all~$\sigma$, which, by Schur's lemma, implies that~$\Pr(c)$ is scalar, contradicting the result above. 
  \end{proof}
  
  The following lemma gives a local condition for a representation not to be of
  the form~$\rho \otimes (\rho^c)^{\vee}$ up to twist for some representation~$\rho$.

\begin{lemma}\label{lem: local condition to avoid tensor induction}
  Let~$F$ be a CM field. Let~$v$ be a prime in~$F^{+}$ which is inert in~$F$, and denote by~$w$ the corresponding prime in~$F$.
 Denote by~$c$ the non-trivial element of~$\Gal(F/F^{+}) =
 \Gal(F_{w}/F^{+}_v)$. Suppose that~$\ell$ is a prime distinct from the characteristic of~$v$ and~$w$.
 Let
$$\psi: G_{F_w} \rightarrow \Qbar^{\times}_{\ell}$$
be a non-trivial ramified character such that~$\psi^c|_{I_{F_w}} =\psi^{-1}|_{I_{F_w}}$. 
Suppose that
$$s_w: G_{F_w} \rightarrow \GL_{n^2}(\Qbar_{\ell})$$
is a representation such that
\[s_w|_{I_{F_w}}\cong \psi|_{I_{F_w}}^{\oplus n}
\oplus \mathds{1}^{\oplus (n^2 - n)}.\]

Then, if~$n > 2$,  we cannot write $s_w$ in the form 
$$s_w\simeq \theta\otimes(\rho \otimes (\rho^{c})^{\vee})$$
where
$\rho: G_{F_w} \rightarrow \GL_n(\Qbar_{\ell})$
and~$\theta:G_{F_w}\to\Qlbartimes$ is a character. 
If additionally~$\psi|_{I_{F_w}}$ has order~$>2$,
then there is no such
decomposition when~$n = 2$ either.
\end{lemma}

\begin{proof} Suppose for the sake of contradiction that we can write
  $s_w\simeq \theta\otimes(\rho \otimes (\rho^{c})^{\vee})$. 
Consider the representations~$\rho$ and~$(\rho^{c})^{\vee}$ restricted
to~$I_{F_w}$. We are assuming for the sake of contradiction that
$$\rho \otimes (\rho^{c})^{\vee} |_{I_{F_w}} \simeq \theta|_{I_{F_w}}^{-1}  \otimes (\psi|_{I_{F_w}}^{\oplus n}
\oplus \mathds{1}^{\oplus (n^2 - n)}).$$
By Lemma~\ref{lemma:onedim}, we deduce that~$\rho$ and~$\rho^c$ restricted to~$I_{F_w}$ are both given by  direct sums of characters.

Suppose that, after restriction to~$I_{F_w}$, the representation~$\rho$ decomposes as the
direct sum of~$n$ copies of a single character~$\phi$. The inertia group~$I_{F_w}$ is normal
in~$G_{F_w}$ and in~$G_{F^{+}_v}$.
 For a representation~$\varrho$ of~$I_{F_w}$ and for~$\sigma \in G_{F^+_v}$, it thus
 makes sense to define a representation~$\varrho^{\sigma}$ of~$I_{F_w}$ by~$\varrho^{\sigma}(g)
= \varrho(\sigma g \sigma^{-1})$. Since~$\rho$ is a representation of~$G_{F_w}$, we see
that~$\rho^{\sigma} \simeq \rho$ for any~$\sigma \in G_{F_w}$.
Hence if~$\sigma\in G_{F_w}$, then
we have
$$ \phi^{\oplus n} \simeq \rho |_{I_{F_w}} \simeq \rho^{\sigma} |_{I_{F_w}} \simeq (\rho|_{I_{F_w}})^\sigma \simeq (\phi^{\sigma})^{\oplus n},$$
and thus~$\phi = \phi^{\sigma}$. It follows that  if~$c \in G_{F^+_v}$ is any lift of the non-trivial
element of~$G_{F^+_v}/G_{F_w} = \Gal(F_{w}/F^{+}_v)$, then~$\phi^{c}$ does not
depend on this lift. 
In particular, still assuming that~$\rho|_{I_{F_w}}$ decomposes as a direct sum of~$n$ copies of~$\phi$, we have
$$\rho^{c} |_{I_{F_w}} \simeq (\phi^c)^{\oplus n}.$$
This would imply that~$\rho \otimes (\rho^c)^{\vee}$ is a direct sum
of~$n^2$ equal characters of~$I_{F_w}$, 
contradicting the fact that~$\psi$ is ramified, and thus non-trivial on inertia. 
In particular, $\rho$ (and by symmetry~$\rho^c$) contains at least~$2$
distinct characters.

Let~$\varphi$ be a character of~$I_{F_w}$ occurring in~$\rho^c$, so that~$\varphi^{-1}$ occurs in~$(\rho^c)^{\vee}$. 
Then
$$\rho |_{I_{F_w}} = (\rho |_{I_{F_w}} \otimes \varphi^{-1}) \otimes \varphi
\subset (\rho \otimes (\rho^c)^{\vee}) |_{I_{F_w}}  \otimes \varphi
= \theta^{-1} \varphi\otimes s_w |_{I_{F_w}}, 
$$
so $\rho |_{I_{F_w}}\subset  \theta^{-1} \varphi\otimes(\psi^{\oplus n} \oplus \mathds{1}^{\oplus (n^2 - n)})$ and it follows that~$\rho$ restricted to~$I_{F_w}$ is a direct sum of at most two distinct characters with some multiplicity.
From the previous discussion, precisely two such characters occur.
The same argument gives the same result for~$\rho^c$. Thus we may write
$$\rho |_{I_{F_w}} = \phi^{r}_{1} \oplus \phi^{n-r}_{2}, \qquad (\rho^c)^{\vee} |_{I_{F_w}} = \varphi^{s}_1 \oplus \varphi^{n-s}_2,$$
for some~$0 < r < n$ and~$0 < s < n$.
Since each pair of characters are distinct, we must have~$\phi_i \varphi_j \ne \phi_i \varphi_k$
and~$\phi_j \varphi_i \ne \phi_k \varphi_i$ when~$j \ne k$.
In order for the tensor product to contain  precisely two distinct characters, this forces the
equalities~$\phi_i \varphi_j = \phi_j \varphi_i$ and~$\phi_i \varphi_i = \phi_j \varphi_j$.
Since one character occurs in~$\rho \otimes (\rho^c)^{\vee}$ exactly~$n$ times (and the other~$n^2 - n$ times), we
deduce, after some appropriate re-ordering, that
$$rs + (n-r)(n-s) = n.$$
This has no solutions in integers~$0 < r < n$ and~$0 < s < n$ when~$n
> 2$, a contradiction. (Indeed, we have $2n=rs+rs+(n-r)(n-s)+(n-r)(n-s)\ge
r+s+(n-r)+(n-s)=2n$ with equality if and only if $r=s=(n-r)=(n-s)=1$.)

Now let us consider the case~$n = 2$.
In this case, the required character identities imply
that~$\phi_i/\phi_j = \varphi_i/\varphi_j$ is a character of order dividing~$2$.
On the other hand, the ratio of the two distinct characters occurring in~$\rho \otimes (\rho^c)^{\vee}|_{I_{F_w}}$
may be identified both with~$\psi$ 
and with~$\phi_i/\phi_j$ with~$i \ne j$, which implies that~$\psi$
has order~$2$, a contradiction.
\end{proof}

Combining the two previous lemmas (as well as lemmas from previous sections),
we are now able to prove a theorem which allows us to factor
a compatible system into the tensor product of two other compatible systems
given such a factorization at one prime~$p$.
  
  \begin{thm} 
  \label{thm: factoring compatible systems}Let~$F$ be a CM field and
  let~$p>2$ be prime. Suppose that $n$ is odd or that
  $n=2$. Let $(\{s_\lambda\},\{\mu_\lambda\})$ be an odd, polarized, regular, weakly
  irreducible compatible system of representations with associated
  $p$-adic representation~$(s,\mu)$. Suppose that we can
  write \[(s,\mu)=(a,\mu_1)\otimes (b,\mu_2)\]where $a$ and $b$ are
  $n$-dimensional representations which are de Rham at all places~
  $v|p$.

Suppose also that:
\begin{itemize}
\item There 
 is a positive density set of places~$l$ such that for
  each~$\lambda|l$, the
  representation~$s_\lambda$ is strongly  irreducible. 
  \item $s$ is strongly irreducible. 
  \item The residual representation~$\bbar$  has image
  containing~$\SL_n(\F_q)$ for~$q$ a sufficiently large power of~$p$ in the sense of {\em Lemma~\ref{lemma:large}}.
\item There is a finite place~$x\nmid p$ of~$F^+$ which is inert
  in~$F$, and a character $\psi:G_{F_x}\to\Qpbar^\times$, such that:
    \begin{itemize}
    \item $\psi^c|_{I_{F_x}} =\psi^{-1}|_{I_{F_x}}$.
  \item $\psi|_{I_{F_x}}$ has \emph{(}finite\emph{)} order greater than~$2$.
  \item $a|_{G_{F_x}}$ is unramified, and
  \item     $b|_{I_{F_x}}\cong\psi|_{I_{F_x}}\oplus\mathds{1}^{\oplus(n-1)}$.
  \end{itemize}
\end{itemize}
 Then there are odd, polarized, regular, weakly irreducible
  compatible systems $(\{a_\lambda\},\{\mu_{1,\lambda}\})$, $(\{b_\lambda\},\{\mu_{2,\lambda}\})$ whose associated
  $p$-adic representations are respectively~$(a,\mu_1)$ and~$(b,\mu_2)$; so in
  particular $\{s_\lambda\}=\{a_\lambda\otimes b_\lambda\}$.
\end{thm}
\begin{proof}
Under our first assumption, we deduce
  from~\cite[Prop.\
  5.3.2]{BLGGT} that there exists a set of primes~$l$ of density~$1$ with $l>2(n+1)$ 
  and~$x\nmid l$, such that for each~$\lambda|l$, $s_\lambda$ is
  Fontaine--Laffaille at all places dividing~$l$, 
  $s_{\lambda}$ is strongly irreducible,
   and~$\sbar_\lambda|_{G_{F(\zeta_l)}}$ is 
  irreducible. 
If~$n=2$, we
  furthermore can and do assume that~$\Sym^2 s_\lambda$ is irreducible
    and Fontaine--Laffaille, and $l\ge 11$. (The irreducibility of~$\Sym^2 s_\lambda$ is
  an easy consequence of the assumption that~$s_\lambda$ is strongly irreducible --- for example,
  one can deduce it from  Lemma~\ref{lem: strongly irreducible is uniquely polarizable}.)

  By Theorem~\ref{thm: first factoring theorem}, we may
  choose~$l$ such that we can write
  $s_\lambda=a_l\otimes b_l$, where~$a_l$ and~$b_l$ are both
  $n$-dimensional representations of~$G_{F}$, the Zariski closure of~$b_l$
  contains~$\SL_n(\Qlbar)$, and the unordered pair~$\{a_l,b_l\}$ is unique up to
  twist. It follows from~\cite[Thm.\ 3.2.10]{2012arXiv1207.6724P} that
  we can
  choose  $a_l$ and~$b_l$ to be unramified at all but
  finitely many places, and de Rham at all places dividing~$l$. (We
  apply the result to the surjection from $\GL_n\times\GL_n$ to its
  image in $\GL_{n^2}$;~\cite[Hyp.\ 3.2.4]{2012arXiv1207.6724P} is
  satisfied because~$F$ is CM and $s_\lambda$ is polarizable.) It then
  follows from~\cite[Prop.\ 3.3.4]{MR3435718} that we can furthermore
  ensure that
  $a_l$ and~$b_l$ are
  in fact crystalline at all places dividing~$l$.
  Moreover, the regularity of~$s_{\lambda}$ immediately implies the regularity of~$a$ and~$b$.
  
  Since $s_\lambda$ is Fontaine--Laffaille
  at all places dividing~$l$
  and $\sbar_\lambda|_{G_{F(\zeta_l)}}$ is 
  irreducible, we see that each of~$a_l$ and~$b_l$ is
  Fontaine--Laffaille at all places dividing~$l$,
  and that~$\abar_l|_{G_{F(\zeta_l)}}$ and
  ~$\bbar_l|_{G_{F(\zeta_l)}}$ are irreducible. Since~$s_\lambda$ is
  strongly irreducible, $a_l$ and~$b_l$ are strongly irreducible.

We now show that~$a_l$ and~$b_l$ are both polarizable. Since $s_\lambda^c\cong \mu_\lambda s_\lambda^\vee$, it
follows that~\[a_l\otimes b_l\cong \mu_\lambda (a_l^c)^\vee\otimes
  (b_l^c)^\vee,\] and by the uniqueness of~$a_l$ and~$b_l$ up to
twist, we see that there are characters
$\psi_l,\varphi_l:G_F\to\Qlbar^\times$ such that either   $a_l\cong
\psi_l(b_l^c)^\vee$, $b_l\cong \varphi_l(a_l^c)^\vee$, or  that $a_l\cong
\psi_l(a_l^c)^\vee$, $b_l\cong \varphi_l(b_l^c)^\vee$. In either case
we have $s_\lambda^c\cong \psi_l\varphi_l s_\lambda^\vee$, and 
 it follows from Lemma~\ref{lem: strongly irreducible is uniquely polarizable}
 that~$\psi_l\varphi_l=\mu_\lambda^c=\mu_\lambda$.

In the first case, we have~$s_\lambda\cong a_l\otimes b_l\cong\varphi_l
a_l\otimes (a_l^c)^{\vee}$. 
This contradicts Lemma~\ref{lem: local condition to avoid tensor
  induction}  (the
hypotheses of which are satisfied by our assumptions on~$a|_{G_{F_x}}$
and~$b|_{G_{F_x}}$ and by Proposition~\ref{prop: weakly
  irreducible implies compatibility at ramified places}). 
  
We are therefore in the second case, and we need to show that~$\psi_l$, $\varphi_l$ both
extend to characters of~$G_{F^+}$. Taking the conjugate dual of the
isomorphism $b_l\cong \varphi_l(b_l^c)^\vee$, we see that also
$b_l\cong \varphi_l^c(b_l^c)^\vee$, so that $b_l\cong
(\varphi_l/\varphi_l^c)b_l$. Since  $b_l$ is strongly
  irreducible,  we  have~$\varphi_l=\varphi_l^c$, and~$\varphi_l$
  extends to~$G_{F^+}$. Similarly, $\psi_l$  also
  extends, as required.

Since $a_l$, $b_l$ are polarizable, it follows from Lemma~\ref{lem: odd dimensional representations are
    odd}  that they are both odd. We can now apply~\cite[Thm.\
  5.5.1]{BLGGT} to the polarized
representations~$(a_l,\psi_l)$ and $(b_l,\varphi_l)$, obtaining compatible
systems $(\{a_\lambda\},\{\psi_\lambda\})$ and~$(\{b_\lambda\},\{\varphi_\lambda\})$ whose
corresponding~$l$-adic representations
are~$(a_l,\psi_l),(b_l,\varphi_l)$ respectively. 
Since~$(s_\lambda,\mu_\lambda)=(a_l,\psi_l)\otimes (b_l,\varphi_l)$,
we have~$(s,\mu)=(a_p,\psi_p)\otimes (b_p,\phi_p)$, where~$(a_p,\psi_p)$ and~$(b_p,\varphi_p)$ are the $p$-adic
representations corresponding to~$(\{a_\lambda\},\{\psi_\lambda\})$ and~$(\{b_\lambda\},\{\varphi_\lambda\})$ 
respectively. 

Since we also have~$(s,\mu)=(a,\mu_1)\otimes (b,\mu_2)$, it follows from
Lemma~\ref{lemma:prep} (since we are assuming that~$s$ is strongly irreducible) 
that the pairs~$\{a_p,b_p\}$ and~$\{a,b\}$
agree up to twist. After possibly exchanging~ $\{a_\lambda\}$
and~$\{b_\lambda\}$, we may suppose that~$b_p$ is a twist
of~$b$. Since~$b$ is strongly irreducible, and both~$b$ and~$b_p$ are
both polarizable and de Rham, the twist must be by an algebraic
character of~$G_{F}$. 
Replacing~$\{b_\lambda\}$
by the twist by the corresponding compatible system of characters,
and~$\{a_\lambda\}$ by the inverse twist, we have constructed the
sought-after compatible systems (note that we must
have~$\mu_2=\varphi_p$ by another application of Lemma~\ref{lem:
  strongly irreducible is uniquely polarizable}, so that also~$\mu_1=\psi_p$).
\end{proof}

\subsection{A technical lemma}We end this section with a technical lemma that will
be used in Section~\ref{subsec: local conditions strong irred}.  We begin
with some equally technical preliminaries.

  \begin{lemma} \label{lemma:somecomponentsone}
  Let~$H$ be a reductive linear algebraic group
  with a faithful irreducible linear representation~$V$,
  such that the restriction of~$V$ to~$H^{\circ}$ remains irreducible. Let~$Z_{H}$ denote the center of~$H$.
  Then there is an injective
  map
  $$H/Z_H H^{\circ} \rightarrow \Out(H^{\circ,\der}).$$
  In particular, if~$Z_H$ is connected, then the component group of~$H$ injects into~$\Out(H^{\circ,\der})$.
  \end{lemma}
  
  \begin{proof}
	  If $Z_{H^{\circ}}$ denotes the centre of $H^{\circ}$,
	  then the natural morphism
	  $H^{\circ,\der} \times Z_{H^{\circ}} \to H^{\circ}$
	  is surjective (since $H^{\circ}$ is a connected reductive
	  linear algebraic group).  Thus we see
	  that the irreducible representation
	  $V$ of $H^{\circ}$ remains irreducible when restricted
	  to $H^{\circ,\der}$; and we also see that the conjugation
	  action of $H^{\circ}$ on $H^{\circ,\der}$ is
	  via inner automorphisms, so that there is indeed a well-defined
	  morphism
	  $H/Z_H H^{\circ} \to \Out(H^{\circ,\der}).$

	  Suppose now that some element $h \in H$ acts
	  on $H^{\circ,\der}$ via an inner automorphism.  We must
	  show that $h \in Z_H H^{\circ}$.  Multiplying
	  $h$ by an element of $H^{\circ,\der}$, we may assume
	  that $h$ actually centralizes $H^{\circ,\der}$. 
	  Since $V$ is an irreducible representation of $H^{\circ,\der}$,
	  we see that $h$ acts on $V$ by scalars, and thus commutes
	  with the action of $H$.  Since $H$ acts faithfully on $V$,
	  we see that $h \in Z_H$, as required.
   \end{proof}
  
   \begin{lemma} \label{lemma:somecomponentstwo}
	   Let $H$ be a connected semisimple algebraic
	   group with a faithful irreducible representation
	   of dimension $\le 2^d$.  
  If~$p > \max(d,3)$,
  then~$p$ does not divide the order of~$\Out(H^{\circ,\der})$.
  \end{lemma}
  
  \begin{proof} Suppose that we have an element
    of~$\Out(H^{\circ,\der})$ of order divisible by~$p$.
	Any such outer automorphism induces a non-trivial outer automorphism
	of the corresponding Lie algebra.
Since the simple Lie algebras only have automorphisms of order at most~$3$, any such
automorphism can 
must consist of a permutation of the simple factors. But if
there is a faithful representation of dimension at most~$2^d$, then
there can be at most~$d$ simple factors, and~$S_d$
has no elements of order~$p$ if~$p > d$.	
  \end{proof}

  The following lemma is the main result of this subsection.
  Recall that if~$M/F$ is a finite extension, we let~$\Mg$ denote the Galois
  closure of~$M$ over~$F$.
  
  \begin{lemma} \label{lemma:noweirdautos} Let~$F$ be a
  number field, 
  and let
$$\abar, \bbar: G_F \rightarrow \GL_n(\Fbar_p)$$ be
such that~$\abar \otimes \bbar$ is irreducible. Let~$\{r_{\lambda}\}$ denote
a weakly compatible system of~$n^2$-dimensional regular representations of~$G_F$ such that~$r_p = a \otimes b$,
where~$a$ is a deformation of~$\abar$, $b$ is a deformation
of~$\bbar$, and~$p > \max(n,3)$. 
Assume that the image of~$\bbar$ contains~$\SL_n(\F_{q})$ for~$q$
sufficiently large, in the sense of {\em Lemma~\ref{lemma:large}}.
 If for some~$\lambda$, $r_{\lambda}$ is induced from an extension~$M/F$,  then~$[\Mg:F]$ has order prime to~$p$.
\end{lemma}

\begin{proof} We may (and so) assume, after tensoring $\{r_{\lambda}\}$ by a compatible system of characters if necessary, 
    that the image of~$\det(r_{\lambda})$
    is infinite. (Note that any twist of~$r_\lambda$ is also induced from~$M/F$.) Let~$G^{\circ}_{\lambda} \subset G_{\lambda}$ denote the connected component
    of the Zariski closure~$G_{\lambda}$ of~$r_{\lambda}$.
    Let~$N/F$ denote the fixed field of the
    inverse image of~$G^{\circ}_{\lambda}$. Since the connected
    component~$G^{\circ}_{\lambda}$ is a normal subgroup of~$G_{\lambda}$, it follows  that~$N/F$ is
    Galois. On the other hand, if~$r_{\lambda}$  is induced from~$M$, then
    certainly~$M \subset N$, and hence~$\Mg \subset N$. 
    Hence it suffices to show that the component group of $G_{\lambda}$
    has order prime to~$p$.
    For this, we work at the prime~$p$ (since the order of the component
    group is independent of~$\lambda$ by Theorem~\ref{thm: Serre theorem on connected component group}).

 To this end,
 we now let~$G$ denote the Zariski closure of~$r_p$ and~$G^{\circ}$ the connected component of the identity in $G$.
 By Lemma~\ref{lem: prep factoring theorem}, we may assume that
 $G^{\der,\circ}$ is the natural representation
of~$A^{\der,\circ} \times B^{\der,\circ}$, where~$A$ and~$B$ are the Zariski
closures of the images of~$a$ and~$b$ respectively, and~$B^{\der,\circ} = \SL_n(\Qbar_p)$.
 Let~$Z$ denote the center of~$G$.
Because~$\det(G)$ is infinite, it follows that the Lie algebra~$\g$ contains a non-trivial
torus~$\ttt$, and hence~$\exp(\ttt) \subset Z$ is infinite.
Because~$G$ acts irreducibly, by Schur's Lemma, it follows that~$Z$ consists of scalars, and thus~$Z = \Qbar^{\times}_p = Z^{\circ}$
is connected.

Because~$r_p$ is regular, it follows 
from Lemma~\ref{lemma:induced} that~$r_p$ is induced from a strongly irreducible representation~$s_p$
over some extension~$M_p/F$.
By comparison with the decomposition of~$G^{\der,\circ}$ above, this strongly irreducible representation
has dimension~$mn$ for some~$m|n$ with~$m [M_p:F] = n$. In particular,
the Galois closure~$\Mg_p/F$ of~$M_p$ over~$F$ has the property that~$\Gal(\Mg_p/F)$
is a subgroup of~$S_n$,  and thus has order prime to~$p$.

Let~$H$ denote the Zariski closure of the image of~$s_p|_{G_{\Mg_p}}$. Since~$[\Mg_p:F]$ has order
prime to~$p$, it is enough to prove that the component group of~$H$
has order prime to~$p$. There is an inclusion~$G^{\circ} \subset H \subset G$,
and~$Z_H = Z = Z^{\circ}$ is connected. 
Moreover,  since~$s_p$ is strongly irreducible, it follows that~$H$   has a faithful irreducible representation of dimension~$mn \le n^2 \le 2^n$
such that the restriction to~$H^{\circ}$  is also irreducible.
The result now follows from Lemmas~\ref{lemma:somecomponentsone}
and~\ref{lemma:somecomponentstwo}.
\end{proof}

\section{Building Lifts}\label{sec: building lifts}
\subsection{Deformations of~\texorpdfstring{$\abar \otimes
    \bbar$}{abar times bbar} and strong irreducibility}\label{subsec:
  local conditions strong irred}
In this
section, we show that strong irreducibility of compatible systems can
be ensured by imposing ramification conditions at finite places. These
ramification conditions become trivial after a finite base change, and
we will use the Khare--Wintenberger method to remove them from our
final results.

We begin with the following preparatory lemma:
\begin{lemma} \label{lemma:freebie} Let~$K/\Q_{\hh}$ be a finite extension of degree prime to~$p$
for some prime~$\hh \equiv 1 \bmod p$.
Suppose that~$\ell \ne \hh$, and let
$$\rho: G_{K} \rightarrow \GL_m(\Qbar_{\ell})$$
be a representation such that~$\rho|_{I_K}$  has finite~$p$-power order.
If~$p \nmid m$, then~$\rho$ admits a 
one dimensional sub-quotient 
which is the restriction of a character of~$G_{\Q_\hh}$.
\end{lemma} 

Note that the lemma applies to both~$\ell = p$ and~$\ell \ne p$.

\begin{proof} 
	Since not every irreducible constituent of $\rho$ can
	be of degree a multiple of $p$,
we may reduce to the case that~$\rho$ is irreducible.
The image of inertia has~$p$-power order and~$\hh \ne p$.
Hence
the representation is  tamely ramified.
 Any such irreducible representation is
of the form~$\rho = \Ind^{G_K}_{G_L} \psi$ for  a character~$\psi$ of~$G_{L}$, 
where~$L/K$ is unramified
(although we don't use this fact)
and the degree~$[L:K] = m$.
 It follows that
 $$n: = [L:\Q_{\hh}] = [L:K][K:\Q_{\hh}]$$
  is prime to~$p$.
Also, since $\rho|_{I_K}$ has $p$-power order, we see that $\psi|_{I_L}$
has $p$-power order.

It is elementary to see
that, since~$p \nmid n$ and~$p|(\hh-1)$, the largest power of~$p$ dividing~$|k^{\times}_L|$ is the same
power of~$p$ which divides~$|\F_{\hh}^\times|=\hh - 1$. (Indeed, $(\hh-1)$ divides
$|k^{\times}_L|$, which in turn divides~$\hh^n - 1$, and we have
$(\hh^n-1)/(\hh-1)\equiv n\pmod{p}$.)  
By local class field theory, it follows that the
character~$\psi$ of~$G_L$ (whose restriction to inertia is, as we have 
observed, of~$p$-power order) 
is the restriction of a character of~$G_{\Q_{\hh}}$, which we also
denote by~$\psi$.
Yet then
$$\rho = \Ind^{G_K}_{G_L} \psi  = \psi \otimes
\Ind^{G_K}_{G_L} 1.$$
The representation $\Ind^{G_K}_{G_L} 1$ contains a copy of the trivial
representation.
Since~$\rho$ is irreducible, it follows that~$L=K$, and~$\rho = \psi$, as claimed.
\end{proof}

Let~$F$ be a CM field. Consider  a pair of irreducible 
representations:
$$\abar, \bbar: G_F \rightarrow \GL_n(\Fbar_p)$$
such that~$\abar \otimes \bbar$ is irreducible and~$\abar$ and~$\bbar$
are polarizable. 
We now build on Lemma~\ref{lemma:noweirdautos}, 
showing how we can control the extensions~$M/F$ from which representations inside a compatible
system containing a lift of~$\abar\otimes\bbar$ could possibly be induced.

We will shortly prove Lemma~\ref{lemma:finiteM}, which enables us to show that, for certain deformation problems, any deformation of~$\abar \otimes \bbar$
which is induced must be induced from one of a finite number of possible fields, independent of
certain classes of auxiliary ramification sets~$\Sigma$.
We start with the following preparatory lemma:

\begin{lemma} \label{lemma:inducedunramified}
Let~$[F_v:\Q_v]$ be a finite extension, and let
$$r: G_{F_v} \rightarrow \GL_n(\Qbar_l)$$
be either potentially unramified if~$l \ne v$ or potentially crystalline if~$l = v$.
Assume that the restriction of the Weil--Deligne representation~$\WD(r)$ to the inertia group~$I_{F_v}$
factors through a group of~$p$-power order.
Suppose that~$r$ is induced from a finite extension~$M_v/F_v$ whose Galois
closure~$\Mgv_v$ over~$F_v$ has  order prime to~$p$.
Then~$M_v/F_v$ is unramified. 
\end{lemma}

\begin{proof}  The formation of Weil--Deligne representations is compatible with inductions.
Hence, if~$r =\Ind_{G_{M_v}}^{G_{F_v}} s$,
then~$\WD(r) = \Ind_{W_{M_v}}^{W_{F_v}} \WD(s)$. It follows that the kernel of~$W_{F_v}$ acting on~$\WD(r)$
is contained in~$W_{M_v}$ and thus also contained in~$W_{\Mgv_v}$. But the image of~$\WD(r)$ restricted to~$I_{F_v}$ 
has~$p$-power order  by assumption, and thus the inertia subgroup of~$\Gal(\Mgv_v/F_v)$ has~$p$-power order.
But~$\Gal(\Mgv_v/F_v)$  has order prime to~$p$ by assumption, and thus~$M_v/F_v$ is unramified.
\end{proof}

\begin{lemma} \label{lemma:finiteM} 
Fix a finite set~$S$ of places of~$F^{+}$ containing all primes
above~$p$ and 
all the primes at which~$\abar$ or~$\bbar$ are ramified.
Let~$\{r_{\lambda}\}$ denote
a weakly irreducible compatible system of odd, polarizable, regular~$n^2$-dimensional representations of~$G_F$ such that~$r_p = c \otimes d$,
where~$c$ is a deformation of~$\abar$, and~$d$ is a deformation of~$\bbar$, and~$p > \max(n,3)$.
Assume that the image of~$\bbar$ contains~$\SL_n(\F_{q})$ for~$q$ some
sufficiently large power of~$p$, in the sense of
{\em Lemma~\ref{lemma:large}}. 

Suppose that~$\{r_{\lambda}\}$
is unramified outside a finite set ~$S \cup \Sigma$  of places of~$F^+$
{\em (}in the sense of {\em Remark~\ref{rem:unramified for compatible systems})}, 
 and that for primes~$v \in \Sigma$, 
 ~$r_{p}(I_{F_v})$ is finite. 
Suppose that~$\lambda$ is a prime such that the 
  representation~$r_{\lambda}$ is  induced from~$M/F$.
  Then:
  \begin{enumerate}
  \item $[\Mg:F]$ has order prime to~$p$, 
  \item $[M:F] \le n^2$, and
  \item $M/F$ is unramified outside~$S$.
  \end{enumerate}
  In particular, if~$S$ is fixed, there are only finitely many such~$M$ independently
  of the choice of~$\lambda$, the compatible system~$\{r_{\lambda}\}$, and the set~$\Sigma$.
  \end{lemma}
  
\begin{proof}
	By shrinking $\Sigma$ if necessary, we may (and do) assume
	that $\Sigma$ is disjoint from~$S$; in particular,
	the representation $\abar\otimes\bbar$ is unramified at
	primes of $\Sigma$.
	In particular, since the primes above $p$ lie in $S$, we may assume
	that $\Sigma$ does not contain any such prime.

	The fact that~$[\Mg:F]$ has order prime to~$p$ is an immediate consequence of Lemma~\ref{lemma:noweirdautos}.
 Let us consider the possible~$M$ from which~$r_{\lambda}$ may be
  induced.  The degree~$[M:F]$ is certainly bounded by~$n^2$.  If we prove that
  $M/F$ is unramified at primes not dividing~$S$, then there are only
  finitely many such~$M/F$ as  an immediate
  consequence of a
  theorem of Hermite (following Minkowski).
Consider the representation~$r_p$ locally at primes~$v \in \Sigma$.
  By assumption, the image~$r_p(I_{F_v})$ of the inertia subgroup~$I_{F_v}$ for~$v \in \Sigma$ factors through a finite quotient.
  On the other hand, the representation~$\abar \otimes \bbar$ is unramified at~$v$.
  Hence the image of inertia factors through a finite quotient
  of~$p$-power order. 
  By Proposition~\ref{prop: weakly irreducible implies compatibility at
  ramified places},  the image of~$I_{F_v}$ factors through a finite
quotient of~$p$-power order for all
representations~$\WD(r_{\lambda}|_{G_{F_v}})$.

Suppose that~$r_{\lambda}$ is induced from~$M$.
Let~$v$ be a finite place of~$F$. Then~$r_{\lambda}|_{G_{F_v}}$ is induced
from~$M_v$. Moreover, the Galois closure~$\Mgv_v$ of~$M_v$ over~$F_v$
is  contained in the localization of the Galois closure~$\Mg$ of~$F$ at~$v$,
and thus~$\Gal(\Mgv_v/F_v)$  also has order prime to~$p$. 
If~$v \notin S$,  then either~$v \in \Sigma$ and the
image of~$I_{F_v}$ in~$\WD(r_{\lambda}|_{G_{F_v}})$ factors through a finite
quotient of~$p$-power order, or~$v \not\in \Sigma$ and the
image of~$I_{F_v}$ in~$\WD(r_{\lambda}|_{G_{F_v}})$ is trivial.
In either case, the claim that~$M_v/F_v$ is unramified follows immediately from Lemma~\ref{lemma:inducedunramified}.
\end{proof}

  Given an integer~$n$ and a finite set~$S$ of places of $F^+$,
  we now introduce certain auxiliary primes~$v$ for  fields~$M/F$
  of the kind which arose in Lemma~\ref{lemma:finiteM}.

\begin{lemma} \label{lemma:existence of primes for strongly irred} For each nontrivial~$M/F$ unramified outside~$S$ with~$\Gal(\Mg/F)$ of order prime to~$p$, either~$M \subset F(\zeta_p)$,
or there exist infinitely many finite places~$v$ of~$F^{+}$ and~$w|v$ in~$F$ with the following properties:
\begin{enumerate}
\item \label{case:1} $N_{F/\Q}(w) = \hh$ is prime, and ~$\hh \equiv 1 \pmod {p^m}$ for some~$p^m > n^2$.
\item \label{case:2} $\hh$ is unramified both in~$M$ and in the fixed fields of~$\abar$ and~$\bbar$ over~$F$.
In particular, $w$ is unramified in~$\Mg/F$.
\item \label{case:3} The images of~$\abar(\Frob_w)$ and~$\bbar(\Frob_w)$ have orders prime to~$p$.
\item \label{case:4} The decomposition group of~$\Mg/F$ at~$w$ is  non-trivial.
\item \label{case:5} $v$ {\em (}resp.\ $w${\em )} lies outside
	any given finite set of places of $F^+$ {\em (}resp.\ $F${\em )}.
\end{enumerate}
\end{lemma}

\begin{proof}
	A choice of~$w$ in~$F$ determines a unique prime~$v$ of~$F^{+}$ with~$w|v$.
Conditions~(\ref{case:2}) and~(\ref{case:5}) exclude only finitely many places. 
(By definition, any prime which ramifies in~$M/\Q$ is either ramified in~$F/\Q$, which is fixed,
or is ramified in~$M/F$, which, by definition, is unramified outside the fixed set~$S$.)
The remaining three conditions are a Chebotarev  condition relative
to the image of~$G_F$ in 
$$\Gal(F(\zeta_{p^m})/F) \times \Gal(\Mg/F) \times \GL_n(\Fbar_p) \times \GL_n(\Fbar_p),$$
where the maps to the two copies of~$\GL_n(\Fbar_p)$ factor via~$\abar$ and~$\bbar$ respectively.
It suffices to show that the image of~$G_{F}$ contains an element whose
projection is trivial in the first group (for~(\ref{case:1})), non-trivial in the second group (for~(\ref{case:4})),  and
semisimple in the last two groups  (for~(\ref{case:3})).
The first two conditions can be satisfied simultaneously unless~$\Mg \subset F(\zeta_{p^m})$.
Since~$[\Mg:F]$ is prime to~$p$,
this is equivalent to~$M \subset \Mg \subset F(\zeta_p)$.
Thus if $M \not\subset F(\zeta_p)$,
then there exists a~$\sigma \in G_{F}$ which is trivial in~$\Gal(F(\zeta_{p^m})/F)$ and non-trivial in~$\Gal(\Mg/F)$. It follows that~$\sigma^{p^k}$ also has this property for any~$k$,  because~$\Gal(\Mg/F)$ has order prime to~$p$. On the other hand, 
the image of any sufficiently large~$p$-power of any element of~$\GL_n(\Fpbar)$
 has order prime to~$p$.
 Thus the image of~$\sigma^{p^k}$
for sufficiently large~$k$ in  the product above has the desired shape.
\end{proof}

\begin{defn} (A suitable choice) \label{defn: suitable}
Continuing on with our running assumptions on~$\abar$ and~$\bbar$, suppose 
in addition that the assumptions of Lemma~\ref{lemma:finiteM} are in effect.
In particular, both~$\abar$ and~$\bbar$ are~$n$-dimensional 
residual
representations of~$G_{F}$  which are polarizable, the representation~$\abar \otimes \bbar$
is absolutely irreducible, and  there exists a weakly irreducible
compatible system~$\{r_{\lambda}\}$ such that~$r_p  = c \otimes d$ is a deformation of~$\abar \otimes \bbar$.
Let~$S$ be  a set containing all primes above~$p$ and all
primes where~$\abar$ and~$\bbar$ are ramified.
Fix characters~$\mu_1,\mu_2$ unramified outside~$S$ 
such that~$\abar$ and~$\bbar$ are~$\mubar_1$ and~$\mubar_2$-polarizable
respectively.
For~$v \in S$, consider any collection of
local $\mu_1$-polarized components~$C_v$ and~ $\mu_2$-polarized
components $D_v$ 
for~$\abar$ and~$\bbar$ respectively.
Suppose, furthermore, that~$C_v \otimes D_v$ is regular for all~$v|p$.
We now choose an auxiliary set~$\Sigma$  disjoint from~$S$
and components~$C_v$, $D_v$ for~$v \in \Sigma$ as
follows. (Ultimately, we shall apply  Lemma~\ref{lemma:finiteM} with precisely this set~$\Sigma$.)
For each of the finitely many fields~$M/F$  not contained in~$F(\zeta_p)$
and satisfying conditions~(1), (2), and~(3) in the conclusion of Lemma~\ref{lemma:finiteM},
we choose \emph{two} primes~$v$ in~$F^{+}$ with~$w|v$ in~$F$  according to Lemma~\ref{lemma:existence of primes for strongly irred}, 
where the auxiliary set of places given by~(\ref{case:5}) is the set of primes which ramify in any of the finitely many fields~$M/F$.  Moreover, for each~$M/F$, we choose this pair of places~$v$ so that they
have
distinct residue characteristics.
We now let~$\Sigma$ be the union of 
these pairs of places~$v$ of~$F^+$ for all~$M/F$.
Our choice guarantees
that, for any prime~$l$, and for any given~$M/F$, there exists a~$v \in \Sigma$ corresponding to~$M/F$ with residue
characteristic prime to~$l$.
We now choose~$C_v$ and~$D_v$ as follows.
Recall that~$F^{+}_v \simeq F_{w} \simeq \Q_{\hh}$.
Let~$\chop: G_{\Q_{\hh}} \rightarrow \Qbar^{\times}_p$ be a character such that~$\chop |_{I_{F_w}}$ has order~$p^m$, which exists
because~$\hh - 1 \equiv 0 \bmod p^m$.
We consider deformations~$c$ and~$d$ at~$v$ such that
$$c | _{I_{F_{w}}} = \bigoplus_{i=0}^{n-1} \chop^i|_{I_{F_w}}, \qquad
d |_{I_{F_w}} = \bigoplus_{i=0}^{n-1} \chop^{ni}|_{I_{F_w}}.$$
Such deformations exist (locally) because~$\abar(\Frob_w)$ and~$\bbar(\Frob_w)$ are semisimple.
Since the characters~$\chop^i|_{I_{F_w}}$ (and similarly the characters~$\chop^{ni}|_{I_{F_w}}$) are pairwise distinct, this
defines  unions of components
 of the corresponding local deformation rings. We choose~$C_v$ and~$D_v$ to be any
 of the resulting components.
 The corresponding inertial type of the $n^2$-dimensional compatible system
 (which is well-defined across the entire compatible system,
 by Proposition~\ref{prop: weakly 
  irreducible implies compatibility at ramified places})
 at each~$v \in \Sigma$ is
 $$\bigoplus_{i=0}^{n^2-1} \chop^i|_{I_{F_w}} = \bigoplus_{i=0}^{n^2-1} \chop^i|_{I_{F^+_v}}.$$
 (Note that~$I_{F_w} = I_{F^+_v}$.)
 In particular, it is a
 direct sum of distinct characters of $p$-power order 
 --- here we use
 the assumption that~$p^m > n^2$. 
As usual, the corresponding polarized
 deformation rings of~$\abar$ and~$\bbar$ are denoted by~$R_{C,\Sigma}$ and~$R_{D,\Sigma}$
 respectively. We refer to the choice of auxiliary set~$\Sigma$
 and corresponding components~$C_v$ and~$D_v$ as a \emph{suitable choice}.
 \end{defn}

Equipped with the notion of a suitable choice, we now prove the following:

 \begin{lemma} 
 \label{lemma:reducetocyclo} Let~$\{r_{\lambda}\}$ denote
a weakly irreducible, polarizable, regular, odd compatible system of~$n^2$-dimensional representations of~$G_F$ such that~$r_p = c \otimes d$,
where~$c$ and~$d$ are deformations of~$\abar$ and~$\bbar$ of types~$R_{C,\Sigma}$
and~$R_{D,\Sigma}$ respectively for a suitable
choice of~$\Sigma$, $C_v$ and~$D_v$ as in {\em Definition~\ref{defn: suitable}}.
Assume that~$p > \max(n,3)$, and 
that the image of~$\bbar$ contains~$\SL_n(\F_q)$ for~$q$ some sufficiently
large power of~$p$, in the sense of {\em Lemma~\ref{lemma:large}}. 

Suppose that some~$r_{\lambda}$ is induced
from~$M/F$.
  Then~$M \subset F(\zeta_p)$. 
   \end{lemma}

 \begin{proof} 
 Assume that~$r_{\lambda}$ is induced from~$M/F$.
 The choice of~$C_v$ and~$D_v$ for~$v \in \Sigma$ ensure that,
 for primes~$v \in \Sigma$, ~$r_{p}(I_{F_v})$ is finite.   Hence,
by Lemma~\ref{lemma:finiteM}, $M/F$ is one of finitely many fields which is unramified outside~$S$, has degree bounded by~$n^2$,
and~$\Mg/F$ has order prime to~$p$.
 Suppose that~$M\not\subset F(\zeta_p)$. By the suitable choice of~$\Sigma$
   as in Definition~\ref{defn: suitable}, there exists a~$v\in\Sigma$ corresponding to~$M$
   which satisfies the conditions of Lemma~\ref{lemma:existence of primes for strongly irred}
   \emph{and} has residue characteristic not dividing~$N(\lambda)$.
   Let us write
 $$r_\lambda = \Ind^{G_F}_{G_M} s: G_{F} \rightarrow \GL_{n^2}(\Qbar_{\ell}).$$
 Since the inertial type of~$r_\lambda$ at~$v$ consists of distinct characters of~$I_{\Q_{\hh}}$,
 it suffices to show that this is incompatible with being an
 induction.

Let ${\ww}$ be a place of~$\Mg$ lying over~$w$ (note that~$F_w = \Q_{\hh}$). 
Since~$\Gal(\Mg/F)$ has order prime to~$p$, it follows that~$[\Mg_{\ww}:F_w] = [\Mg_{\ww}:\Q_{\hh}]$ has order prime to~$p$.
We have a representation
$$s |_{G_{\Mg_{\ww}}}: G_{\Mg_{\ww}} \rightarrow \GL_{m}(\Qbar_{\ell}),$$
where~$m[M:F] = n^2$, and so~$p \nmid m$ (since~$p \nmid n$).
 It follows from Lemma~\ref{lemma:freebie} that~$s|_{\Mg_{\ww}}$ contains at least one 
 subquotient~$\omega$ which is the restriction of a
character of~$G_{\Q_\hh}$ (note that~$s(I_{\Mg_{\ww}})$ has finite
$p$-power order, because~$r_\lambda(I_{F_w})$ does, by the definition
of a suitable choice of~$\Sigma$). By the definition of an induction, there is an identification
$$r_\lambda |_{G_{\Mg}} = \bigoplus_{\sigma\in \Gal(\Mg/F)/\Gal(\Mg/M)} s^{\sigma}|_{G_{\Mg}},$$ 
where 
$s^{\sigma}(g) = s(\sigma g \sigma^{-1})$. 

The decomposition group of~$w$ in~$\Gal(\Mg/F)$ is non-trivial by
the choice of~$v$ and~$w$ (more precisely, by condition~(4) of Lemma~\ref{lemma:existence of primes for strongly irred}). 
Moreover, because~$\Mg$ is the Galois closure of~$M$, the intersection of the conjugates of~$\Gal(\Mg/M)$ inside~$\Gal(\Mg/F)$
is trivial. Hence, for a suitable choice of~$\ww|w$ in~$\Mg$, we may ensure that there exists an element~$\sigma$ in the decomposition
group of~$\ww$ above~$w$ in~$\Gal(\Mg/F)$ that does not lie in~$\Gal(\Mg/M)$.
It follows that~$s |_{G_{\Mg}} \oplus s^{\sigma} |_{G_{\Mg}}$ is a summand of~$r_{\lambda} |_{G_{\Mg}}$.
But~$\sigma$ lies inside the decomposition group of~$\ww$,  and hence~$\sigma \ww  = \ww$,
and~$s^{\sigma} |_{G_{\Mg_{\ww}}}$ is the conjugate by~$\sigma$ of~$s |_{G_{\Mg_{\ww}}}$.
Since~$\omega$ occurs as a subquotient as~$s |_{G_{\Mg_{\ww}}}$, it follows that~$\omega^{\sigma}$ is
a subquotient of~$s^{\sigma} |_{G_{\Mg_{\ww}}}$, and hence~$\omega \oplus \omega^{\sigma}$
is a subquotient of~$r_{\lambda}  |_{G_{\Mg_{\ww}}}$. But~$\omega$ is the restriction
of a character of~$G_{\Q_{\hh}}$, and thus~$\omega^{\sigma} = \omega$, and~$\omega \oplus \omega$
is a subquotient of~$r_{\lambda}  |_{G_{\Mg_{\ww}}}$.
By assumption, the restriction ~$r_{\lambda} |_{I_{F_w}}$  is a direct sum of distinct characters.
Since~$w$ is unramified in~$\Mg/F$ by construction of~$\Sigma$ (cf. Lemma~\ref{lemma:existence of primes for strongly irred}~(\ref{case:2})),
 the restriction~$r_{\lambda}  |_{G_{\Mg_{\ww}}}$
must also be a direct sum of distinct characters, 
and hence we have a contradiction.
\end{proof}

\begin{lemma}\label{lem: strong irreducibility for compatible system 
    given local conditions} 
    In the context of {\em Lemma~\ref{lemma:reducetocyclo}},
assume also that~$\rbar_p|_{G_{F(\zeta_p)}}$ is absolutely irreducible. 
Then~$r_p$ is strongly irreducible, and 
 for a positive density of primes~$l$,  the representations~$r_{\lambda}$
are strongly irreducible for all~$\lambda|l$ \emph{(}that is, the compatible system~$\{r_{\lambda}\}$ is
``weakly strongly irreducible''\emph{)}.
\end{lemma}

\begin{proof}
By Lemma~\ref{lemma:induced}, the representations~$r_\lambda$  which
are absolutely irreducible are strongly irreducible
unless they are induced. 
Since ~$\{r_{\lambda}\}$ is
weakly irreducible by assumption, for a positive density of primes~$l$,  the representations~$r_{\lambda}$
 are  irreducible for all~$\lambda|l$, and certainly~$r_p$ is irreducible.
Hence it suffices to show 
that  the representations~$r_{\lambda}$ and~$r_p$ are not induced from some finite extension~$M/F$.
By the previous lemma (Lemma~\ref{lemma:reducetocyclo}), 
 this can happen only if~$M \subset F(\zeta_p)$. Thus we will be done if we can show that~$r_p|_{G_{F(\zeta_p)}}$ is absolutely
irreducible, and that $\{r_\lambda|_{G_{F(\zeta_p)}}\}$ is weakly irreducible.

Since~$\rbar_p|_{G_{F(\zeta_p)}}$ is absolutely irreducible by
assumption, $r_p|_{G_{F(\zeta_p)}}$ is absolutely irreducible.
Since~$F(\zeta_p)/F$ is CM,
it follows from Lemma~\ref{lem:restricting weakly irr}
that~$\{r_{\lambda}|_{G_F(\zeta_p)}\}$ is weakly irreducible, as required. 
\end{proof}

We end this subsection with some results (versions of the Khare--Wintenberger
argument) that will allow us to remove the auxiliary conditions
discussed above from our final results. We remind the reader of the
conventions introduced in Definition~\ref{defn: polarized local isomorphisms} and Conventions~\ref{convention: w
  lies over v} and~\ref{convention: w lies over v 2},
which will be in force throughout the rest of this
section; namely, we write~$w$ (possibly decorated by subscripts and superscripts) for
a place of a CM field lying over the place~$v$ of its totally real
subfield, and we do not explicitly mention prolongations.

\begin{lem}[Descending the existence of a compatible system]
  \label{lem: existence of compatible system from potential
    existence} Let~$S$ be a finite set of finite places
of~$F^+$ which contains all the places dividing~$p$, and let $(\abar,\mubar)$ be a polarized
  representation which is unramified outside of~$S$. Let~$\mu:G_{F^+}\to\Zpbartimes$ be a de Rham lift
of~$\mubar$ which is unramified outside of~$S$. For each $v\in
  S$, let $A_v$ be a $\mu$-polarized component 
for~$\abar|_{G_{F_w}}$. Assume that $\abar:G_F\to\GL_n(\Fpbar)$ is reasonable.  

Suppose that there is a finite Galois extension of CM fields
$L/F$, linearly disjoint from~$\overline{F}^{\ker{\abar}}(\zeta_p)$ over~$F$, 
and
an odd, polarized, regular, weakly irreducible compatible system of
representations $(\{s_\lambda\},\{\mu'_\lambda\})$ of~$G_L$, 
with associated $p$-adic representation~$(s,\mu')$,  with the
properties that:
\begin{enumerate}
\item $\abar|_{G_L}$ is reasonable.
\item $\sbar\cong\abar|_{G_L}$.
\item $\mu'=\mu|_{G_L}$.
\item $s$ is unramified outside of the places lying over~$S$.
\item For each place $v_L$ of~$L^+$ lying over a place~$v\in S$,
  $s|_{G_{L_{w_L}}}$ lies on $A_v|_{L^+_{v_L}}$.
\end{enumerate}
Then there is an odd, polarized, regular, weakly irreducible compatible system of
representations $(\{a_\lambda\},\{\mu_\lambda\})$ of~$G_F$, with associated $p$-adic representation~$(a,\mu)$,  with the
properties that:
\begin{enumerate}
\item $a$ lifts~$\abar$.
\item $a$ is unramified outside of~$S$.
\item For each place~$v\in S$,
  $a|_{G_{F_w}}$ lies on $A_v$.
\end{enumerate}
\end{lem}
\begin{proof}By Theorem~\ref{thm: BLGGT splitting at primes}, after
  replacing~$L$ by a finite extension, we can and do assume
  that~$(\{s_\lambda\},\{\mu'_\lambda\})$ is automorphic. Let~$R_A$ be the universal $\cO$-deformation algebra for
  $\abar$, for the deformations which are:
  \begin{itemize}
  \item $\mu$-polarized and odd.
  \item unramified outside of~$S$.
  \item If $v\in S$, then the corresponding lift lies on~$A_v$.
  \end{itemize}
Let $R_{A'}$ be the  universal $\cO$-deformation algebra for
  $\sbar$, for the deformations which are:
  \begin{itemize}
  \item $\mu'$-polarized and odd.
  \item unramified outside of the primes lying over~$S$.
  \item If $v\in S$, and $v_L |v$ is a place of~$L^+$, then the corresponding lift lies on~$A_v|_{L^+_{v_L}}$.
  \end{itemize}
 By~Proposition~\ref{prop: global deformation ring is positive dimensional}, $R_{A}$ has
positive dimension, and by~\cite[Lem.\ 1.2.3]{BLGGT} it is finite
over~$R_{A'}$. Since~$R_{A'}$ is finite over~$\cO$ by
Lemma~\ref{finiteness lemma}, we deduce that~$R_A$ is finite over~$\cO$. It follows that~$R_A$ has $\Qpbar$-points, and we
let~$a$ be the representation corresponding to such a point.

It remains to check that~$(a,\mu)$ is part of a weakly irreducible
compatible system. Since (by construction)~$(s,\mu')$ is automorphic, it follows 
from Theorem~\ref{thm: FC's pst R=T} that~$(a|_{G_L},\mu')$ is automorphic. Then~$(a,\mu)$ is part of a compatible system by
the usual argument with Brauer's theorem; to be precise, it follows
from the proof of~\cite[Thm.\ 5.5.1]{BLGGT}, as the appeal
to~\cite[Thm.\ 4.5.1]{BLGGT} in that proof is only in order to prove
potential automorphy, which we have already established. (Note that
the hypothesis on the component groups of the Galois representations
made in the proof of~\cite[Thm.\ 5.5.1]{BLGGT} is guaranteed by
Lemma~\ref{lem: control of component groups}; we are free to twist our
representations by an algebraic character in order to guarantee that
the determinant has infinite order.) This compatible system is  weakly irreducible
by Lemma~\ref{lem: weakly irreducible equals
  potentially automorphic}.
  \end{proof}

\begin{cor}[Level lowering for a compatible system]
  \label{cor: removing finite ramification} Let~$S$ be a finite set of finite places
of~$F^+$, containing all of the places dividing~$p$, 
and let $(\abar,\mubar)$ be a polarized
  representation which is unramified outside of~$S$. Let~$\mu:G_{F^+}\to\Zpbartimes$ be a 
  de Rham lift
of~$\mubar$ which is unramified outside of~$S$. Assume that
$\abar:G_F\to\GL_n(\Fpbar)$ is reasonable. For each $v\in S$, let
$A_v$ be a $\mu$-polarized component 
for~$\abar|_{G_{F_w}}$. Let~$\Sigma$ be a finite set of finite places
of~$F^+$, which is disjoint from~$S$, and for each $v\in\Sigma$, let
$A_v$ be a $\mu$-polarized component for~$\abar|_{G_{F_w}}$ which is potentially
unramified. 

Suppose that there is 
an odd, polarized, regular, weakly irreducible system of
representations $(\{s_\lambda\},\{\mu_\lambda\})$ of~$G_F$, with associated $p$-adic representation~$(s,\mu)$,  with the
properties that:
\begin{enumerate}
\item $\sbar\cong\abar$.
\item $s$ is unramified outside of~$S\cup\Sigma$.
\item For each place~$v\in S\cup\Sigma$,
  $s|_{G_{F_w}}$ lies on $A_v$.
\end{enumerate}
Then there is an odd, polarized, regular, weakly irreducible system of
representations $(\{a_\lambda\},\{\mu_\lambda\})$ of~$G_F$, with associated $p$-adic representation~$(a,\mu)$,  with the
properties that:
\begin{enumerate}
\item $a$ lifts~$\abar$.
\item $a$ is unramified outside of~$S$.
\item For each place~$v\in S$,
  $a|_{G_{F_w}}$ lies on $A_v$.
\end{enumerate}
\end{cor}
\begin{proof}
  We may choose a finite Galois extension~$L/F$, linearly disjoint
  from~$\overline{F}^{\ker{\abar}}(\zeta_p)$ over~$F$, 
  with the properties
  that~$\abar|_{G_{L}}$ remains reasonable, and for each~$v\in\Sigma$, and
  each place~$v_L|v$ of~$L$, $s|_{G_{L_{w_L}}}$ is unramified. The result
  then follows from Lemma~\ref{lem: existence of compatible system from potential
    existence}, applied to  $(\{s_\lambda|_{G_L}\},\{\mu_\lambda|_{G_{L^+}}\})$.
\end{proof}

\subsection{Local Swapping}
In this subsection we prove our main theorem
(Theorem~\ref{thm: main theorem on existence of lifts} below), 
building on a series of lemmas.
We begin with the following lemma, which will be used to move components
between different mod~$p$ representations. Much of the rest of this
section is devoted to relaxing the rather restrictive hypotheses made
in this lemma,
culminating in Lemma~\ref{lem: local swapping II} below. 

 We again remind the reader that we are using the
conventions introduced in Definition~\ref{defn: polarized local isomorphisms} and Conventions~\ref{convention: w
  lies over v} and~\ref{convention: w lies over v 2}. We extend Convention~\ref{convention: w
  lies over v} in the obvious way to subscripts and superscripts, so
that~$w_1$ is a place of~$F$ over~$v_1$ in~$F^{+}$, $w_L$ is a place of~$L$ over~$v_L$ in~$L^{+}$, and so on.
\begin{lemma}[Local swapping I] \label{lem: local swapping}Suppose that either $n>1$ is odd, or~$n=2$. 
  Let~$F$ be a CM field, and let
$(\{a_{\lambda}\},\{\mu_\lambda\})$
and~$(\{b_{\lambda}\},\{\nu_\lambda\})$ be two weakly irreducible, odd, regular,
polarized compatible systems of $n$-dimensional representations with corresponding~$p$-adic
representations~$(a,\mu)$ and~$(b,\nu)$. Let~$S$ be a finite set
of finite places 
of~$F^+$ containing all of the places lying over~$p$,  such that each of~$a$, $b$, $\mu$, or~$\nu$ is
unramified outside of~$S$. For each $v\in S$, write $A_v$
for the $\mu$-polarized component determined by~$a|_{G_{F_w}}$, and~$B_v$ for the
$\nu$-polarized component determined by~$b|_{G_{F_w}}$. 
Assume that:
\begin{enumerate} 
\item For each place~$v\in S$ with $v|p$, the component~$A_v \otimes B_v$ is
  regular.
\item The representations~$\abar$ and~$\bbar$ are reasonable, and
$(\abar\otimes\bbar)(G_{F(\zeta_p)})$ is adequate. 
\item The image~$\bbar(G_F)$ contains~$\SL_n(\Fp)$. 
\item There is a finite place $x \nmid p$ with~$x \in S$ of~$F^+$ which is inert
  in~$F$, and a character~$\psi:G_{F_x}\to\Qpbar^\times$ such that:
  \begin{itemize}
   \item $\psi^c|_{I_{F_x}} =\psi^{-1}|_{I_{F_x}}$.
    \item $\psi|_{I_{F_x}}$ has \emph{(}finite\emph{)} order greater than~$2$.
  \item $a|_{G_{F_x}}$ is unramified, and
  \item     $b|_{I_{F_x}}\cong\psi|_{I_{F_x}}\oplus\mathds{1}^{\oplus(n-1)}$.
  \end{itemize}
\end{enumerate}
We let $T\subset S\setminus\{x\}$ be a set of places with the property that if $v\in
T$, then there is
an equality~$\mu|_{G_{F^+_v}}=\nu|_{G_{F^+_v}}$
and a polarized isomorphism $\abar|_{G_{F_w}}\cong\bbar|_{G_{F_w}}$. 
Furthermore, if $v\in T$, then we
set $C_v=B_v$ and $D_v=A_v$, while if $v\in S\setminus T$, then we set
$C_v=A_v$ and $D_v=B_v$.

 Then there exist  odd, regular, polarized, weakly irreducible
 compatible systems~$(\{c_{\lambda}\},\{\mu_\lambda\})$ and~$(\{d_\lambda\},\{\nu_\lambda\})$ with corresponding $p$-adic
 representations~$(c,\mu)$ and~$(d,\nu)$, having the following properties:
 \begin{itemize}
 \item $\cbar\cong\abar$ and $\dbar\cong\bbar$.
 \item For each place $v\in S$, $c|_{G_{F_w}}$ lies on~$C_v$ and  $d|_{G_{F_w}}$ lies on~$D_v$.
 \item $c$ and~$d$ are unramified outside of~$S$.
 \end{itemize}
\end{lemma}

\begin{proof}Note firstly that since~$\bbar$ is reasonable, we have
  $p>2(n+1)\ge 6$, so in particular $\bbar(G_F)$ is large enough that
  Lemma~\ref{lemma:large} applies. Note also that
  since~$\abar|_{G_{F_x}}$ is unramified but~$\bbar|_{G_{F_x}}$ is
  ramified, we have $x\notin T$.  By Corollary~\ref{cor: removing finite ramification},
  it suffices to construct the desired compatible systems after increasing the set~$S$, provided that the additional
  components~$A_v$ that we choose are potentially
  unramified.
Note that the assumption that $(\abar\otimes\bbar)(G_{F(\zeta_p)})$ is adequate includes the assumption that~$\abar\otimes\bbar$ is absolutely irreducible; 
  accordingly, we allow ramification in~$S\cup\Sigma$, where
    $\Sigma$, and the components~$C_v$ and~$D_v$ for $v\in\Sigma$,
    are a suitable choice as in Definition~\ref{defn: suitable}. 
     We set $A_v=C_v$ and $B_v=D_v$ for $v\in\Sigma$.
  
We consider the following four global deformation $\cO$-algebras $R_{A,\Sigma}$, $R_{B,\Sigma}$, $R_{C,\Sigma}$,
and~$R_{D,\Sigma}$, defined as follows (from this point onwards we drop the~$\Sigma$ from the
notation):
\begin{enumerate}
\item $R_{A}$ and~$R_{C}$ are deformation rings for~$\abar$; $R_{B}$ and~$R_{D}$ are deformation rings for~$\bbar$.
\item The deformations are polarized and odd; the multiplier
  characters of~$R_A, R_B, R_C, R_D$ are respectively~$\mu,\nu,\mu,\nu$.
\item If~$v\in S \cup \Sigma$, then 
  the restriction to~$G_{F_w}$ of the universal deformation
  corresponding to~$R_{A}$, $R_{B}$, $R_{C}$, or~$R_{D}$, lies on the component~$A_{v}$, $B_{v}$, $C_{v}$, or~$D_v$
respectively.
\item The representations are unramified outside of~$S \cup \Sigma$.
\end{enumerate}

We also consider a fifth deformation $\cO$-algebra~$R_{A \otimes B}$ for~$\abar \otimes \bbar$, which is defined to have the following properties:
\begin{enumerate}
\item The deformations are polarized and odd, with multiplier~$\mu\nu\delta_{F/F^+}$.
\item If~$v\in S \cup \Sigma$, then 
the corresponding lift  lies on the component~$A_v \otimes
B_v=C_v\otimes D_v$.
\item The representations are unramified outside of~$S \cup \Sigma$.
\end{enumerate}
Note that $\{a_\lambda\otimes b_\lambda\}$ is weakly irreducible
by Lemma~\ref{lem: tensor product weak irreducibility}.
It follows from Lemma~\ref{finiteness lemma} (with~$S$ replaced by~$S \cup \Sigma$)
that $R_{A\otimes B}$ is a finite $\cO$-algebra. 
We claim that~$R_{C}$ and~$R_{D}$
are also both finite over~$\OL$. 

To prove this, we first note that
any deformation coming from~$R_{C}$ tensored with one from~$R_{D}$ gives, functorially, a deformation
of type~$R_{A \otimes B}$. The representation~$a\otimes b$ also gives
such a point. 
By Yoneda's Lemma, there exist corresponding morphisms:
$$R_{A \otimes B} \rightarrow R_{A} \hotimes_{\cO} R_{B},
\qquad R_{A \otimes B} \rightarrow R_{C} \hotimes_{\cO} R_{D}.$$
Since $R_{A\otimes B}$ is finite over~$\cO$, it then suffices to show that
the morphism $R_{A \otimes B} \rightarrow R_{C} \hotimes_{\cO} R_{D}$ is
finite. 

By Nakayama's lemma, we are reduced to showing that if~$\cA$ is an Artinian~$k$-algebra
such that~$\abar$ and~$\bbar$ admit deformations~$\widetilde{a}$ and~$\widetilde{b}$ to~$\cA$ (of types~$C$
and~$D$ respectively) so that~$\widetilde{a} \otimes \widetilde{b}$
is the trivial deformation of~$\abar \otimes \bbar$, then the
corresponding map ~$ R_{C} \hotimes_{\cO} R_{D}\to\cA$ factors through
some subalgebra~$\cA'$ of $\cA$ of
uniformly bounded length. 
To show this,
let~$M$ be the fixed field of the kernel of~$\abar\oplus\bbar$. Then 
we find that 
$$\widetilde{a}|_{G_{M}} \otimes  \widetilde{b} |_{G_{M}} = (\abar \otimes_k \cA)|_{G_{M}} \otimes (\bbar \otimes_{k} \cA) |_{G_{M}} $$
is a free~$\cA$-module with a trivial action of~$G_M$.
Now, we claim that if~$V$ and~$W$ are free~$\cA$-modules with an action
of~$G_M$ such that the diagonal action on~$V \otimes_{\cA} W$
is trivial,
then~$G_M$ acts on~$V$ and $W$ by scalars.
Assume for the sake of contradiction that~$G_M$ does not act on~$V$ via a character. Then there exists an element~$v \in V$ and~$g \in G_M$
such that~$g v$ is not a multiple of~$v$. 
Yet then (choosing
$w$ to be any element of $W \setminus \mathfrak{m}_{\cA} W$), $g (v \otimes w)$ cannot possibly
be a multiple of~$v \otimes w$, a contradiction.
Hence the action of~$G_M$ on~$V$ (and~$W$) is by scalars.

Let~$R_{\det(C)}$, $R_{\det(D)}$, etc.\ denote the corresponding deformation rings for the determinants of our representations.
We have the following diagram:
\[
\begin{tikzcd}
R_{\det(A \otimes B)} \arrow{r} \arrow{d} & R_{\det(C)} \widehat{\otimes}  R_{\det(D)}  \arrow{d}  \\
R_{A \otimes B} \arrow{r} &  R_{C} \widehat{\otimes} R_{D}
\end{tikzcd}
\]
We begin by showing that~$R_{\det(C)}$ and~$R_{\det(D)}$ are finite over~$\cO$. The Hodge-theoretic conditions imply that any two characters of
type~$\det(C)$ (or~$\det(D)$) differ by a finite order character unramified outside~$S \cup \Sigma$, and with ramification at~$v|p \in S \cup \Sigma$
bounded purely by the corresponding type. The finiteness of~$R_{\det(C)}$ and~$R_{\det(D)}$ over~$\OL$ is now
an immediate consequence of class field theory. (Fixing one such pair of characters~$\tau_C$ and~$\tau_D$, the deformation
rings over~$k$ are identified with group rings~$k[\Gamma]$ for some finite ray class group~$\Gamma$.)

Hence we may additionally assume that the determinants of~$\widetilde{a}$ and~$\widetilde{b}$ are fixed.
From the argument above, we have also shown that the action of~$G_M$ on the corresponding~$\cA$-modules~$V$ and~$W$ is via a scalar.
Since there are only finitely many characters of~$G_M$ of order~$n$
which are unramified outside any fixed finite set of primes, we deduce that there exists a finite Galois extension~$N/F$ such that the action
of~$G_N$ on~$V$ and~$W$ is trivial. But this implies that the corresponding maps from~$R_C$ and~$R_D$ to~$\cA$ factor
through the quotient of the universal deformation ring over~$k$ of~$\abar$ and~$\bbar$ as representations of the finite group~$\Gal(N/F)$. 
But the finiteness of these deformation rings  follows exactly as in the proof of Lemma~1.2.3 of~\cite{BLGGT}.

Since $R_C$ and $R_D$ are finite over~$\cO$, and have dimension at
least one by~Proposition~\ref{prop: global deformation ring is
  positive dimensional}, this shows that~$R_C$ and~$R_D$ both have
non-trivial~$\Qbar_p$-valued points. Make a choice of such points, and
let $c$, $d$ be the corresponding $p$-adic representations. Recall
~that $\{a_\lambda\otimes b_\lambda\}$ is weakly irreducible, and thus
potentially automorphic by Lemma~\ref{lem: weakly irreducible equals
  potentially automorphic}. In particular, given any finite Galois extension~$\Favoid/F$, we can find a CM Galois
extension~$L/F$ which is linearly disjoint from~$\Favoid/F$ and is such
that~$\{(a_\lambda\otimes b_\lambda)|_{G_L}\}$ is
automorphic. Furthermore, by replacing~$\Favoid$
by~$\Favoid\overline{F}^{\ker\overline{a}\otimes\overline{b}}(\zeta_p)$,
we may assume that~$(\overline{a}\otimes\overline{b})(G_{L(\zeta_p)})$ is adequate,
and~$\zeta_p\notin L$. Making a further quadratic base change if
necessary, we can also assume that all places at 
which~$a\otimes b$ is ramified, and all places lying over~$p$, are
split places. Then
$(c\otimes d)|_{G_L}$ is  automorphic by Theorem~\ref{thm: FC's pst
  R=T}. 
  As in the proof
of Lemma~\ref{lem: existence of compatible system from potential
  existence}, it follows from the proof of~\cite[Prop.\ 5.5.1]{BLGGT}
that~$c\otimes d$ is part of a regular, odd, polarizable weakly
irreducible compatible system~$\{t_\lambda\}$ (again, using
Lemma~\ref{lem: control of component groups} to guarantee the
hypothesis on the component groups made in~\cite[Prop.\
5.5.1]{BLGGT}).  

We will now apply
Theorem~\ref{thm: factoring compatible systems} to the compatible
system~$\{t_\lambda\}$, and so deduce that $c$, $d$ are the $p$-adic representations
corresponding to compatible systems~$\{c_\lambda\}$, $\{d_\lambda\}$,
which (by construction) satisfy all the requirements of the lemma.
To complete the proof, it suffices to verify that the hypotheses
of Theorem~\ref{thm: factoring compatible systems}
are satisfied by~$\{t_\lambda\}$. 
The strong irreducibility assumptions are satisfied by
Lemma~\ref{lem: strong irreducibility for compatible system
    given local conditions} and the conditions at the places
  in~$\Sigma$, which imply that the hypotheses of
  Lemma~\ref{lemma:reducetocyclo} hold.
  The hypothesis on the image of $\bbar$ is satisfied by assumption.
Finally, the place $x$ required in the hypotheses 
of Theorem~\ref{thm: factoring compatible systems}
can be taken to be the place $x$ appearing in hypothesis~(4)
of the present lemma: indeed, the hypotheses on the behavior of 
$a$ and $b$ at $x$ involve just the restriction
of these representations to the inertia group at $x$, and thus they are
constant along the components $A_x$ and $B_x$.  As noted above,
$x \not\in T$, and thus $C_x = A_x$ and $D_x = B_x$,
so that these same hypotheses are satisfied by $c$ and $d$.
This completes the verification of the hypotheses 
of Theorem~\ref{thm: factoring compatible systems},
and so also completes the proof of the lemma.
\end{proof}

\begin{rem} The polarizability requirement in the previous lemma is essential, even for~$n = 1$.
In particular, there are genuine global obstructions (arising from units)
to producing characters with prescribed ramification
properties at all primes. However, the polarizable condition implies that (in our setting)
the corresponding characters are unitary and come from the~$-1$-part of the corresponding ray class groups;
in particular, there are no unit obstructions provided that~$p \ne 2$
(and in this paper, $p$ is never~$2$).
\end{rem}
We now introduce potentially diagonalizable lifts into the picture. We
remind the reader of Convention~\ref{convention: lifts will always be
  well spread} (that we will assume without explicit mention that the
gaps between the Hodge--Tate weights of our potentially diagonalizable
lifts are sufficiently large); we will also sometimes assume without further comment that in base change arguments
 the weights and
types of the potentially diagonalizable representations have been
chosen compatibly.
\begin{lem}[Existence of potentially diagonalizable
  lifts]\label{lem: existence of PD lifts} Let~$\sbar:G_F\to\GL_n(\Fpbar)$ be a 
pleasant  representation in the sense of {\em Definition~\ref{defn:pleasant}},
  namely:
  \begin{itemize}
  \item $\zeta_p\notin F$, and $\sbar|_{G_{F(\zeta_p)}}$ is
    irreducible.
  \item $\sbar$ is polarizable and odd.
  \item $p> 2(n+1)$.
  \item All the primes~$v|p$ in~$F^{+}$ split completely in~$F$.
  \item For each place~$w|p$ of~$F$, $\sbar|_{G_{F_w}}$ admits many
    diagonalizable lifts.
   \end{itemize}
 Let~$\mubar$ be a character such that~$(\sbar,\mubar)$ is
  polarized. Let~$\mu$ be a de Rham lift of~$\mubar$. Let~$S$ be a
  finite set of finite places of~$F^+$, including all places at
  which~$\sbar$ or~$\mu$ is ramified, and all places lying
  over~$p$. 
  
  Then there is an odd, regular, polarized weakly irreducible
  compatible system~$(\{s_\lambda\},\{\mu_\lambda\})$ of
  $G_F$-representations whose associated $p$-adic representation
  is~$(s,\mu)$, where:
  \begin{itemize}
  \item $s$ lifts~$\sbar$, 
  \item $s$ is unramified outside of~$S$, and
  \item $s$ is potentially diagonalizable at all places~$v|p$.
  \end{itemize}
Furthermore, if we fix a $\mu$-polarized component~$C_v$ of~$\sbar$ for
each place $v \nmid p$ in~$S$, then we may assume that
$s|_{G_{F_w}}$ lies on~$C_v$ for each~$v$.
  \end{lem}
   \begin{proof}The existence of~$s$ follows immediately from~\cite[Thm.\
    5.2.1]{2017arXiv170804885B} (the hypotheses of which hold
    by~\cite[Thm.\ 4.3.1, 4.5.1]{BLGGT}). 
That~$s$ is part of a compatible system follows from~\cite[Thm.\ 5.5.1]{BLGGT}. 
  \end{proof}
We will also make use of the following variant on the previous result,
where we no longer require that~$\sbar|_{G_{F_w}}$ admit
potentially diagonalizable lifts for~$w|p$ in~$F$,
but we allow a finite base change.

\begin{lem}[Existence of potentially diagonalizable
  lifts II]\label{lem: existence of PD lifts II} Let~$\sbar:G_F\to\GL_n(\Fpbar)$ be a reasonable
  representation in the sense of Definition~\ref{defn:reasonable}, that is:
  \begin{itemize}
  \item $\zeta_p\notin F$, and $\sbar|_{G_{F(\zeta_p)}}$ is
    irreducible.
  \item $\sbar$ is polarizable and odd.
  \item $p> 2(n+1)$.
   \end{itemize}
 Let~$\mubar$ be a character such that~$(\sbar,\mubar)$ is
  polarized. Let~$\mu$ be a de Rham lift of~$\mubar$. Let~$S$ be a
  finite set of finite places of~$F^+$, including all places at
  which~$\sbar$ or~$\mu$ is ramified, and all places lying
  over~$p$. Let~$\Favoid/F$ be a finite extension.

  Then there is a finite Galois extension of CM fields~$L/F$, linearly disjoint
  from the field~$\Favoid(\zeta_p)$ over~$F$, and an odd, regular, polarized weakly irreducible
  compatible system~$(\{s_\lambda\},\{\mu_\lambda\})$ of
  $G_L$-representations whose associated $p$-adic representation
  is~$(s,\mu|_{G_L})$, where:
  \begin{itemize}
  \item $\sbar|_{G_L}$ is pleasant,
  \item $s$ lifts~$\sbar|_{G_L}$, 
  \item $s$ is unramified outside of the places lying over~$S_L$, and
  \item $s$ is potentially diagonalizable at all places~$v|p$.
  \end{itemize}
Furthermore, if we fix a $\mu$-polarized component~$C_v$ of~$\sbar$ for
each place $v \nmid p$ in~$S$, then we may assume that
$s|_{G_{F_{w_L}}}$ lies on~$C_v|_{L^+_{v_L}}$ for each~$v_L|v$.
  \end{lem}
  \begin{proof}
  By Lemma~\ref{lem:can base change reasonable to pleasant},
  we may find a finite extension~$L/F$ linear
  disjoint from~$\Favoid(\zeta_p)/F$ such that~$\sbar|_{G_L}$ is pleasant.
  The result then follows immediately from Lemma~\ref{lem: existence of PD lifts}.
  \end{proof}

\begin{rem}[Arguments both reasonable and pleasant] \label{rem:copout}
The notions of reasonable and pleasant (Definitions~\ref{defn:reasonable} and~\ref{defn:pleasant},  which
were also just recalled in the statements of Lemmas~\ref{lem: existence of PD lifts II} and~\ref{lem: existence of PD lifts}
respectively) are very closely related. Pleasant representations are automatically reasonable, and reasonable
representations are pleasant after a finite extension (Lemma~\ref{lem:can base change reasonable to pleasant}).  Since, in
the ultimate application, we only assume our residual representations are reasonable, we have endeavored
to structure the arguments below to only require reasonable hypotheses on the relevant residual representations.
However, at a number of points (including  in the \emph{conclusions} of Lemmas~\ref{lem: inverse Galois PD} 
and~\ref{cor: inverse Galois component swapping} as well as during the proofs of
 Lemma~\ref{lem: local swapping II} and Theorem~\ref{thm: main theorem on existence of lifts}), we deduce that certain residual
representations satisfy the stronger condition of pleasantness, even though, almost all of the time, we only ever
use the fact that these resulting representations are reasonable. (One notable exception is the proof
of Lemma~\ref{lem: inverse Galois PD},  where, to invoke Lemma~\ref{lem: existence of PD lifts}, we require
that~$\rbar$ is reasonable.) Thus, to the close reader, it should be born in mind that when it is stated that
a certain residual representation is pleasant, the main implication to take away is that is reasonable.
\end{rem}

 The following lemma and its corollaries will be used in our later arguments to replace
 a given  representation with one which behaves well under base
 change. 
 \begin{lem}[Auxiliary representations]
  \label{lem: inverse Galois PD}  Let~$(a,\mu)$ be polarized
  and odd, where $\abar:G_F\to\GL_n(\Fpbar)$ is reasonable, and let~$S$ be a finite set of finite places
of~$F^+$, containing all the places lying over~$p$, and all places at
which~$(a,\mu)$ is ramified. 
Let $\Favoid/F$ be a finite
extension, and let~$q$ be a power of~$p$.

Then there is a finite Galois extension $L/F$ with $\zeta_p\notin L$, which is linearly
disjoint from $\Favoid$ over~$F$, and a weakly irreducible, odd,
polarized, regular compatible
system~$(\{r_\lambda\},\{\mu_\lambda|_{G_{L^+}}\})$ of
$G_L$-representations with associated 
$p$-adic representation~$(r,\mu|_{G_{L^+}})$, which satisfies:
\begin{enumerate}
\item $\abar|_{G_L}$ is pleasant.
\item $\rbar$ and~ $\rbar|_{G_{\Lg}}$ are pleasant,
  and~$\rbar(G_{\Lg})\supset\SL_n(\F_{q})$;
  here~$\Lg$ denotes the Galois closure of~$L$ over~$\Q$. 
  \item $(\rbar\otimes\abar|_{G_L})(G_{L(\zeta_p)})$ is adequate.
\item Each place in~$S$ splits completely in~$L^+$. For each
  place~$v_L$ of~$L^+$ lying over a place~$v\in S$, we have polarized isomorphisms
  $\rbar|_{G_{L_{w_L}}}\cong\abar|_{G_{L_{w_L}}}=\abar|_{G_{F_w}}$.
  
\item There is a finite place~$x$ of~$L^+$ which is inert in~$L$ and
  does not lie over any place in~$S$, and a character~$\psi:G_{L_x}\to\Qpbar^\times$
  with~$\psi^c |_{I_{L_x}} =\psi^{-1} |_{I_{L_x}}$, such that
  $\mu|_{G_{L^+_x}}$ is unramified, $\psi|_{I_{L_x}}$ has finite order
  greater than~$2$,
and  ~$\rbar|_{I_{L_x}}\cong
\psibar |_{I_{L_x}}  \oplus\mathds{1}^{\oplus(n-1)}$.
\item There is an isomorphism~$r |_{I_{L_x}}\cong
\psi |_{I_{L_x}}  \oplus\mathds{1}^{\oplus(n-1)}$.
\item For any place~$v$ of~$L^+$ not lying over~$S$ at which~$r|_{G_{L_w}}$ is ramified,
  $r(I_{L_w})$ is finite.
\item $r|_{G_{L_{w}}}$ is potentially diagonalizable for all
  places~$w|p$ of~$F$.
\end{enumerate}
\end{lem}

\begin{proof} By Lemma~\ref{lem: existence of PD
    lifts}, it is enough to construct the odd, polarized
  pair~$(\rbar,\mubar|_{G_{L^+}})$ satisfying properties (1)--(5), since we may choose
  local components which force conditions~(6)--(8). Replacing~$\Favoid$ 
  with~$\Favoid\overline{F}^{\ker\abar}(\zeta_p)$, and applying
  Lemma~\ref{lem:can base change reasonable to pleasant}, we see that we can ignore the
  requirements that~$\abar|_{G_L}$ is pleasant, and that
  $\zeta_p\notin L$. 

  Choose a finite place~$y\notin S$ of~$F^+$ which is inert in~$F$ and
  is such that $\mu|_{G_{F^+_y}}$ is unramified, and  choose a
  character~$\psi:G_{F_y}\to\Qpbar^\times$
  with~$\psi^c |_{I_{F_y}}=\psi^{-1} |_{I_{F_y}}$ such that $\psi|_{I_{F_y}}$ has finite order
  greater than~$2$.
  (Note that we can do this for every place~$y$
  which is inert in~$F$.)
  
  Replacing~$q$ with a  power of~$q$ if necessary, we may assume
  that~$\psibar$, $\mubar$, and all of the
  representations~$\abar|_{G_{F_w}}$ with $v\in S$ are defined
  over~$\F_{q}$. We may also assume that~$\PSL_n(\F_{q})$ is
  simple, and is not isomorphic to any Jordan--H\"older factor
  of~$\Gal(F^{\mathrm{gal}}/\Q)$, and that~$q$ is large enough so that
  condition~(3) follows automatically from condition~(2) by
  Lemma~\ref{lem: can always make a tensor product adequate}. (Note
  that if $\rbar(G_{\Lg})\supset\SL_n(\F_{q^k})$ for some~$k$, then in particular
  $\rbar(G_{\Lg})\supset\SL_n(\F_{q})$.) 
  
Arguing exactly as in the proof of~\cite[Prop.\ A2]{MR3134019},
by~\cite[Prop.\ 3.2]{MR2869026}
(see also~\cite[Thm.~1.2]{MBdiditfirst})
we can find a finite totally real extension
$L^+/F^+$, linearly disjoint from $\Favoid$ over~$F^+$, in which all places
above~$S\cup\{y\}$ split completely, and a surjective odd reasonable
$\mubar|_{G_{L^+}}$-polarized representation~$\rbar:G_L\to\GL_n(\F_{q})$ (where~$L=L^+F$) with the property that for
each place~$v_L$ of~$L$ lying over a place~$v\in S$, we have a polarized isomorphism
$\rbar|_{G_{L_{w_L}}}\cong\abar|_{G_{L_{w_L}}}$. 
In addition, for each
place~$x$ of~$L^+$ lying over~$y$, we can ensure that $\rbar|_{I_{L_x}}\cong \psi |_{I_{L_x}} \oplus\mathds{1}^{\oplus(n-1)}$. Furthermore, by the same
restriction of scalars trick that we used in the proof of
Theorem~\ref{thm: BLGGT splitting at primes}, we can assume that~$L^+$
is of the form~$M^+F^+$, where~$M^+/\Q$ is Galois. 

It remains to check property~(2). Since~$M^+/\Q$ is Galois, we
have~$\Lg=M^+F^{\mathrm{gal}}=LF^{\mathrm{gal}}$, and since we are
assuming that~$\PGL_n(\F_{q})$ is
  simple and is not isomorphic to any Jordan--H\"older factor
  of~$\Gal(F^{\mathrm{gal}}/\Q)$, it follows that the projective image
  of~$\rbar|_{\Lg(\zeta_l)}$ contains~$\PGL_n(\F_{q})$, from which
  the required property follows at once. 
  \end{proof}

\begin{cor}[Auxiliary
  representations with specified components]
  \label{cor: inverse Galois component swapping}
  Suppose that either $n$ is odd or~$n=2$. 
  Let~$F$ be a CM field, and let
$(\{a_{\lambda}\},\{\mu_\lambda\})$
 be a weakly irreducible, odd, regular,
polarized compatible system of $n$-dimensional representations with corresponding~$p$-adic
representation~$(a,\mu)$. Assume that $\abar$ is reasonable. Let~$S$ be a
finite set of finite places
of~$F^+$ containing all of the places lying over~$p$, and chosen so that~$(a,\mu)$ is unramified outside of~$S$.

For each $v\in S$, write $A_v$ for the $\mu$-polarized component determined
by~$a|_{G_{F_w}}$.  Let $\Favoid/F$ be a finite extension, and let~$q$
be a power of~$p$.

Then there is a finite Galois extension $L/F$ with $\zeta_p\notin L$, which is linearly
disjoint from $\Favoid$ over~$F$, such that for each set $T$ of places
of~$L^{+}$ which divide~$p$, there is a weakly irreducible, odd,
polarized, regular compatible
system~$(\{s_\lambda\},\{\mu_\lambda|_{G_{L^+}}\})$ of  $G_L$-representations with associated
$p$-adic representation~$(s,\mu|_{G_{L^+}})$, which satisfies:
\begin{enumerate}
\item $\abar|_{G_L}$ is pleasant.
\item  $\sbar$ is pleasant, and is independent of the choice of~$T$
  \emph{(}up to isomorphism\emph{)}. 
  \item $\sbar|_{G_{\Lg}}$ is pleasant,
  and~$\sbar(G_{\Lg})\supset\SL_n(\F_{q})$.

\item $(\sbar\otimes\abar|_{G_L})(G_{L(\zeta_p)})$ is adequate.
\item for each
  place~$v_L$ of~$L^+$ lying over a place $v\in S$, we have a polarized isomorphism
  $\sbar|_{G_{L_{w_L}}}\cong\abar|_{G_{L_{w_L}}}$. 
\item \label{condition: onemore} for each
  place~$v_L\notin T$ of~$L^+$  lying over a place ~$v\in S$,
  $s|_{G_{L_{w_L}}}$ lies on~$A_v|_{G_{L^+_{v_L}}}$.

\item $s|_{G_{L_{w_L}}}$ is potentially diagonalizable for all
  places~$w_L$ lying over places in~$T$.

\item There is a finite place~$x$ of~$L^+$ which is inert in~$L$ and
  does not lie over a place in~$S$, and a character~$\psi:G_{L_x}\to\Qpbar^\times$
  with~$\psi^c |_{I_{L_x}} =\psi^{-1}|_{I_{L_x}}$, such that
  $\mu|_{G_{L^+_x}}$ is unramified, $\psi|_{I_{L_x}}$ has finite order
  greater than~$2$,
and~$s|_{I_{L_x}}\cong
\psi |_{I_{L_x}} \oplus\mathds{1}^{\oplus(n-1)}$.
\item For any place~$v$ of~$L^+$ not lying over~$S$ at which~$s|_{G_{L_w}}$ is ramified,
  $s(I_{L_w})$ is finite.
    \item \label{anyfield} Both~$L/F$, $x$,  and~$\psi$  can be chosen independently of the choice of~$T$.
\end{enumerate}
\end{cor}
\begin{proof}
By Lemma~\ref{lem: inverse Galois PD}, we may construct a compatible 
system~$(\{r_\lambda\}, \{\mu_{\lambda} |_{G_{L^{+}}}\})$ 
which satisfies all the required conditions of the theorem
except perhaps~(\ref{condition: onemore}), and is potentially diagonalizable
at all places above~$p$.
We now apply Lemma~\ref{lem: local swapping} to the compatible systems
~$(\{a_\lambda|_{G_L}\}, \{\mu_{\lambda} |_{G_{L^{+}}}\})$
and~$(\{r_\lambda\}, \{\mu_{\lambda} |_{G_{L^{+}}}\})$ with respect to the set~$T$.
The result is a compatible system~$(\{s_\lambda\},\{\mu_\lambda|_{G_{L^+}}\})$
with~$\sbar = \rbar$ independent of~$T$ and
such that~$(s,\mu|_{G_{L^{+}}})$ lies on the component associated to~$a$ for
all~$v_{L} \notin T$ and the component associated to~$r$ (which is a
potentially diagonalizable component) for all~$v_{L} \in T$, as required. Note that~$L$ is the field
on whose absolute Galois group~$\sbar$ is defined and so can be chosen independently of~$T$, and similarly~$x$ and~$\psi$
can be chosen independently of~$T$ (as an examination of the relevant
arguments shows).
  \end{proof}
  \begin{cor}
  \label{cor: could omit x if we wanted}
  Maintaining the notation and assumptions of~{\em Corollary~\ref{cor:
  inverse Galois component swapping}}, we can instead produce~$L/F$ and 
$\{s_\lambda\}$ satisfying conclusions \emph{(1)--(7)} of {\em Corollary~\ref{cor:
  inverse Galois component swapping}}, such that in addition~$s$ is
unramified outside of~$S$.
\end{cor}
\begin{proof}This follows from~Corollary~\ref{cor:
  inverse Galois component swapping} by making a further base change
to remove the ramification at places (including~$x$) not lying over~$S$.
\end{proof}
We can now establish the following improvement on Lemma~\ref{lem:
  local swapping}, where we no longer need to make assumptions (1)--(4)
there; we even allow~$\abar=\bbar$.

\begin{lemma}[Local swapping II] \label{lem: local swapping II}Suppose that either $n$ is odd or~$n=2$. 
  Let~$F$ be a CM field, and let
$(\{a_{\lambda}\},\{\mu_\lambda\})$
and~$(\{b_{\lambda}\},\{\nu_\lambda\})$ be two weakly irreducible, odd, regular,
polarized compatible systems of $n$-dimensional representations of~$G_F$ with corresponding~$p$-adic
representations~$(a,\mu)$ and~$(b,\nu)$. Let~$S$ be a finite set of finite places
of~$F^+$, containing all of the places lying over~$p$, and chosen such that~$(a,\mu)$ and~$(b,\nu)$ are unramified
outside of~$S$. For each $v\in S$, write $A_v$
for the $\mu$-polarized component determined by~$a|_{G_{F_v}}$, and~$B_v$ for the
$\nu$-polarized component determined by~$b|_{G_{F_v}}$. Assume that the representations~$\abar$ and~$\bbar$ are reasonable.

Let $T\subset S$ be a set of places with the property that if $v\in
T$, then~$\mu|_{G_{F^+_v}}=\nu|_{G_{F^+_v}}$, and there is a polarized
isomorphism $\abar|_{G_{F_w}}\cong\bbar|_{G_{F_w}}$. 
 If $v\in T$, then we
set $C_v=B_v$ and $D_v=A_v$, and if $v\in S\setminus T$, then we set
$C_v=A_v$ and $D_v=B_v$. 

 Then there exist  odd, regular, polarized, weakly irreducible
 compatible systems $(\{c_{\lambda}\},\{\mu_\lambda\})$ and~$(\{d_\lambda\},\{\nu_\lambda\})$ with corresponding $p$-adic
 representations~$(c,\mu)$ and~$(d,\nu)$, with the properties that:
 \begin{itemize}
 \item $\cbar\cong\abar$ and $\dbar\cong\bbar$.
 \item For each place $v\in S$, $c|_{G_{F_w}}$ lies on~$C_v$ and  $d|_{G_{F_w}}$ lies on~$D_v$.
 \item $c$ and~$d$ are unramified outside of~$S$.
 \end{itemize}
\end{lemma}
\begin{proof}
  Since the statement is symmetric in~$(\{a_\lambda\},\{\mu_\lambda\})$
  and~$(\{b_\lambda\},\{\nu_\lambda\})$, it is enough to prove that~$(\{c_\lambda\},\{\lambda_\mu\})$
  exists. We apply Corollary~\ref{cor: inverse Galois
    component swapping} \emph{twice}
    to the compatible system~$(\{b_{\lambda}\},\{\nu_\lambda\})$.
  The point of this construction is to provide auxiliary compatible systems
  to which we can apply Lemma~\ref{lem: local swapping}.
    In the first application, we take~$T$ in Corollary~\ref{cor: inverse Galois
    component swapping} to be
    the set of places of~$F^+$ lying over~$p$,  and in the second application, we take~$T$ to be
    the
  set of places in our~$S\setminus T$ which lie over~$p$. The extension~$L/F$, the place~$x$, and the character~$\psi$ can be chosen
  independently of~$T$ by Corollary~\ref{cor: inverse Galois
    component swapping} part~(\ref{anyfield}), and hence we may make the same choice on both cases.
     We deduce the existence
  of a finite Galois
extension of CM fields~$L/F$, and odd, regular, polarizable, weakly irreducible
 compatible systems~$(\{s_{\lambda}\},\{\nu_\lambda|_{G_{L^+}}\})$ and~$(\{t_\lambda\},\{\nu_\lambda|_{G_{L^+}}\})$ of
 $G_L$-representations, with associated $p$-adic representations~$(s,\nu|_{G_{L^+}})$
 and~$(t,\nu|_{G_{L^+}})$ respectively, such that:
 \begin{itemize}
 \item $\sbar\cong\tbar$.
 \item $\abar|_{G_L}$, $\bbar|_{G_L}$, and $\sbar=\tbar$ are pleasant.
 \item $\sbar(G_L)$ contains $\SL_n(\Fp)$.
 \item  $(\sbar \otimes \abar|_{G_L})(G_{L(\zeta_p)})$ is adequate. 
 \item for each
  place~$v_L$ of~$L^+$ lying over a place~$v\in S$, we have polarized isomorphisms
  $\sbar|_{G_{L_{w_L}}}=\tbar|_{G_{L_{w_L}}}\cong\abar|_{G_{L_{w_L}}}$.
\item for each
  place~$v_L$ of~$L^+$  not dividing~$p$ and lying over a place~$v\in T$,
  $s|_{G_{L_{{w_L}}}}$  and $t|_{G_{L_{w_L}}}$ lie
  on~$B_v|_{{L^+_{v_L}}}$.
\item $s$ is potentially diagonalizable at every place dividing~$p$.
\item $t$ is potentially diagonalizable at every place dividing~$p$
  and lying over a place in~$S\setminus T$.
\item at each place $v_L$ of~$L$ lying over a place $v|p$ with $v\in T$,
  $t|_{G_{L_{w_L}}}$ lies on~$B_v|_{{L^+_{v_L}}}$.
\item There is a finite place~$x$ of~$L^+$ which is inert in~$L$ and
  does not lie over a place in~$S$, and a character~$\psi:G_{L^+_x}\to\Qpbar^\times$
  with~$\psi^c |_{I_{L_x}}=\psi^{-1} |_{I_{L_x}}$, such that
  $\mu|_{G_{L^+_x}}$ is unramified, $\psi|_{I_{L_x}}$ has finite order
  greater than~$2$, and
  $$t|_{I_{L_x}} \cong s|_{I_{L_x}}\cong
\psi |_{I_{L_x}} \oplus\mathds{1}^{\oplus(n-1)}.$$
 \end{itemize}
Applying Lemma~\ref{lem: local swapping} to~$(\{a_\lambda|_{G_L}\},\{\mu_\lambda|_{G_{L^+}}\})$
and~$(\{s_\lambda\},\{\nu_\lambda|_{G_{L^+}}\})$, we deduce the existence of an odd, regular, polarizable, weakly irreducible
 compatible
system~$(\{u_\lambda\},\{\mu_\lambda|_{G_{L^+}}\})$ of $G_L$-representations with associated
$p$-adic representation~$(u,\mu|_{G_{L^+}})$, with the properties that:
\begin{itemize}
 \item $\ubar\cong\abar|_{G_L}$.
 \item For each place $v_L$ of~$L^+$ lying over a place $v\in S$ not
   dividing~$p$, $u|_{G_{L_{w_L}}}$ lies on~$A_v|_{G_{L^+_{v_L}}}$.
 \item for every place $v_L|p$ of~$L^+$ lying over a place $v\in S\setminus T$,
   $u|_{G_{L_{w_L}}}$ lies on~$A_v|_{{L^+_{v_L}}}$.
 \item $u$ is potentially diagonalizable at every place dividing~$p$
   and lying over a
   place in~$T$.

 \item $u$ is unramified outside of~$S$.
\end{itemize}
Applying Lemma~\ref{lem: local swapping} to~$(\{u_\lambda\},\{\mu_\lambda|_{G_{L^+}}\})$ 
and~$(\{t_\lambda\},\{\nu_\lambda|_{G_{L^+}}\})$, we then obtain an odd, regular, polarized, weakly irreducible
 compatible
system~$(\{e_\lambda\},\{\mu_\lambda|_{G_{L^+}}\})$ of $G_L$-representations with associated
$p$-adic representation~$(e,\mu|_{G_{L^+}})$, with the properties that:
\begin{itemize}
 \item $\ebar\cong\abar|_{G_L}$.
 \item For each place $v_L$ lying over a place $v\in S\setminus T$, $e|_{G_{L_{w_L}}}$ lies on~$A_v|_{{L^+_{v_L}}}$.
 \item For each place $v_L$ lying over a place $v\in  T$,
   $e|_{G_{L_{w_L}}}$ lies on~$B_v|_{{L^+_{v_L}}}$.
 \item $e$ is unramified outside of~$S$.
\end{itemize}
The existence of~$(\{c_\lambda\},\{\mu_\lambda\})$ now follows from Lemma~\ref{lem: existence of compatible system from potential
    existence}, applied to~$(\{e_\lambda\},\{\mu_\lambda|_{G_{L^+}}\})$.
\end{proof}

\begin{cor}[Merging components] \label{cor: merging components}Suppose that either $n$ is odd, or~$n=2$. 
  Let~$F$ be a CM field, and let~$S$ be a finite set of finite places
of~$F^+$, containing all of the places lying over~$p$. Let ~$T$ be a
set of
places of~$F^+$ which divide~$p$.  Let
$(\{a_{\lambda}\},\{\mu_\lambda\})$
and~$\bigl\{(\{r^v_{\lambda}\},\{\nu^v_\lambda\})\bigr\}_{v \in T}$ 
be weakly irreducible, odd, regular,
polarized compatible systems of $n$-dimensional representations of~$G_F$ with corresponding~$p$-adic
representations~$(a,\mu)$ and~$(r^v,\nu^v)$.

For each $v\in S$, write $A_v$
for the $\mu_v$-polarized component determined by~$a|_{G_{F_w}}$. Assume that all of the
representations~$\abar$ and~$\{\rbar^v\}_{v \in T}$ are reasonable, and that the
representations~$(a,\mu)$ and~$\{(r^v,\nu^v)\}_{v \in T}$ are unramified outside
of~$S$. Assume also, for each~$v\in T$,
that $\mu|_{G_{F^+_v}}=\nu^v|_{G_{F^+_v}}$,
and that there is a polarized isomorphism
$\abar|_{G_{F_w}}\cong\rbar^v|_{G_{F_w}}$.

 Then there exists an  odd, regular, polarized, weakly irreducible
 compatible system~$(\{c_{\lambda}\},\{\mu_\lambda\})$ with corresponding $p$-adic
 representation~$(c,\mu)$, with the properties that:
 \begin{itemize}
 \item $\cbar\cong\abar$.
 \item For each place $v\in S\setminus T$, $c|_{G_{F_w}}$ lies on~$A_v$.
 \item For each place~$v\in T$, $c|_{G_{F_w}}$ lies on the component
   determined by~$r^v|_{G_{F_w}}$.
 \item $c$ is unramified outside of~$S$.
 \end{itemize}
\end{cor}
\begin{proof}Let $v_1,\dots,v_m$ be the places in~$T$. We claim
  that for each~$0\le i\le m$, we can find an odd, regular, polarized, weakly irreducible
 compatible system~$(\{c^i_{\lambda}\},\{\mu_\lambda\})$ with corresponding $p$-adic
 representation~$(c^i,\mu)$, with the properties that:
 \begin{itemize}
 \item $\cbar^i\cong\abar$.
 \item For each place $v\in S\setminus T$, $c^i|_{G_{F_w}}$ lies on~$A_v$.
 \item For each $j\le i$, $c^i|_{G_{F_{w_j}}}$ lies on the component
   determined by~$r^{v_j}|_{G_{F_{w_j}}}$.
 \item For each $j>i$, $c^i|_{G_{F_{w_j}}}$ lies on the component
   determined by~$a|_{G_{F_{w_j}}}$.
 \item $c^i$ is unramified outside of~$S$.
 \end{itemize}
  Assuming we can do this, we can take $\{c_\lambda\}:=\{c^m_\lambda\}$. We prove the
  existence of the~$\{c^i_\lambda\}$ by induction on~$i$,
  taking~$\{c^0_\lambda\}:=\{a_\lambda\}$. Then the existence
  of~$\{c^{i+1}_\lambda\}$ is immediate from Lemma~\ref{lem: local swapping II}
  applied to~$(\{c^i_\lambda\},\{\mu_\lambda\})$ and $(\{r^{v_{i+1}}_\lambda\},\{\nu^{v_{i+1}}_\lambda\})$ (taking
  the set~$T$ there to be~$\{v_{i+1}\}$).
\end{proof}

Our main theorem is the following (recall that the notion of a
globally realizable component was defined in Definition~\ref{def:glob}).

\begin{theorem}\label{thm: main theorem on existence of lifts} Assume
  that either $n$ is odd or $n=2$. Let~$F$ be a CM field, and let
  $(\sbar,\mubar)$ be a polarized representation, where
$\sbar: G_{F} \rightarrow \GL_n(\Fbar_p)$
is reasonable. 
Let $S$ be a finite set of finite places of~$F^+$, such that~$S$ contains
all of the places at which~$(\sbar,\mubar)$ is ramified and all of the
places lying over~$p$. Let~$\mu$ be a de Rham lift of~$\mubar$ which is
unramified outside of~$S$.
 For each place $v\in S$, let~$C_v$ be a $\mu$-polarized component
for~$\sbar|_{G_{F_w}}$. 

Assume that~$C_v$ is  globally realizable for each $v\in S$ which divides~$p$.
 Then there exists an odd, regular, polarized, weakly irreducible compatible
system~$(\{s_{\lambda}\},\{\mu_\lambda\})$ of~$G_F$-representations
with associated $p$-adic representation~$(s,\mu)$, which satisfies:
\begin{enumerate}
\item $s$ lifts~$\sbar$, and for each place $v\in S$, the
  representation~$s |_{G_{F_w}}$ lies on~$C_v$.
\item $s$ is unramified outside~$S$. 
\end{enumerate}
\end{theorem}
\begin{proof}
Fix a place $v_1\in S$ which lies over~$p$.    By definition, 
   the hypothesis that $C_{v_1}$ is globally realizable
  means that we can find a CM field $E$, a place~$v_{E,2}|p$ of~$E^+$ such
  that $E^+_{v_{E,2}}\cong F^+_{v_1}$, a finite set~$S_E$ of  places
  of~$E^+$ containing all of the places lying over~$p$, and an odd, regular,
  polarized, weakly irreducible compatible system~$(\{r_\lambda\},\{\mu_{E,\lambda}\})$ of
  $G_{E}$-representations with associated $p$-adic
  representation~$(r,\mu_E)$, such that $\rbar$ is reasonable, $r$ is
  unramified outside of~$S_E$, 
  and
  $r|_{G_{E_{w_{E,2}}}}$ lies on~$C_{v_1}$ (which as usual implies in
  particular 
  that $\mu_E |_{G_{E^+_{v_{E,2}}}} \simeq\mu |_{G_{F^{+}_{v_1}}}$ and 
  that there is a
polarized isomorphism $\rbar|_{G_{E_{w_{E,2}}}}\cong\sbar|_{G_{F_{w_1}}}$). 
    The reason we use the distinct numbers~$1$ and~$2$  in the subscripts is that we shall ultimately construct a field~$M$ containing both~$E$
    and~$F$, and there will be no \emph{a priori}  relation between the primes of~$M^{+}$ over~$v_{1}$ in~$F^{+}$
    and~$v_{E,2}$ in~$E^{+}$.

Applying Lemma~\ref{lem: existence of PD lifts II} to~$\sbar$, we may
find a finite Galois extension $L/F$ of CM fields and an odd, regular,
  polarizable, weakly irreducible compatible system~$(\{\sss_\lambda\},\{\mu_\lambda|_{G_{L^+}}\})$ 
  of
  $G_{L}$-representations with associated $p$-adic
  representation~$\sss$, such that~$\sbar|_{G_L}$ is pleasant, ~$\sss$
  lifts~$\sbar|_{G_L}$, and~$\sss$ is potentially diagonalizable at all
  places over~$p$. (Here~$\{\mu_\lambda\}$ is the compatible system
  containing the character~$\mu$.) This compatible system is our initial seed; we now
  begin constructing auxiliary compatible systems in order to apply our component
  swapping to ultimately construct the desired representation.

Applying
Corollaries~\ref{cor: inverse Galois component swapping} and~\ref{cor: could omit x if we wanted}
to~$(\{\sss_\lambda\},\{\mu_\lambda|_{G_{L^+}}\})$,
we can, after
possibly further extending~$L$, find an odd, regular,
  polarizable, weakly irreducible compatible system~$(\{t_\lambda\},\{\mu_\lambda|_{G_{L^+}}\})$ of
  $G_{L}$-representations with associated $p$-adic
  representation~$t$, such that:
  \begin{itemize}
  \item $\sbar|_{G_L}$, $\tbar$ and
    $\tbar|_{G_{(LE)^{{\gal}}}}$ are pleasant. (The pleasantness
    of $\sbar|_{G_L}$ and $\tbar,$ and also of $\tbar|_{L^{\gal}},$
    is ensured by the very statements of 
Corollaries~\ref{cor: inverse Galois component swapping} and~\ref{cor: could omit x if we wanted}.  
    The pleasantness of $G_{(LE)^{\gal}}$ can then 
    be achieved by taking~$q$
in Corollaries~\ref{cor: inverse Galois component swapping}  and~\ref{cor: could omit x if we wanted} to be sufficiently large, and in particular large enough that~$\PGL_n(\F_{q})$ is
    simple and not isomorphic to any Jordan--H\"older factor
    of~$\Gal(E^{\gal}/\Q).$)
     \item For each
  place~$v_L$ of~$L^{+}$ lying over a place~$v\in S$, we have a polarized isomorphism
  $\tbar|_{G_{L_{w_{L}}}}\cong\sbar|_{G_{L_{w_L}}}$. 
  \item At each place $v_L$ of~$L^{+}$ lying over a place~$v\in S$ which does
  not divide~$p$, $t|_{G_{L_{w_L}}}$ lies on~$C_v|_{{L^+_{v_L}}}$.
\item $t|_{G_{L_{w_L}}}$ is potentially diagonalizable for all
  places~$v_L |p$ of~$L^+$.
\item $t$ is unramified outside of the places lying over~$S$.
  \end{itemize}
Now applying Corollaries~\ref{cor: inverse Galois component swapping} and~\ref{cor: could omit x if we wanted} to~$(\{r_\lambda\},\{\mu_{E,\lambda}\})$,
we may find a finite Galois extension $K/E$ of CM fields and an odd, regular,
  polarizable, weakly irreducible compatible system~$(\{u_\lambda\},\{\nu_\lambda|_{G_{K^+}}\})$ of
  $G_{K}$-representations with associated $p$-adic
  representation~$(u,\nu|_{G_{K^+}})$, such that:
  \begin{itemize}
  \item 
   $\ubar|_{G_{(LK)^{{\gal}}}}$ and~$\tbar|_{G_{(LK)^{{\gal}}}}$ are pleasant. 
        (This can be achieved by noting that we may choose~$K/E$ to be linearly disjoint not
       only from~$(LE)^{\gal}$ but also from the Galois closure of the splitting field of~$\tbar$.)   
             \item For each
  place~$v_K$ of~$K$ lying over a place in~$S_E$, 
  we have a polarized isomorphism
  $\ubar|_{G_{K_{v_K}}}\cong\rbar|_{G_{K_{v_K}}}$.
 \item  $u|_{G_{K_{w_K}}}$ is potentially diagonalizable for all
   places~$v_K| p$ of~$K^+$ lying over places in~$S_E$, except for the places~$v_{K,2}$
lying over~$v_{E,2}$, for which~$u|_{G_{K_{w_{K,2}}}}$ lies on~$C_{v_1}|_{{K^+_{v_{K,2}}}}$. 
  \item $\nu |_{G_{K^{+}_{v_{K,2}}}} \simeq \mu_E  |_{G_{K^{+}_{v_{K,2}}}}$,
  and~$\mu_E |_{G_{E^+_{v_{E,2}}}} \simeq\mu |_{G_{F^{+}_{v_1}}}$. The first condition follows
  from the corollary we are invoking and the second condition was already assumed to be true.
   \item $u$ is unramified outside of the places lying over~$S_E$. 
   \end{itemize}
   For ease of notation, we now write~$M=(LK)^{{\gal}}$; note then that~$M^{\gal} = M$.
  We also draw the following picture which represents the inclusions between various fields
  which occur in this argument:
  
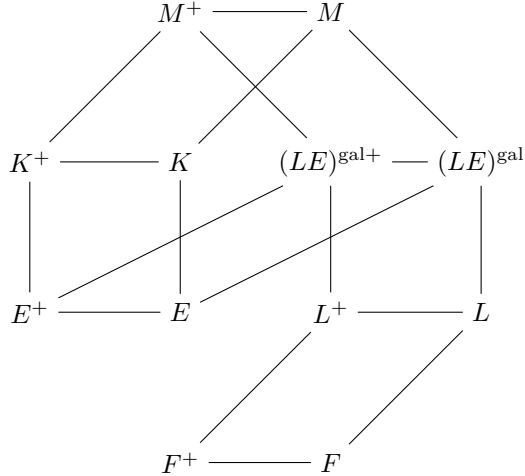
\begin{center}
\begin{tikzpicture}[node distance = 2cm, auto]
\node (Q) {$F$};
\node (A)[left of=Q, node distance =2cm]{$F^{+}$};
\node(B)[above of =A]{$E$};
\node(C)[above of = Q]{$L^{+}$};
\node(D)[above of =Q, right of =Q]{$L$};
\node(E)[above of=B, left of =B]{$K^{+}$};
\node(F)[right of =E]{$K$};
\node(G)[above of =C]{$(LE)^{\gal +}$};
\node(H)[above of = D]{$(LE)^{\gal}$};
\node(I)[above of = E, right of =E]{$M^{+}$};
\node(J)[above of =G]{$M$};
\node(K)[below of=E]{$E^{+}$};
\draw[-] (A) -- (C);
\draw[-] (A) -- (Q);
\draw[-] (Q) -- (D);
\draw[-] (K) -- (E);
\draw[-] (K) -- (B);
\draw[-] (K) -- (G);
\draw[-] (B) -- (H);
\draw[-]  (B) -- (F);
\draw[-] (C) -- (G);
\draw[-] (C) -- (D);
\draw[-] (D) -- (H);
\draw[-] (E) --  (F);
\draw[-] (E) -- (I);
\draw[-] (F) -- (J);
\draw[-] (G) -- (I);
\draw[-] (G)--(H);
\draw[-] (H)--(J);
\draw[-] (I) --(J);
\end{tikzpicture}
\end{center}

Choose a place~$v_{M,2}$  above~$v_{K,2}$ and~$v_{E,2}$ in~$M^{+}$.
We now conjugate $\{u_\lambda\}$ by elements of~$\Gal(M/\Q)$.
Because~$\Gal(M/\Q)$ acts transitively on the primes above~$p$ in~$M^{+}$,
 there exists in particular a~$\sigma \in \Gal(M/\Q)$ with~$\sigma v_{M,2} = v_{M,1}$ for any choice
 of prime~$v_{M,1}$ above our fixed place ~$v_{1}$ of~$F^{+}$.
 Note that, by construction, there will be a polarized isomorphism
 $$ \ubar^{\sigma} | G_{M_{w_{M,1}}} \simeq
 \ubar | G_{M_{w_{M,2}}}\simeq \rbar | G_{M_{w_{M,2}}} \simeq \sbar | G_{M_{w_{M,1}}} \simeq \tbar  | G_{M_{w_{M,1}}},$$
 where the second isomorphism comes from the construction of~$\{u_{\lambda}\}$, and the third isomorphism was one of defining properties of $\{r_{\lambda}\}$.
 Similarly,
 there is an identification of characters
 $$ \nu^{\sigma} | G_{M^{+}_{v_{M,1}}} \simeq
 \nu | G_{M^{+}_{v_{M,2}}} \simeq \mu | G_{M^{+}_{w_{M,1}}}.$$
 We now apply
Lemma~\ref{lem: local swapping II}
to~$(\{t_\lambda|_{G_M}\},\{\mu_\lambda|_{G_{M^+}}\})$ and ~$(\{u^{\sigma}_\lambda|_{G_M}\},\{\nu_\lambda|_{G_{M^+}}\})$
relative to the set~$T = \{v_{M,1}\}$. 
We deduce the existence of an odd, regular,
  polarized, weakly irreducible compatible system~$(\{a_\lambda\},\{\mu_\lambda|_{G_{M^+}}\})$ of $G_M$-representations with associated $p$-adic
Galois representation~$(a,\mu|_{G_{M^+}})$, with the properties that:
\begin{itemize}
\item $\abar=\tbar|_{G_M}$ is pleasant.
\item At the place~$v_{M,1}$ of~$M^{+}$ lying over~$v_1$ in~$F^{+}$,
the representation $a|_{G_{M_{w_{M,1}}}}$ lies
  on~$C_{v_1}|_{{M^+_{v_{M,1}}}}$.
\item  $a$ is potentially diagonalizable at all places dividing~$p$
  other than the places over~$v_{M,1}$. 
  \item At each place $v_M$ of~$M^+$ lying over a place~$v\in S$ which does
  not divide~$p$, $a|_{G_{M_{w_M}}}$ lies on~$C_v|_{{M^+_{v_M}}}$.
\item $a$ is unramified outside of the places lying over~$S$.
\end{itemize}
As above, we now conjugate $\{a_\lambda\}$ by elements of~$\Gal(M/F) = \Gal(M^+/F^{+})$.
Because~$\Gal(M^{+}/F^{+})$ acts transitively on the primes above~$v_1$ in~$M^{+}$,
for any such prime~$\vvv_{M,1} | p$, there exists a~$\sigma \in \Gal(M/F)$ with~$\sigma \vvv_{M,1} = v_{M,1}$,
and thus
we deduce that there exists an odd, regular
  polarized, weakly irreducible compatible system~$(\{a^{\sigma}_\lambda\},\{\mu^{\sigma}_\lambda\})$ of $G_M$-representations with associated $p$-adic
Galois representation~$(a^{\sigma},\mu^{\sigma})$, with the properties that:
\begin{itemize}
\item $\vvv_{M,1}$ is a prime in~$M^{+}$ above~$v_1$ in~$F^{+}$.
\item $\abar^{\sigma}$ is pleasant.
\item $\mu^{\sigma}|_{G_{M^+_{\vvv_{M,1}}}}=  \mu|_{G_{M^+_{\vvv_{M,1}}}}$, and there is
  a polarized isomorphism 
  $$\abar^{\sigma}|_{G_{M_{\www_{M,1}}}}= \ \tbar^{\sigma}|_{G_{M_{\www_{M,1}}}}
  = \ \tbar|_{G_{M_{\sigma \www_{M,1}}}} 
= \ \tbar|_{G_{M_{w_{M,1}}}} 
  \cong \  \tbar|_{G_{M_{\www_{M,1}}}}$$ for~$\www_{M,1} | \vvv_{M,1}$ in~$M$.
  This is
  because the representations~$\tbar|_{G_{M_{\www_{M,1}}}}$ are all  isomorphic to
  the restrictions of~$\sbar |_{G_{F_{w_1}}}$ to~${G_{M_{w_1}}}$, and so do
  not depend on the choice of~$\vvv_{M,1}$ above~$v_1$ in~$M$.
\item  $a^{\sigma}|_{G_{M_{\www_{M,1}}}}$ lies
  on~$C_{v_1} |_{{M^+_{\vvv_{M,1}}}}$.
\item  $a^{\sigma}$ is potentially diagonalizable at all places of~$M$ dividing~$p$
  other than those lying over~$\vvv_{M,1}$. 
  \end{itemize}
  
We now apply Corollary~\ref{cor: merging components} to the compatible
systems~$(\{t_\lambda|_{G_M}\},\{\mu_\lambda|_{G_{M^+}}\})$ and~$(\{a^{\sigma}_\lambda\},\{\mu^{\sigma}_\lambda\})$,
where we let~$\sigma$ range over a set of elements of~$\Gal(M/F) = \Gal(M^{+}/F^{+})$
such that~$\sigma \vvv_{M,1} = v_{M,1}$  where~$\vvv_{M,1}$ ranges exactly over the primes of~$M^{+}$ above~$v_1$ in~$F^{+}$.
We deduce the existence of an odd, regular,
  polarized, weakly irreducible compatible system~$(\{b_\lambda\},\{\mu_\lambda|_{G_{M^+}}\})$ of $G_M$-representations with associated $p$-adic
Galois representation~$(b,\mu|_{G_{M^+}})$, with the properties that:
\begin{itemize}
\item $\bbar=\tbar|_{G_M}$ is reasonable.
\item At \emph{every} place $\vvv_{M,1}|p$ of~$M^+$ which lies over~$v_1$, $b|_{G_{M_{\www_{M,1}}}}$ lies on~$C_{v_1}|_{M^+_{\vvv_{M,1}}}$.
\item $b$ is potentially diagonalizable at all places of~$M$
  dividing~$p$ and not
  lying over~$v_1$.
\item At each place $v_M$ of~$M^+$ lying over a place~$v\in S$ which does
  not divide~$p$, $b|_{G_{M_{w_M}}}$ lies on~$C_v|_{M^+_{v_M}}$. 
  \item $b$ is unramified outside of the places lying over~$S$.
\end{itemize}
We now use this deformation of~$\tbar|_{G_M}$ to descend to a deformation of~$\tbar$ over~$G_L$.
Namely,
applying Lemma~\ref{lem: existence of compatible system from potential
    existence}, we deduce the existence of an odd, regular,
  polarizable, weakly irreducible compatible system~$(\{c_\lambda\},\{\mu_\lambda|_{G_{L^+}}\})$ of
  $G_{L}$-representations with associated $p$-adic
  representation~$(c,\mu|_{G_{L^+}})$, such that:
  \begin{itemize}
  \item $\cbar\cong\tbar$.
\item At each place $v_L$ of~$L^+$ lying over a place~$v\in S$ which does
  not divide~$p$, $c|_{G_{L_{w_L}}}$ lies on~$C_v|_{L^+_{v_L}}$.
\item At each place $v_{L,1}$ of~$L^+$ lying over $v_1$, $c|_{G_{L_{w_{L,1}}}}$ lies on~$C_{v_1}|_{L^+_{v_{L,1}}}$.
\item $c|_{G_{L_{w_L}}}$ is potentially diagonalizable for all
  places~$v_L|p$ not lying over~$v_1$.
\item $c$ is unramified outside of the places lying over~$S$.
  \end{itemize}
Hence we have succeeded in finding deformations of a representation~$\tbar$ (which looks locally
like~$\sbar |_{G_{L}}$ at each~$v|p$ but is globally different) which
lies on the desired component at the places lying over~$v_{1}$
and is potentially diagonalizable at all other primes.

We now use local swapping to find a corresponding
deformation of~$\sbar |_{G_{L}}$ which has the correct behavior at the
places lying over~$v_{1}$
and  is potentially diagonalizable at all other primes.
That is,
by applying Lemma~\ref{lem: local swapping II}
to~$(\{\sss_\lambda\},\{\mu_\lambda|_{G_{L^+}}\})$ and ~$(\{c_\lambda\},\{\mu_\lambda|_{G_{L^+}}\})$ (taking the subset $T$ of that lemma to
be the set of all places lying over $p$),
 we deduce the existence of an odd, regular,
  polarized, weakly irreducible compatible system~$(\{d_\lambda\},\{\mu_\lambda|_{G_{L^+}}\})$ of
  $G_{L}$-representations with associated $p$-adic
  Galois  representation~$(d,\mu|_{G_{L^+}})$, such that:
  \begin{itemize}
  \item $\dbar\cong\sbar|_{G_L}$.
 \item At each place $v_L$ of~$L^+$ lying over a place~$v\in S$ which does
  not divide~$p$, $d|_{G_{L_{w_L}}}$ lies on~$C_v|_{L^+_{v_L}}$.
\item At each place $v_{L,1}$ of~$L^+$ lying over $v_1$, $d|_{G_{L_{w_{L,1}}}}$ lies on~$C_{v_1}|_{L^+_{v_{L,1}}}$.
\item $d|_{G_{L_{{w_L}}}}$ is potentially diagonalizable for all
  places~$v_L |p$ not lying over~$v_1$.
\item $d$ is unramified outside of the places lying over~$S$.
  \end{itemize}
We now descend from~$L$ to~$F$: Applying Lemma~\ref{lem: existence of compatible system from potential
    existence}, we construct an odd, regular,
  polarized, weakly irreducible compatible system~$(\{e_\lambda\},\{\mu_\lambda\})$ of
  $G_{F}$-representations with associated $p$-adic
  representation~$(e,\mu)$, such that:
  \begin{itemize}
  \item $\ebar\cong\sbar$.
\item At each place $v\in S$ which does
  not divide~$p$, $e|_{G_{F_{w}}}$ lies on~$C_v$.
\item  $e|_{G_{F_{w_1}}}$ lies on~$C_{v_1}$.
\item $e|_{G_{F_{w}}}$ is potentially diagonalizable for all
  places~$v|p$ of~$F^+$ other than~$v_1$.
\item $e$ is unramified outside of the places lying over~$S$.
  \end{itemize}
Since~$v_1$ was arbitrary, there exists a compatible system with these
properties for each choice of~$v_1$. Applying Corollary~\ref{cor: merging components} to these compatible systems, we
obtain the required compatible system~$(\{s_\lambda\},\{\mu_\lambda\})$.
\end{proof}

As an immediate consequence, we have the following potential
automorphy theorem. 

\begin{cor}\label{cor: potential automorphy for globally
    realizable}Assume that either $n$ is odd or $n=2$. Let~$F$ be a CM
  field, and let~$(s,\mu)$ be a polarized representation, where
$$s: G_{F} \rightarrow \GL_n(\Qpbar)$$
is odd and ramified at only finitely many primes. 
Suppose that~$\sbar$ is reasonable. Let~$\rho$ be the
corresponding prolongation of~$s$,  and assume that~$\rho |_{G_{F^+_v}}$ is globally realizable for
each~$v|p$. Then~$(s,\mu)$ is potentially automorphic. 
\end{cor}
\begin{proof}
Let~$\Favoid/F$ be a finite Galois extension.   By 
  Theorems~\ref{thm: main theorem on existence of lifts} and~\ref{thm:
    BLGGT splitting at primes}, together with Lemma~\ref{lem:can base change reasonable to pleasant}, there is a finite Galois extension of
  CM fields~$L/F$, linearly disjoint from~$\Favoid/F$ and
  a polarizable regular algebraic cuspidal automorphic representation
  of $\GL_n(\A_L)$ such that $\rbar_p(\pi)\cong\sbar|_{G_L}$ is
  pleasant,
  $s|_{G_L}$ is only ramified at split primes, 
  and for each place~$v$ of $L^+$, 
  $r_p(\pi)|_{G_{L_v}}$  lies on the component determined
  by~$s|_{G_{L_v}}$. Then~$s|_{G_L}$ is automorphic by
  Theorem~\ref{thm: FC's pst R=T}.
  \end{proof}

Finally, as promised in Remark~\ref{rem: potentially globally realizable}, we show that ``potentially globally
realizable'' representations are globally realizable.

\begin{cor}
  \label{cor: potentially globally realizable}Assume that either~$n$
  is odd or~$n=2$. A component~$C$ for
$\rhobar:G_K\to\cG_n(\F)$ is globally realizable if and only if there
exists a finite extension~$L/K$ such that~$C|_L$ is globally realizable.
\end{cor}
\begin{proof}If~$C$ is globally realizable, then it is easy to see
  from the definitions
  that~$C|_L$ is globally realizable, by choosing an appropriate
  extension of CM fields~$E/F$. Conversely, suppose that~$C|_L$ is
  globally realizable. Let~$\mu_{\rhobar}$ be the multiplier character
  for~$C$, so that~$\mu_{\rhobar}|_{G_{L}}$ is the multiplier
  character for~$C|_L$. By the definition of global realizability,
  $\mu_{\rhobar}|_{G_{L}}$ is the restriction of a de Rham character
  of a totally real field, so it is a power of the cyclotomic
  character times a finite order character; so the same is true
  of~$\mu_{\rhobar}$. 

 Exactly as in the proof of~\cite[Prop.\
  A2]{MR3134019}, by~\cite[Prop.\ 3.2]{MR2869026}
(see also~\cite[Thm.~1.2]{MBdiditfirst})
  we can find a CM
  field~$F$ with maximal totally real subfield~$F^+$, such that we
  have~$F_v^+\cong K$ for all~$v|p$, and a
  representation~$\varrhobar:G_{F^+}\to\cG_n(\F)$ with
  multiplier~$\mubar$,  such that for
  each~$v|p$, $\varrhobar|_{G_{F^+_v}}\cong\rhobar$,
  and~$\varrhobar|_{G_F}$ is reasonable. In particular, we
  have~$\mubar|_{G_{F^+_v}}=\mubar_{\rhobar}$ for
  all~$v|p$.

We claim that we can find a (necessarily de Rham) lift~$\mu$
of~$\mubar$ which satisfies~$\mu|_{G_{F^+_v}}=\mu_{\rhobar}$ for
  all~$v|p$. Indeed, as explained above, we can
  write~$\mu_{\rhobar}=\varepsilon^r\chi_{\rhobar}$ for some integer~$r$
  and some finite order character~$\chi_{\rhobar}$.
  By~\cite[Lem.\ 4.1.1]{CHT}, there exists a finite order global character~$\chi$
  such that~$\chi|_{G_{F^+_v}}=\chi_{\rhobar}$ for each~$v|p$.
  By construction, it follows that~$\varepsilon^r \chi |_{G_{F^{+}_v}} = \mu_{\rhobar}$
  for all~$v|p$,  and thus
  $$\overline{\varepsilon^r \chi} = \overline{\mu} \cdot \psi$$
  for some finite order residual character~$\psi$ which is trivial for all~$v|p$.
 But then~$\mu = \varepsilon^r \chi  \widetilde{\psi}^{-1}$ has the required property,
 where~$\widetilde{\psi}$ is the Teichm\"uller lift of~$\psi$.
  
  Let~$E/F$ be a finite Galois extension of CM fields, linearly
  disjoint from~$\overline{F}^{\ker\varrhobar}(\zeta_p)$ such that for
  each place~$v_{E}|p$ of~$E^+$ we have $E^+_{v_E}\cong L$. The
  result follows from Theorem~\ref{thm: main theorem on existence of
    lifts} (applied with~$F$ equal to our~$E$, and~$\sbar$
  our~$\varrhobar|_{G_E}$), together with Lemma~\ref{lem: existence of
    compatible system from potential existence} (applied with~$\abar$
  there being our~$\varrhobar|_{G_F}$, and~$L$ our
  $E$). 
  \end{proof}


\bibliographystyle{amsalpha}
\bibliography{transfer}

 \end{document}